\documentclass{imsart}

\RequirePackage[OT1]{fontenc}
\usepackage{a4wide}
\usepackage{xr-hyper}
\usepackage{hyperref}
\usepackage{amsmath,amsthm,amssymb}
\usepackage[isolatin]{inputenc}
\usepackage{amsfonts,dsfont}
\usepackage{pstricks}
\usepackage{comment}
\usepackage{natbib}
\usepackage{colordvi}
\usepackage{graphicx}
 \usepackage{pgf, tikz}
  \usetikzlibrary{fit,patterns,decorations.pathreplacing}

\newtheorem{theorem}{Theorem}[section]
\newtheorem{proposition}[theorem]{Proposition}
\newtheorem{corollary}[theorem]{Corollary}
\newtheorem{lemma}[theorem]{Lemma}

\newtheorem{remark}[theorem]{Remark}

\newtheorem{assumption}[theorem]{Assumption}
\usepackage[ruled]{algorithm2e}

\SetKwInput{KwData}{Input}
\SetKwInput{KwResult}{Output}

\newcommand{\R}{\mathbb{R}} 
\newcommand{\Q}{\mathbb{Q}} 

\renewcommand{\P}{\mathbb{P}}

\newcommand{\cA}{\mathcal{A}}

\newcommand{\bell}{\boldsymbol{\ell}}
\newcommand{\q}{\boldsymbol{q}}

\newcommand{\pimin}{\pi_{{\tiny min}}}
\newcommand{\wmax}{w_{{\tiny max}}}
\newcommand{\wmin}{w_{{\tiny min}}}
\newcommand{\E}{\mathbb{E}} 

\renewcommand{\l}{\ell}
\newcommand{\FDR}{\mbox{FDR}}

\newcommand{\maFDR}{\mbox{{M}FDR}}
\newcommand{\TDR}{\mbox{TDR}}

\newcommand{\mtc}{\mathcal}

\newcommand{\wh}[1]{{\widehat{#1}}}

\newcommand{\ind}[1]{{\mathds{1}\{#1\}}}
\newcommand{\e}[1]{\mbe\brac{#1}}

\newcommand{\eps}{\varepsilon}

\usepackage{graphics}
\newcommand{\FDP}{\mbox{FDP}}

\newcommand{\pizero}{\boldsymbol{\pi}_0}
\newcommand{\pizerohat}{\boldsymbol{\widehat{\pi}}_0}
\newcommand{\pione}{\boldsymbol{\pi}_1}
\renewcommand{\Q}{\boldsymbol{Q}}
\newcommand{\varphiVEM}{\varphi^{\mbox{\tiny VEM}}}
\newcommand{\varphisimple}{\varphi^{\mbox{\tiny $Z$}}}

\renewcommand{\e}{\mathrm{e}}

\graphicspath{{fig_simu/}}

\begin{document}

\begin{frontmatter}
\title{Graph inference with clustering and false discovery rate control}
\runtitle{Graph inference with FDR control}

\begin{aug}
\author{\fnms{Tabea} \snm{Rebafka}\ead[label=e1]{tabea.rebafka@upmc.fr}},
\author{\fnms{Etienne} \snm{Roquain}\ead[label=e2]{etienne.roquain@upmc.fr}}
\and
\author{\fnms{Fanny} \snm{Villers}\ead[label=e3]{fanny.villers@upmc.fr}}
\address{
Sorbonne Universit\'e, Universit\'e de Paris, CNRS, \\
 Laboratoire de Probabilit\'es, Statistique    et    Mod\'elisation, Paris, France.\\ 
  \printead{e1}
\printead{e2}
\printead{e3}
}
\end{aug}

\begin{abstract}
In this paper, a noisy version of the stochastic block model (NSBM) is introduced and we investigate the three following statistical inferences in this model: estimation of the model parameters, clustering of the nodes  and identification of the underlying graph. While the two first inferences are done by using a variational expectation-maximization (VEM) algorithm, the graph inference is done  by controlling the false discovery rate (FDR), that is, the average proportion of errors among the edges declared  significant, and by maximizing the true discovery rate (TDR), that is, the average proportion of edges declared  significant among the true edges. 
Provided that the VEM algorithm provides reliable parameter estimates and clustering, we theoretically show that our procedure does control the FDR  while satisfying an optimal TDR property, up to remainder terms that become small when the size of the graph grows. 
Numerical experiments show that our method outperforms the classical FDR controlling methods that ignore the underlying SBM topology.
In addition, these simulations demonstrate that the FDR/TDR properties of our method are robust to model mis-specification, that is,  are essentially maintained outside our model.
\end{abstract}



\end{frontmatter}

\tableofcontents

\section{Introduction}\label{sec:intro}

\subsection{Context}


Network analysis is concerned with modeling and describing the interactions of a given population of individuals.
Networks arise in a large variety of domains as social, biological or information sciences, just to name a few.
In such applications, an essential  task is to infer a reliable version of the network. For instance, when a dense network is observed, this one is often considered to be a perturbation of some underlying less dense graph, and we should remove the edges that are only due to "noise". Once the graph is inferred, a deeper analysis can be done to describe the  communication structures of the network, for instance by detecting the communities, that is,  clustering of the nodes in groups with similar connection behavior. 
The literature provides many different clustering algorithms as $k$-means, hierarchical clustering,  
random walk algorithms,  spectral clustering, modularity maximization and likelihood methods, without being exhaustive. 

To be more concrete, consider the widespread example of clustering a set of common data points with pairwise distances or similarities.
Typically, a similarity graph is first constructed, then a  clustering algorithm, as for instance spectral clustering based on the graph Laplacian, is applied providing a partition of  the data points into groups with high intra-group similarity and low inter-group similarity.
There are several ways to construct  similarity graphs: one can just use the fully connected graph, where the adjacency matrix is defined by all pairwise similarities.  However, it is more common to use a sparser version of the graph, e.g. a $\varepsilon$-neighborhood graph obtained by  thresholding similarities, or a $k$-nearest neighbor graph, where for a given node only the edges to the $k$ nodes with highest similarities are conserved. It is well known that the specific choice of the similarity graph for the  clustering procedure has a strong influence on the clustering result, and the appropriate choice of the connectivity parameters (the neighborhood threshold $\varepsilon$ or the number of neighbors $k$) is still a headache \citep{Luxburg07}.

In general,  such  two-stage procedures, where   graph inference and clustering are treated separately, may not be optimal as both tasks are very interrelated. That is, the inferred graph has a considerable impact on the obtained clustering, and conversely, using the cluster memberships may improve graph inference. 
Unifying these inferences is one important motivation for our work and it relies on considering an appropriate  probabilistic network model.

\subsection{Stochastic block model}

A popular random graph model for clustering is the stochastic bloc model (SBM)  \citep{Holland1983}, that models network heterogeneity by varying connecting behavior of different groups of nodes.  
More precisely, each node is supposed to belong to exactly one group and the edge probability of a pair of nodes depends entirely on the group membership of these two nodes. Thus, clustering becomes the problem of estimating the group memberships in the stochastic block model. 
Furthermore, in some sense, using the stochastic block model corresponds to summarizing   a complex network by a meta-network by grouping vertices to  a few meta-vertices with a few meta-edges without losing too much information. 
It is noteworthy that SBM may detect more complex connecting schemes than simple communities (as bipartite graph) and so, is able to  describe a wide spectrum of graph topologies.
We refer the reader to \cite{NowickiSnijders2001} and \cite{Picard2009} for applications of this idea to  social and biological networks. 
Many variants of the SBM have been developed in the literature (e.g., weighted \cite{MatiasRobin2014}, valued \cite{MRV2010}, overlapping \cite{Lat2014} or dynamic \cite{MRV2018}, among others). 

As in most latent variable models,  parameter estimation is a difficult task in  the SBM. Due to  the complex dependency structure in the graph, the classical 
EM algorithm \citep{DempsterLR}  does not apply, but a variational EM algorithm has been proposed 
to approach the maximum likelihood estimator and estimate group memberships \citep{Daudin2008}. 
While the
EM algorithm  is known to  converge to the maximum likelihood estimator under appropriate assumptions \citep{wu1983},  this property is  in general lost when 
adding variational approximations.  However,  in the case of  the SBM,  variational estimators can be shown to be consistent and asymptotically equivalent to the maximum likelihood estimators  \citep{Celisse_etal}.

We consider in Section~\ref{sec:model} a  variant of  SBM which is suitable for simultaneously inferring the clustering and the graph:
the  so-called noisy stochastic block model (NSBM). In this model, we do not observe the graph, which is itself a latent structure, but only a noisy version of it, with the following blurring mechanism:  in place of missing edges,  pure random noise is observed, and in place of present edges, we observe an effect, whose intensity depends on the group memberships of the nodes in the latent graph. 
We develop in Section~\ref{sec:VEMintro} a VEM algorithm that aims at estimating the model parameters. By using a simple maximum a posteriori criterion, this also provides an estimation of the (latent) clustering. This method has an interest on its own in applications where only a clustering of the nodes is desired. The advantage of our method with respect to most standard methods
 does not require the selection of some connectivity parameters. 
Here, in addition, we take advantage of this clustering to improve the accuracy of the graph inference.

\subsection{False discovery rate}

Let us first mention that graph inference is a task with a long history, especially in the case where when one tries to estimate the marginal correlation or partial correlation between node observations.  In that case, a Gaussian graphical model \cite{Lau1996} is often used, and one estimates either the correlation matrix (marginal correlations) or the precision matrix (partial correlation). In the literature, this  task is classically done by "graphical lasso" type approaches \cite{Mein2006}, \cite{Fri2007}, \cite{Ban2008}, \cite{Rav2011}.

However, when inferring a graph, adding a non-existing edge between two nodes is in many applications more problematic than missing an existing edge, especially for sparse graphs. To this respect,  the practitioner thus wants to avoid false positives, that is, edges that are wrongly declared significant. We thus adopt a multiple hypothesis testing formulation of the graph reconstruction problem. Markedly, the number of null hypotheses to test can be particularly high:  $m=n(n+1)/2$ where $n$ is the number of nodes.

In large scale multiple testing, a popular method is the Benjamini Hochberg procedure (BH), introduced in \cite{BH1995} and widely popularized afterwards, which controls the false discovery rate (FDR), defined as the averaged proportion of errors among the items declared as significant. 
Among the abundant literature in that area, a successful modeling is to assume that the observations follow a mixture model  \cite{ETST2001}, which allows to exploit the dimension of the data to fit crucial quantities as the null distribution \cite{Efron2004}, and the alternative distribution \cite{SC2007}, which allow a better multiple testing inference, both in terms of FDR and power (items correctly declared as significant).

More sophisticated model mixture have then been be considered, that incorporate some underlying (latent or known) structure of the null hypotheses. While  FDR control under dependence is known to be a challenging issue, as the most classical results use independence or positive dependence of the test statistics \cite{BH1995,BY2001}, these models circumvent this difficulty by assuming that the test statistics are independent conditionally on the structure. 
Former studies include group structure \cite{SC2009} and Markov structures \cite{CS2009,LZP2016}. These methods have the strong advantage to both control the FDR under dependencies while allowing more detections than procedures ignoring the structure, as the BH procedure do. 
In this paper, we follow this general line of research by controlling the FDR in the noisy SBM. 

\subsection{Presentation of the paper}

The contributions of this paper are as follows:
\begin{itemize}
\item We develop a VEM type approach to estimate the NSBM parameters, which also leads to a clustering. It  also estimates the posterior probabilities that there is no edge between each node couple $(i,j)$, that we will called the $\l$-values;
\item We combine suitably these $\l$-values by adapting procedures of the multiple testing literature in mixture models, and notably through a $q$-value-based approach \cite{Storey2003,CR2018};
\item Combining these two approaches leads to a new procedure for inferring both the graph and the clustering, with a clear interpretation in terms of false positives: among the edges discovered by the procedure, there are, on average, at most $5\%$ (say) of errors;
\item The theoretical validity in terms of false/true positives is established via careful model assumptions and concentration inequalities, which leads to non-asymptotical results. This goes one-step further existing results in the multiple testing literature concerning mixture models;
\item Numerical experiments support the validity of our approach both in the NSBM and outside the NSBM, which shows the robustness of our method.
\end{itemize}

The paper is organized as follows: Section~\ref{sec:setting} introduces the main mathematical tools that will be used throughout the paper, including the NSBM, multiple testing procedures and the graph inference criteria. The VEM approach to fit the model parameters and the clustering is then developed in Section~\ref{sec:VEMintro}. Our testing procedure is defined in Section~\ref{sec:method} and its theoretical properties are provided in Section~\ref{sec:theory}. 
Numerical experiments are given in Section~\ref{sec:simu} and a discussion is given in Section~\ref{sec:discussion}.
Detailed proofs are deferred to Section~\ref{sec:proofs}. Finally, Section~\ref{sec:supp} is a supplement containing auxiliary results (e.g., calculations in the Gaussian case, additional lemmas and proofs).

\section{Setting}
\label{sec:setting}

\subsection{Stochastic block model}\label{sec:modelSBM}

Let us first recall the definition of a standard (binary) stochastic block model (SBM). 
Let $n\geq 2$ be the number of nodes in the graph and $Q\in\{1,\dots,n\}$. Denote  $\cA=\{(i,j)\::\: 1\leq i <j\leq n\}$ the set of all possible (undirected) edges and $m=n(n-1)/2$ 
 its cardinal. The SBM corresponds to the observation of an adjacency matrix $A=(A_{i,j})_{1\leq i,j\leq n}\in\{0,1\}^{n^2}$ (there is an edge between node $i$ and node $j$ if and only if $A_{i,j}=1$), generated by the following random layers:
\begin{itemize}
\item The vector $Z=(Z_1,\dots,Z_n)$ of group memberships of the nodes is such that $Z_i$, $1\leq i\leq n$, are i.i.d. with values in $\{1,2,\dots,Q\}$ with probability 
$$\pi_q=\P(Z_1=q), \:\:q\in\{1,\dots,Q\},$$
for some parameter $\pi=(\pi_q)_{q\in\{1,\dots,Q\}}\in[0,1]^Q$ such that $\sum_{q=1}^Q \pi_q=1$. 
\item Conditionally on $Z$, the variables $A_{i,j}$, $(i,j)\in\cA$, are independent Bernoulli variables with parameter $w_{Z_i,Z_j}$,  that is, 
$$(A_{i,j})_{(i,j)\in \cA}\:|\:Z \sim \bigotimes_{(i,j)\in \cA} \mathcal{B}(w_{Z_i,Z_j}),$$
for some parameter $w=(w_{q,\ell})_{q,\l\in\{1,\dots,Q\}}\in[0,1]^{Q^2} $. 
Since we focus on the undirected model here, we generate only $A_{i,j}$, $(i,j)\in\cA$ and we set $A_{j,i}=A_{i,j}$ for all $(i,j)\in\cA$ and $A_{i,i}=0$ for $i\in\{1,\dots,n\}$. We also  impose that $w$ is symmetric, that is, $w_{q,\l}=w_{\l,q}$ for all $q,\l\in\{1,\dots,Q\}$.
\end{itemize}

\subsection{Noisy stochastic block model}\label{sec:model}

We now define the model that will be used throughout the manuscript. We refer to it as the {\it  noisy stochastic block model} (NSBM in short), as  we do not directly  observe the adjacency matrix $A$ but only  a noisy version of it. The observation $X \in \R^{\mtc{A}}$ is thus the result of an additional random layer:
\begin{itemize}
\item The variables $(Z,A)$ are latent and generated according to an SBM with parameter $n$, $Q$, $\pi$ and $w$, as defined in Section~\ref{sec:modelSBM};
\item Conditionally on $(Z,A)$,  the observed variables $X_{i,j}$, $(i,j)\in\cA$ are independent and each $X_{i,j}$ has the following  distribution
$$
X_{i,j} \sim (1-A_{i,j}) g_{0,\nu_0} + A_{i,j}  g_{\nu_{Z_i,Z_j}},
$$
for some unknown parameters $\nu_0\in \mathcal{T}_0$ and $\nu_{q,\l}\in \mathcal{T}$, $1\leq q,\l\leq Q$,  where $\{g_{0,t},t\in \mathcal{T}_0\}$ (resp. $\{g_{t},t \in  \mathcal{T}\}$) is a given parametric density family, where $\mathcal{T}_0$ (resp. $ \mathcal{T}$) is a non-empty open subset of $\R^{d_0}$ (resp.  $\R^{d_1}$). 
These densities are meant to be taken with respect to the Lebesgue measure on $\R$. 
\end{itemize}

The rationale behind this model is that, in place of missing edges ($A_{i,j}=0$), we observe  pure random noise modeled by the density $g_{0,\nu_0}$ (also called null density), and in place of present edges ($A_{i,j}=1$), we observe an effect, whose intensity depends on the group memberships of the nodes in the underlying SBM, which is modeled by the density $g_{\nu_{q,\l}}$ (also called alternative density for parameter $\nu_{q,\l}$).

The unknown global model  parameter is $\theta=(\pi,w,\nu_0,\nu)$, where $\pi$ and $w$ come from the SBM, $\nu_0$ denotes the null parameter and 
 $\nu=(\nu_{q,\ell})_{1\leq q , \ell\leq Q}\in  \mathcal{T}^{Q^2}$ denotes the parameter vector of the effects.
As our focus is on   undirected graphs,  $\nu$ is symmetric, that is,  $\nu_{\l,q}=\nu_{q,\l}$ for all $1\leq q , \ell\leq Q$.
From the symmetry property of $w,\nu$, we will sometimes consider, with some abuse of notation, that the parameter $(w,\nu)$ belongs to $\R^{Q(Q+1)/2}$, rather than $\R^{Q^2}$ when appropriate.
Overall, the   parameter $\theta$ is of dimension $(Q-1)+Q(Q+1)/2+d_0+d_1 Q(Q+1)/2$.
The distribution of $(X,A,Z)$ in the NSBM with parameters $n$, $Q$, $\theta$ is denoted by $P_{Q,\theta}$ (or $P_\theta$ for short).
We denote by $\Theta_Q$  (or $\Theta$ for short) the parameter space, which can be used to define additional restrictions on $(\pi,w,\nu_0,\nu)$. 
For instance, we will always assume in the sequel that $\pi_q\in (0,1)$, $w_{q,\l}\in (0,1)$, for $1\leq q\leq \l\leq Q$.
 The NSBM is defined by the distribution family $\{P_{n,Q,\theta}, \theta\in\Theta\}$. For $\theta\in \Theta$, we denote by $\P_\theta$ the distribution of the underlying probability space such that $(X,A,Z)\sim P_\theta$.

In this paper, the Gaussian case will be our leading example. 
 It is particularly suitable for modeling situations where the observations $X_{i,j}$ correspond to correlations, which are known to be approximately Gaussian in different asymptotic settings, see \cite{DP2007,Liu2013}. 
The {\it Gaussian NSBM} corresponds to the NSBM with the following choice of the parametric density families: 
\begin{align}\label{gaussmodel}
\{g_{0,t},t\in  \mathcal{T}_0\}=\{\mathcal N(0,\sigma_0^2),\sigma_0>0\}, \:\:\:\{g_{t},t \in  \mathcal{T}\}=\{\mathcal N(\mu,\sigma^2),\mu\in\R,\sigma>0\},
\end{align}
With the notation above, we have in this case $d_0=1$ and $d_1=2$ and $\theta=(\pi,w,\sigma_0,\mu,\sigma)$, where $\mu=(\mu_{q,\l})_{1\leq q \leq \l\leq Q} \in \R^{Q(Q+1)/2}$ and $\sigma=(\sigma_{q,\l})_{1\leq q \leq \l\leq Q} \in (0,\infty)^{Q(Q+1)/2}$. 
An illustration for the Gaussian NSBM is given in
Figure~\ref{fig_noisysbm_model}.

\begin{figure}[h!]
\begin{center}
\begin{tikzpicture}[scale=0.85]
 \tikzstyle{quadri}=[circle,draw,text=black,thick]
 \tikzstyle{estun}=[-,>=latex,very thick]
 \node[quadri, fill=gray!25] (1) at (0,3) {$1$};
 \node[quadri, fill=gray!25] (2) at (-4,0) {$2$};
  \node[quadri, fill=gray!25] (3) at (0,-3) {$3$};
 \node[quadri] (4) at (4,0) {$4$};
   \draw[estun] (1)--(2) node[sloped,above,pos=0.5]{$X_{1,2}\sim\mathcal{N}(\mu_{11},\sigma_{11}^2)$};
 \draw[estun,dashed,thick] (1)--(3) node[sloped,below,pos=0.59]{$X_{1,3}\sim\mathcal{N}(0,\sigma_{0}^2)$};
 \draw[estun] (1)--(4) node[sloped,above,pos=0.5]{$X_{1,4}\sim\mathcal{N}(\mu_{12},\sigma_{12}^2)$};
 \draw[estun] (2)--(3) node[sloped,above,pos=0.55]{$
 X_{2,3}\sim\mathcal{N}(\mu_{11},\sigma_{11}^2)$};
 \draw[estun,dashed,thick] (2)--(4) node[sloped,above,pos=0.75]{$
 X_{2,4}\sim\mathcal{N}(0,\sigma_{0}^2)$};
 \draw[estun,dashed,thick] (3)--(4) node[sloped,above,pos=0.5]{$X_{3,4}\sim\mathcal{N}(0,\sigma_{0}^2)$};
\end{tikzpicture}
\end{center}
\caption{
 Gaussian NSBM: illustration of the distribution of $X$ conditionally on $A,Z$. $n=4$ nodes (circles). 
$Z_1=Z_2=Z_3=1$ (gray color), $Z_4=2$, $A_{1,2}=A_{1,4}=A_{2,3}=1$ (solide edges) and $A_{1,3}=A_{2,4}=A_{3,4}=0$ (dashed edges).  
\label{fig_noisysbm_model}}
\end{figure}
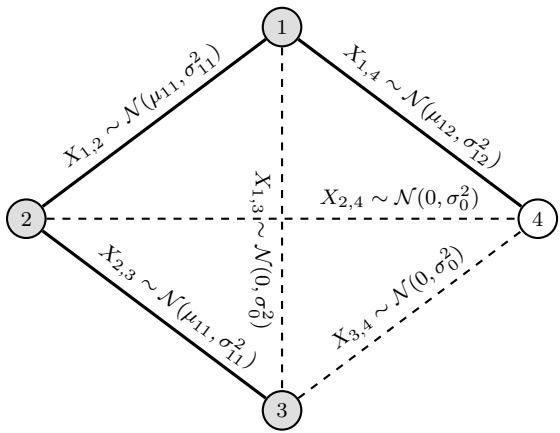

In this paper the NSBM is an undirected graph model. However, the extension to  the directed case is straightforward by relaxing the symmetry constraints on $A$, $w$ and $\nu$. 
In this case the dimension of the parameter $\theta$ is increased and given by $(Q-1)+Q^2+d_0+d_1 Q^2$.

The NSBM $\mathcal{M}_Q=\{P_{Q,\theta}, \theta\in\Theta_Q\}$ involves  the number $Q$ of groups, which is unknown in practice and has to be estimated from the data. 
It is common to  consider a  family of  models  $\{\mathcal{M}_Q,1\leq Q\leq Q_{\max}\}$ for some pre-specified
$Q_{\max}\leq n$, and to  choose the   best number of groups $Q$   by some model selection device, see Section~\ref{subsec:vem_icl} for more details.

Finally, as in all SBM-type models, identifiability in the NSBM is a delicate issue. Here, the NSBM shares similarities with the parametric random graph mixture model with weighted edges introduced in \cite{AMR2011}. Compared to their model, the NSBM replaces the mass point in $0$ by the distribution $g_{0,\nu_0}$. Following the proof of Theorem~12 therein, we can prove identifiability  for the Gaussian NSBM with parameter set \eqref{sec:thetagauss} (defined below) when $n\geq 3$ and $Q\geq 2$.

\subsection{Criteria}

For graph inference, the goal  is to recover the adjacency matrix  $A$ from the observation $X$.
In the multiple testing paradigm, the aim is to make a simultaneous test of
$$
H_{0,i,j} : ``A_{i,j}=0 " \mbox{ against } H_{1,i,j} : ``A_{i,j}=1 ",
$$
which corresponds to test $H_{0,i,j}$ : ``there is no edge between $i$ and $j$  in the latent graph"  against  $H_{1,i,j}$ : ``there is an edge between $i$ and $j$  in the latent graph".
A multiple testing procedure is any measurable function $\varphi(X)\in\{0,1\}^\cA$, with the convention that $\varphi_{i,j}(X)=1$ if and only if $H_{0,(i,j)} $ is rejected for any $(i,j)\in\mathcal{A}$. 

Let us denote the expected proportion of non-connected/connected vertices, in the NSBM with parameter $\theta=(\pi,w,\nu_0,\nu)$, by
\begin{align}\label{equpi0pi1}
\pizero=\sum_{q,\l} \pi_q \pi_\l (1-w_{q,\l}) ; \:\:\:\pione= \sum_{q,\l} \pi_q \pi_\l  w_{q,\l}
\end{align}
(dependence in $\theta$ is removed to lighten the notation).

The false discovery rate (FDR) of a given multiple testing procedure $\varphi(X)$ is 
the average proportion of errors among the discovered edges. It is
defined as
\begin{equation}\label{equ:FDR}
\FDR(\theta,\varphi) = \E_\theta\left[\frac{\sum_{(i,j)\in\mathcal{A}} (1-A_{i,j}) \varphi_{i,j}(X)}{\left(\sum_{(i,j)\in\mathcal{A}}  \varphi_{i,j}(X)\right)\vee 1}\right],
\end{equation}
where $\E_\theta$ refers to the expectation in the NSBM of Section~\ref{sec:model}.  
Sometimes, for simplicity, the following substitute is used
 \begin{equation}\label{equ:maFDR}
\maFDR(\theta,\varphi) = \frac{\E_\theta\left[\sum_{(i,j)\in\mathcal{A}} (1-A_{i,j}) \varphi_{i,j}(X)\right]}{\E_\theta\left[\sum_{(i,j)\in\mathcal{A}}  \varphi_{i,j}(X)\right]},
\end{equation}
where the expectation is taken inside the ratio (with the convention $0/0=0$, that is used throughout the paper). This is called the marginal false discovery rate (MFDR) and is used mostly to mimic the asymptotic behavior of the FDR.

The corresponding power of $\varphi(X)$ is classically defined as the true discovery rate (TDR) by
\begin{equation}\label{equ:TDR}
\TDR(\theta,\varphi) =\frac{ \E_\theta\left[\sum_{(i,j)\in\mathcal{A}} A_{i,j} \varphi_{i,j}(X)\right]}{\E_\theta\left[\sum_{(i,j)\in\mathcal{A}}  A_{i,j}\right]} =  (m\pione)^{-1}\E_\theta\left[\sum_{(i,j)\in\mathcal{A}} A_{i,j} \varphi_{i,j}(X)\right].
\end{equation}
Hence, $\TDR(\theta,\varphi)$
corresponds to the average proportion of discovered edges in the true underlying graph.

A good testing procedure detects a maximum number of significant edges, without making too many false detections. In this sense,  for a given level  $\alpha\in(0,1)$, we aim at finding a testing procedure  $\varphi = \varphi_\alpha$ such that for all $\theta$,
\begin{equation}\label{equ:aim}
 \FDR(\theta, \varphi) \leq \alpha, \:\:\:\mbox{ with } \TDR(\theta, \varphi) \mbox{  ``as large as possible"},
\end{equation}
that is, both  the FDR is controlled at level $\alpha$ and  many true edges are discovered.

\subsection{BH procedure}\label{subsec:BH}

A classical procedure to control the FDR is the so-called BH procedure \citep{BH1995}. It consists in first  computing the $p$-values $p_{i,j}(X)$ of the individual tests for $H_{0,i,j}$ against  $H_{1,i,j}$ for all $(i,j)\in\cA$. In the NSBM, this amounts to computing
\begin{equation}\label{equ:pvalues}
p_{i,j}(X)=\bar{F}_0(|X_{i,j}|),\:\:\:(i,j)\in\mathcal{A},
\end{equation}
where 
$\bar{F}_0(t)=\int_\R \mathds{1}\{|x|\geq t\}  g_{0,\nu_0}(x)dx$ is the probability that the test statistic is larger than $t$ under the null.  
Next,  the $p$-values are ordered in   increasing order such that
$
0=p_{(0)}\leq p_{(1)} \leq \dots \leq p_{(m)}.
$
Then, all  null hypotheses $H_{0,i,j}$ with $p_{i,j}(X)\leq \alpha \hat{k}/m$, where $\hat{k}=\max\{k\in\{1,\dots,m\}\::\:p_{(k)}\leq \alpha k/m\}$ are rejected.
That is, the BH procedure is given by $\varphi^{BH}_{i,j}(X)=\ind{p_{i,j}(X)\leq \alpha \hat{k}/m}$ for $(i,j)\in\mathcal{A}$.
 
In the NSBM defined in Section~\ref{sec:model}, since the $X_{i,j}$'s are mutually independent conditionally on $Z,A$, the classical result in \citep{BH1995} entails that the BH procedure controls the FDR conditionally on $Z,A$, and thus also unconditionally, that is, 
\begin{align}\label{BHcontrol}
\mbox{for all $\theta\in\Theta$, $\FDR(\theta, \varphi^{BH})\leq \pizero \alpha\leq \alpha$.}
\end{align}

However, as we will see, the power of $\varphi^{BH}$ can be suboptimal in our model, and learning the latent clustering of the graph can help to improve the decision.
On an intuitive point of view, the reason is that $p_{i,j}(X)$ is solely based on the individual value of $X_{i,j}$, while one can in principle take advantage of the values of the $X_{i',j'}$ for $(i',j')$ sharing the same node membership as $(i,j)$.
For instance, in Figure~\ref{fig_noisysbm_model}, how to recover $A$ from the values of the edges $X$? The parameter $\theta$ and the clustering $Z$ can help in the decision: if $|X_{1,2}|$ is not too large, it will be difficult to detect the edge $(1,2)$ solely on the value of $X_{1,2}$. However, 
if we know that the nodes $\{1,2,3\} $ belong to group~$1$ and that $w_{1,1}$  is small, this information provides additional evidence that will help to detect the edge $(1,2)$ from $X_{1,2}$.

However, this requires the additional effort to estimate the parameters $\theta$ of the NSBM and of the latent clustering $Z$. 

\section{Estimation and clustering by VEM}
\label{sec:VEMintro}

The NSBM is a latent variable model, so that an  EM-type algorithm may be used to approximate the maximum likelihood estimator of the model parameter $\theta=(\pi,w,\nu_0,\nu)\in\Theta$  using the observation $X$. We develop in this section such an approach.
Proofs of the results are given in Section~\ref{sec:proofvem}.

\subsection{ML estimation} The ML estimator is defined as the maximizer of the observed likelihood function $\theta\in\Theta\mapsto\mathcal{L}(X;\theta)$, which is the marginal of the complete likelihood function $\theta\in\Theta\mapsto \mathcal{L}(X,A,Z;\theta)$ given as follows: for all $\theta\in \Theta$,
\begin{align}
\mathcal{L}&(X,A,Z;\theta) 
= \mathcal{L}(X\:|\:A,Z; \nu_0,\nu) \mathcal{L}(A \:|\: Z;w) \mathcal{L}( Z;\pi)\nonumber\\
 &= \prod_{(i,j)\in \cA} (g_{0,\nu_0}(X_{i,j}))^{1-A_{i,j}} (g_{\nu_{Z_i,Z_j}}(X_{i,j}))^{A_{i,j}}
\times \prod_{(i,j)\in \cA}w_{Z_i,Z_j}^{A_{i,j}} (1-w_{Z_i,Z_j})^{1-A_{i,j}}
\times \prod_{i=1}^n \pi_{Z_i}\nonumber\\
 &= \prod_{\substack{(i,j)\in \cA:\\ A_{i,j}=0}} g_{0,\nu_0}(X_{i,j}) \times\prod_{q=1}^Q\prod_{\l=1}^Q\prod_{\substack{(i,j): A_{i,j}=1\\ Z_{i,q}Z_{j,\l}=1}} g_{\nu_{q,\l}}(X_{i,j})
\nonumber\\
&\quad \hspace{3cm}\times \prod_{1\le q\le \l\le Q} w_{q,\l}^{M_{q,\l}} (1-w_{q,\l})^{\bar M_{q,\l}}
\times \prod_{q=1}^Q \pi_{q}^{\sum_{i=1}^n Z_{i,q}},\label{equvraiscompl}
\end{align}
where we let $Z_{i,q}=\mathds{1}\{Z_i=q\}$ and
\begin{align*}
M_{q,\l}&= \#\{ (i,j)\in \cA: A_{i,j}=1, Z_{i,q}Z_{j,\l} + Z_{i,\l}Z_{j,q}>0\};\\
\bar M_{q,\l}&= \#\{ (i,j)\in \cA: A_{i,j}=0, Z_{i,q}Z_{j,\l} + Z_{i,\l}Z_{j,q}>0\}.
\end{align*}
To derive the likelihood function $\theta\in\Theta\mapsto\mathcal{L}(X;\theta)$ from the complete likelihood function, we should integrate over all possible configurations of the latent variables $(A,Z)\in \{0,1\}^\mathcal{A}\times \{1,\dots,Q\}^n$, which is  prohibitive  for any reasonable values of $n$ and $Q$. 
Hence,  the ML estimator cannot be approached by maximizing directly the likelihood function, and we propose to approach it by using an EM-type algorithm.

\subsection{E-step using variational approximation}

Let $\theta\in\Theta$ denote the current value of the model parameter obtained at the previous (M-)step.
The  E-step of the EM-algorithm consists in computing 
$\mathcal{L}(A,Z\:|\:X; \theta)$,
the conditional likelihood
  of the latent variables $A,Z$ given the observation $X$, when $(X,A,Z)\sim P_\theta$.
Here we encounter the difficulty  that this conditional likelihood is intractable due to the involved dependence structure of the $Z_i$'s given the observations $X$. For this reason we use an approximation 
by some factorized likelihood, that is, a mean-field approximation. We denote by $\tilde P_{\theta,\tau}$ the distribution on $\{0,1\}^\mathcal{A}\times \{1,\dots,Q\}^n$ (potentially depending on $X$) such that the corresponding likelihood is of the form
\begin{align}\label{approx:VEM}
\mathcal{L}(A,Z ; \tilde P_{\theta,\tau})
&=\mathcal{L}(A\:|\: Z, X;\theta) \prod_{i=1}^n \tau_{i,Z_i},   
\end{align}
for a parameter $\tau=(\tau_{i,q})_{i,q}$ belonging to the set
$$
\mathcal T=\left\{\tau=(\tau_{i,q})_{1\leq i\leq n,1\leq q\leq Q}\in [0,1]^{nQ}:
\sum_{q=1}^Q\tau_{i,q}=1,\mbox{ for all } i\in\{1,\dots,n\}
\right\}.
$$
Then, the variational E-step consists in  searching the variational parameters $\hat\tau$ that give the best approximation of the conditional distribution 
$
P_{A,Z|X;\theta}
$
of $A,Z$ given $X$ under $P_\theta$ by a factorized distribution $\tilde P_{\theta,\tau}$ in terms of the Kullback-Leibler divergence. 
More precisely, for all $\theta\in\Theta$, 
\begin{align}
\hat \tau &=\hat \tau(\theta)=\arg\min_{\tau\in\mathcal T} K\!L\left(\tilde P_{\theta,\tau} \:\|\: P_{A,Z|X;\theta}\right)\label{eq:vestep:optim}.
 \end{align}

The following proposition states that the optimisation problem in \eqref{eq:vestep:optim} is equivalent to solving a fixed point equation, which in practice is  solved numerically by an iterative algorithm. 

 \begin{proposition} \label{prop_fixed_point}
For all $\theta=(\pi,w,\nu_0,\nu)\in\Theta$, {any} solution $\hat\tau=(\hat\tau_{i,q})_{i,q} \in \mathcal T\cap(0,1)^Q$ of  \eqref{eq:vestep:optim} verifies the following  fixed point equation
{
\begin{align*}
\hat\tau_{i,q} &= C_i \pi_q\exp\left( \sum_{\substack{j=1\\j\neq i}}^n\sum_{\l=1}^Q\hat\tau_{j,\l}\: d_{q,\l}^{i,j}\right), \:\: i\in\{1,\dots,n\},\:\: q\in\{1,\dots,Q\},
\end{align*}
where $C_i>0$, $i\in\{1,\dots,n\}$, are normalization constants such that $\sum_{q=1}^Q \hat\tau_{i,q}=1$ and} where
\begin{align}
  \rho_{q,\l}^{i,j}&=  \rho_{q,\l}^{i,j}(\theta)=
  \frac{w_{q,\l}g_{\nu_{q,\l}}(X_{i,j}) }{w_{q,\l} g_{\nu_{q,\l}}(X_{i,j}) + (1-w_{q,\l})g_{0,\nu_{0}}(X_{i,j})};
\label{eq:def:rho}
\\
d_{q,\l}^{i,j}&=d_{q,\l}^{i,j}(\theta)=\rho_{q,\l}^{i,j}\left[\log g_{\nu_{q,\l}}(X_{i,j}) +\log w_{q,\l}-1\right] + (1-\rho_{q,\l}^{i,j})\left[\log g_{\nu_{0}}(X_{i,j}) +\log (1-w_{q,\l})\right].\label{eq:def:d}
\end{align}
\end{proposition}

\subsection{M-step} \label{sec:Mstep}

Let $\tau\in \mathcal{T}$ be the current value of the variational parameters  and $\theta'\in\Theta$ the current value of the model parameter obtained at the previous M-step. Now, the M-step consists in updating the value of the model parameter  $\theta$ by maximizing $\theta \in\Theta\mapsto \tilde\E_{\theta',\tau}[\log \mathcal{L}(X,A,Z;\theta)\:|\:X]$, where $\tilde\E_{\theta',\tau}$ denotes a distribution on the underlying probabilistic space that generates $\tilde P_{\theta',\tau}$ as distribution of $(A,Z)$ conditionally on $X$.  

\begin{proposition}[M-step]\label{prop_mstep_pi_w}
The optimisation problem 
\begin{align} \label{eq_general_mstep}
\arg\max_{\theta\in\Theta}\tilde\E_{\theta',\tau}[\log \mathcal{L}(X,A,Z;\theta)\:|\:X],
\end{align} 
splits into three independent problems. The solutions for $\pi$ and $w$ are given by
\begin{align} 
\hat\pi_q& = \frac1n\sum_{i=1}^n\tau_{i,q},\quad q\in\{1,\dots,Q\}\label{eq:mstep:pi}\\
\hat w_{q,\l}&= \frac{\sum_{(i,j) \in\cA} \kappa_{q,\l}^{i,j} }{\sum_{(i,j) \in\cA}(\tau_{i,q}\tau_{j,\l} +\tau_{i,\l}\tau_{j,q} )},\quad q\neq \l,\label{eq:mstep:wql}\\
\hat w_{q,q}&= \frac{\sum_{(i,j) \in\cA}\kappa_{q,q}^{i,j}}{\sum_{(i,j) \in\cA}\tau_{i,q}\tau_{j,q}},\quad   q\in\{1,\dots,Q\},\label{eq:mstep:wqq}
\end{align}  
where
  \begin{align} \label{eq:mstep:kappa}
\kappa_{q,\l}^{i,j}=\left\{\begin{array}{ll}
(\tau_{i,q}\tau_{j,\l} +\tau_{i,\l}\tau_{j,q} )\rho_{q,\l}^{i,j}&\quad\text{if }q\neq \l;\\
\tau_{i,q}\tau_{j,q}{
\rho_{q,q}^{i,j}
}
&\quad\text{if }q=\l,
\end{array}\right.
\end{align} 
where ${\rho_{q,\l}^{i,j}=\rho_{q,\l}^{i,j}(\theta')}$ is defined by \eqref{eq:def:rho}.
In addition, the solution of \eqref{eq_general_mstep} in $(\nu_0, \nu)$ is given by
\begin{align}  
&\arg\max_{\nu_0 \in\mathcal{T}_0}
\sum_{(i,j)\in \cA} \log g_{0,\nu_0}(X_{i,j}) \sum_{q\le \l}  \bar\kappa_{q,\l}^{i,j},\:\:\:
\arg\max_{\nu_{q,\l} \in\mathcal{T}}
 \sum_{(i,j)\in \cA}  \kappa_{q,\l}^{i,j} \log(g_{\nu_{q,\l}}(X_{i,j})),\quad  1\le q\le \l\le Q, 
\label{eq:mstep:general_nuql}
\end{align}  
where
 \begin{align*}  
 \bar\kappa_{q,\l}^{i,j}=\left\{\begin{array}{ll}
(\tau_{i,q}\tau_{j,\l} +\tau_{i,\l}\tau_{j,q} )(1-\rho_{q,\l}^{i,j})&\quad\text{if }q\neq \l;\\
\tau_{i,q}\tau_{j,q}(1-{\rho_{q,q}^{i,j}})&\quad\text{if }q=\l.
\end{array}\right.
 \end{align*}  
\end{proposition}

Concerning the maximization in $(\nu_0, \nu)$, we see that the terms to maximize in \eqref{eq:mstep:general_nuql} have the form of  weighted likelihood functions. This implies that the solutions have the form of the traditional ML estimates where sample means are replaced with weighted means. 
For instance, in the Gaussian model \eqref{gaussmodel},  
the solution of \eqref{eq_general_mstep} in $\nu_0=\sigma^2_0$ and $\nu_{q,\l}=(\mu_{q,\l},\sigma^2_{q,\l})$
 is given by
  \begin{align*} 
\hat \mu_{q,\l}&=\frac{\sum_{(i,j)\in \cA} \kappa_{q,\l}^{i,j}X_{i,j} }{\sum_{(i,j)\in \cA} \kappa_{q,\l}^{i,j}},\quad
\hat\sigma_{q,\l}^2= \frac{\sum_{(i,j)\in \cA} \kappa_{q,\l}^{i,j}(X_{i,j} -\hat \mu_{q,\l})^2}{\sum_{(i,j)\in \cA} \kappa_{q,\l}^{i,j}},\quad\forall q\le \l,\\
\hat\sigma_{0}^2&= \frac{\sum_{q\le \l}\sum_{(i,j)\in \cA} \bar\kappa_{q,\l}^{i,j}X_{i,j}^2}{\sum_{q\le \l}\sum_{(i,j)\in \cA} \bar\kappa_{q,\l}^{i,j}},
 \end{align*}  
where $\kappa_{q,\l}^{i,j}$ and $\bar\kappa_{q,\l}^{i,j}$ are given in Proposition~\ref{prop_mstep_pi_w}.\\

Overall, we summarize the VEM algorithm as follows.

 \begin{algorithm}[H]
  \label{VEMalgo}
   \KwData{Observation $X$, number $Q$ of latent groups.}
  \KwResult{{Estimator $\wh{\theta}$},  {clustering $\wh{Z}$}, variational parameters $\tau$.}
   Initialization of $\theta$  {and $\tau$}\;
   \While{not converged}{
  VE-step: update $\tau=\tau(\theta)$ by solving the fix point equation  in Proposition~\ref{prop_fixed_point}\;
M-step: update $\theta$ according to Proposition \ref{prop_mstep_pi_w} \;
 }
{Let $\wh{\theta}=\theta$;}\\ 
{Let $\wh{Z}_i=\arg\max_{q\in\{1,\dots,Q\}}\{ \tau_{i,q} (\theta)\}$, $i\in\{1,\dots,n\}$.}
    \caption{VEM algorithm for the noisy stochastic bloc model}
\end{algorithm}


\section{New procedure for graph inference}
\label{sec:method}

\subsection{Notions of $\ell$-values}

In the NSBM, when inferring the latent parameter $A$, it is well known that the optimal classification rule is the Bayes rule, based on the posterior distribution of $A$.
It is therefore natural to consider  the following quantities as test statistics:
\begin{align}
\l_{i,j}(X,z,\theta)&=\P_\theta(A_{i,j}=0\:|\: X,Z=z), \:\:\:\: (i,j)\in \cA,\: z\in\{1,\dots,Q\}^n,\:\theta\in\Theta .\label{equ-lvaluescond}
\end{align}

We refer to the quantities $\l_{i,j}(X,z;\theta)$ as the {\it $\ell$-values} (it is also called the local FDR, see \cite{Efron2004}). 
Applying Bayes formula, the $\ell$-values can be obtained as follows:
\begin{align}
\l_{i,j}(X,z,\theta)& = \bell(X_{i,j},z_i,z_j,\theta)\:\:\:\: (i,j)\in \cA,\: z\in\{1,\dots,Q\}^n,\:\theta\in\Theta\label{formulalvalueZknown},
\end{align}
for a functional $\bell(\cdot)$ defined by the likelihood ratio
\begin{equation}
\bell(x,q,\l,\theta)= 
  \frac{ (1-w_{q,\l}) g_{0,\nu_0}(x) }{(1-w_{q,\l})g_{0,\nu_0}(x) +w_{q,\l} g_{\nu_{q,\l}}(x) }, \:\:\:\:x\in\R,\: q,\l \in\{1,\dots,Q\} ,\:\theta=(\pi,w,\nu_0,\nu)\in\Theta.
 \label{formulalvalfunction}
\end{equation}
In particular, the latter shows the following  useful property: 
\begin{align}
\left\{\begin{array}{l}\mbox{For all $\theta_0,\theta\in\Theta$, $z\in \{1,\dots,Q\}^n$, conditionally on $Z=z$}\\\mbox{the variables $\l_{i,j}(X,z,\theta)$, $(i,j)\in \cA$,  are independent under $\P_{\theta_0}$.}\end{array}\right.\label{indep}\tag{Indep}
\end{align}

Our multiple testing procedure will thus reject $H_{0,i,j}$ provided that $\ell_{i,j}(X,Z;\theta) \leq t$, for some threshold $t$ to be appropriately chosen. As illustrated on Figure~\ref{fig:optimalGaussian} in the Gaussian case, the induced rejection region for $X_{i,j}$ is not (necessarily) of the classical form $|X_{i,j}|\geq c$, $c>0$, but is driven by the value of $\theta=(\pi,w,\nu_0,\nu)$.

\begin{remark}
Note that $\bell(X_{i,j},q,\l,\theta)$ was already a crucial quantity in the VEM algorithm: it corresponded to the quantities  $1-\rho_{q,\l}^{i,j}(\theta)$ defined by \eqref{eq:def:rho} with  notation of  Section~\ref{sec:Mstep}.
\end{remark}

\subsection{Notion of $q$-values}\label{sec:qvalue}

How to choose $t$ in the decision $\ell_{i,j}(X,Z,\theta) \leq t$? According to our motivation, it should be fixed so that the FDR is smaller than or equal to $\alpha$. However, the FDR is difficult to compute because of the denominator inside the expectation, see \eqref{equ:FDR}. It is therefore useful to consider as a substitute the marginal FDR  \eqref{equ:maFDR}, which should be close to the FDR (at least when the numerator and denominator concentrate around their expectation).

For $\theta,\theta'\in\Theta$ and $t\in[0,1]$, the marginal FDR (under $P_{\theta'}$) of the procedure rejecting the null $H_{0,i,j}$ whenever $\ell_{i,j}(X,Z;\theta) \leq t$ is given by
 the following quantity:  
\begin{align}
\Q_{\theta'}(\theta,t) &=\frac{ \E_{\theta'}\left[\sum_{(i,j)\in\mathcal{A}} (1-A_{i,j}) \ind{\ell_{i,j}(X,Z,\theta) \leq t}\right]}{ \E_{\theta'}\left[\sum_{(i,j)\in\mathcal{A}}  \ind{\ell_{i,j}(X,Z,\theta) \leq t} \right]}\label{equ:Qcond}\\
&=\frac{  \sum_{q,\l} \pi'_q\pi'_\l (1-w'_{q,\l})  \q_0(t,q,\l; \theta',\theta) }{ \sum_{q,\l} \pi'_q\pi'_\l [(1-w'_{q,\l})  \q_0(t,q,\l; \theta',\theta)+w'_{q,\l}  \q_1(t,q,\l; \theta',\theta)] },\label{equ:Qcond2}
 \end{align}
where we let for $\theta=(\pi,w,\nu_0,\nu),\theta' =(\pi',w',\nu'_0,\nu')$, $\delta\in\{0,1\}$, $q,\l\in\{1,\dots,Q\}$,
\begin{align}
\q_\delta(t,q,\l; \theta',\theta)&=\P_{\theta'}(\l_{i,j}(X,Z;\theta) \leq t\:|\:Z, Z_i=q,Z_j=\l, A_{i,j}=\delta)\nonumber \\
&=\P_{\theta'}( \bell(X_{i,j},q,\l;\theta)\leq t\:|\:
 Z_i=q,Z_j=\l, A_{i,j}=\delta) .\label{functionq}
\end{align}
Note that the latter quantity does not depend on  $(i,j)\in\mathcal{A}$.
In the Gaussian case, the quantities $\q_0(t,q,\l; \theta',\theta)$ and $\q_1(t,q,\l; \theta',\theta)$ can be explicitly calculated, see Section~\ref{sec:computGaussian}.  They can be interpreted as the size of the rejection area, under the null and the alternative, respectively, see Figure~\ref{fig:optimalGaussian}. 

\begin{figure}[h!]
\begin{center}
\begin{tabular}{ccc}
\vspace{-0.5cm}
$\sigma_0=1$, $\mu_{q,\l}=1$, $\sigma_{q,\l}=1$&$\sigma_0=1$,  $\mu_{q,\l}=1$, $\sigma_{q,\l}=2$ \\
\vspace{-0.5cm}
\includegraphics[scale=0.4]{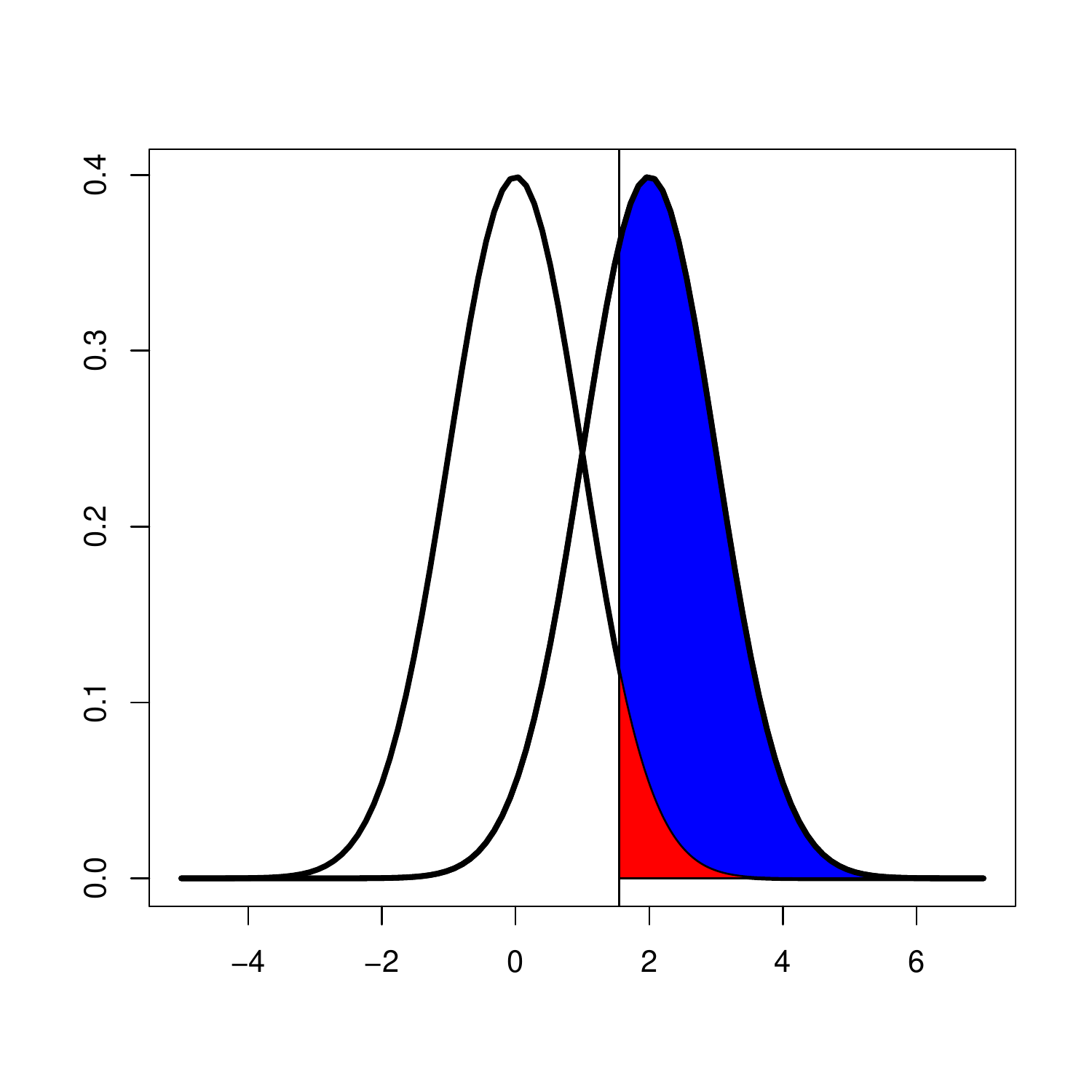}&\includegraphics[scale=0.4]{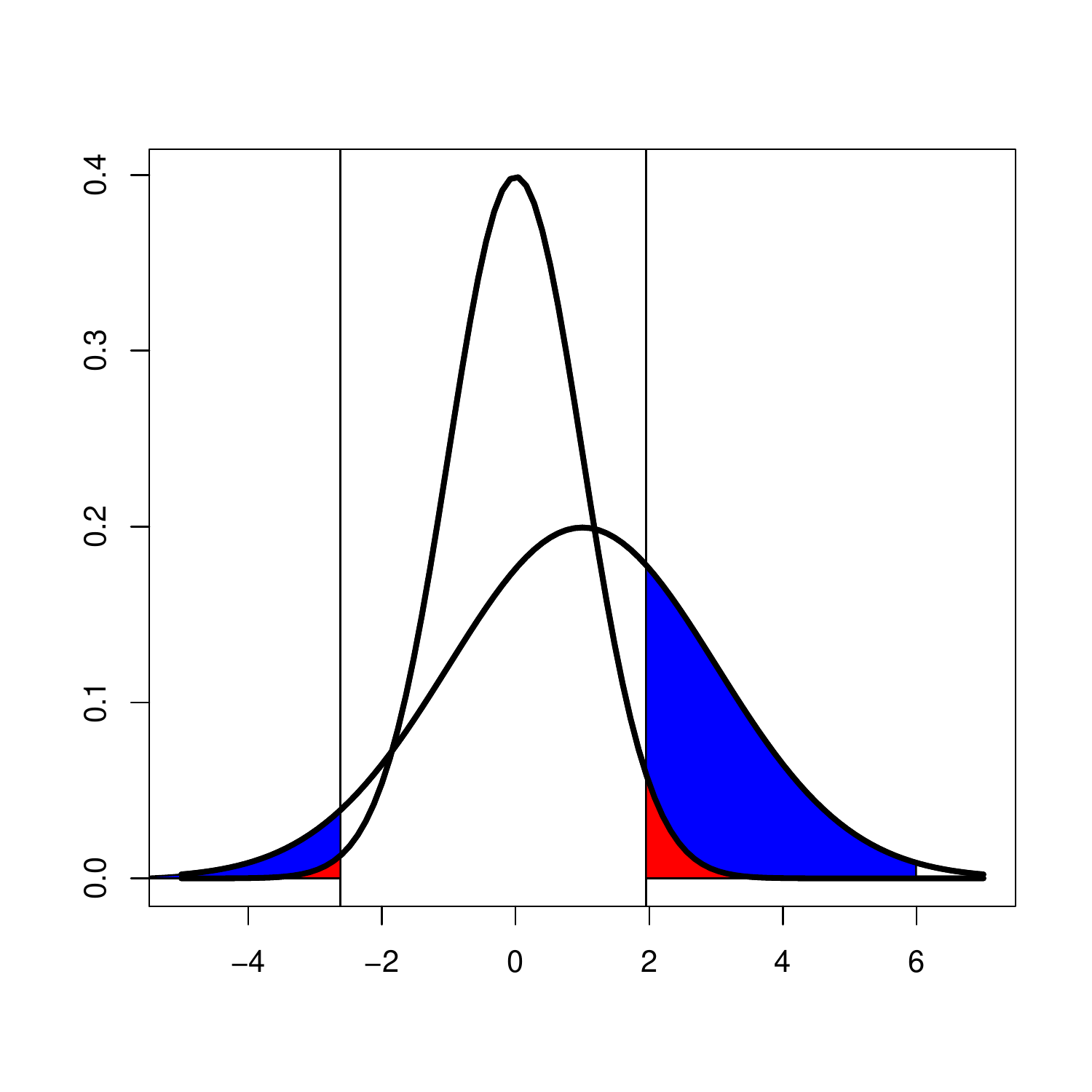}
\end{tabular}
\end{center}
\caption{Rejection area $\{x\::\:\bell(x,q,\l;\theta) \leq t\}$, for two choices of the parameters $(\pi_{q,\l},w_{q,\l},\sigma_0, \mu_{q,\l}, \sigma_{q,\l})$.  $t=0.25$, $w_{q,\l}=0.5$ (does not depend on the value of  $\pi_{q,\l}$). 
The value of $\q_0(t,q,\l; \theta,\theta)$ corresponds to the red area. 
The value of $\q_1(t,q,\l; \theta,\theta)$ corresponds to the red+blue area. 
  \label{fig:optimalGaussian}}
\end{figure}

Now, rejecting the null $H_{0,i,j}$ whenever $\ell_{i,j}(X,Z;\theta) \leq t$ for $t$ the largest such that $\Q_{\theta_0}(\theta,t)\leq \alpha$ leads to the decision of rejecting $H_{0,i,j}$ when 
$\Q_{\theta_0}(\theta,\ell_{i,j}(X,Z;\theta)) \leq \alpha$, that is, $q_{i,j}(X,Z;\theta)\leq \alpha$
where
\begin{align}
q_{i,j}(X,z;\theta)&=\Q_\theta(\theta,\ell_{i,j}(X,z;\theta)), \: (i,j)\in \cA.\label{equ-qvaluescond}
\end{align}  
are referred to as the {\it $q$-values}, a term that comes back to \cite{Storey2003}. 

\subsection{New procedure}\label{sec:proc}
   
When $X$ follows the NSBM with  "true" parameter $\theta_0\in\Theta$ and with latent clustering $Z$,  the oracle $\l$-value (resp. $q$-value) is $\l_{i,j}(X,Z;\theta_0)$ (resp. $q_{i,j}(X,Z;\theta_0)$). 
We thus define the {\it oracle multiple testing procedure} as
\begin{align}\label{def:optimalproc}
\varphi^*_{i,j}=
\ind{q_{i,j}(X,Z;\theta_0)\leq \alpha} ,\:\:(i,j)\in\mathcal{A}.
\end{align}
As proved in Lemma~\ref{lem:optimalproc}, it enjoys the following optimal property: it maximizes the TDR among procedure controlling the MFDR at level $\alpha$ (under appropriate assumptions).  So the procedure $\varphi^*$ will be considered in the sequel as the optimal procedure.

Obviously, the oracle procedure is unknown. However, it can be approximated thanks to the estimator $\wh{\theta}=(\wh{\pi},\wh{w},\wh{\nu}_0,\wh{\nu})$ of $\theta_0$ and the clustering $\wh{Z}$ built in Section~\ref{sec:VEMintro}. Let us define the estimated version of the $q$-values/$\l$-values as follows: 
\begin{align}
\wh \l_{i,j}(X)&= \l_{i,j}(X,\wh{Z};\wh{\theta}) = \frac{(1-\wh w_{\hat{Z}_i,\hat{Z}_j}) g_{0,\hat{\nu}_0}(X_{i,j})}{(1-\wh w_{\hat{Z}_i,\hat{Z}_j})  g_{0,\hat{\nu}_0}(X_{i,j}) +\wh w_{\hat{Z}_i,\hat{Z}_j}g_{\hat{\nu}_{\hat{Z}_i,\hat{Z}_j} }(X_{i,j})}; \label{formulalvalueVEMMAP}\\
\wh q_{i,j}(X)&= q_{i,j}(X,\wh{Z};\wh{\theta}) = \Q_{\hat{\theta}}(\wh{\theta},\wh{\ell}_{i,j}(X)).\label{formulaqvalueVEMMAP}
\end{align}
Then, we define our main procedure as follows: 

\IncMargin{1em}
\begin{algorithm}[H]
  \label{MTVEM}
   \KwData{$X_{i,j}$, $(i,j)\in\mathcal{A}$, level $\alpha$.}
  \KwResult{ $\varphiVEM_{i,j}$, $(i,j)\in\mathcal{A}$.}
  Apply the VEM algorithm of Section~\ref{sec:VEMintro}, that is, Algorithm~\ref{VEMalgo},  to obtain $\wh{\theta}$ and $\wh{Z}$;\\
 Compute the $\ell$-values $\wh{\ell}_{i,j}(X)$ according to \eqref{formulalvalueVEMMAP};\\
 Compute the $q$-values $\wh{q}_{i,j}(X)$ according to \eqref{formulaqvalueVEMMAP};\\
  Set $\varphiVEM_{i,j}=\mathds{1}\{\wh{q}_{i,j}(X)\leq \alpha\}$, $(i,j)\in\mathcal{A}$.
     \caption{Graph inference with FDR control in the noisy SBM}
\end{algorithm}

\section{Theoretical results for the new procedure}\label{sec:theory}

We propose a theoretical study showing that our procedure $\varphiVEM$ both correctly controls the FDR and has a TDR close to the one of the optimal procedure $\varphi^*$ \eqref{def:optimalproc}. 
 
While Sections~\ref{sec:assumption}~and~\ref{sec:results} present the general theory, Section~\ref{sec:exempleGaussian} focus on the Gaussian case. In particular,  a simplified and self-contained Gaussian version of our results can be found in Corollary~\ref{corGaussian}. 
All the results of this section are proved in Section~\ref{sec:proof}.
\subsection{Assumptions and notation}\label{sec:assumption}

According to \eqref{formulaqvalueVEMMAP}, the behavior of $\Q_{\theta}(\cdot)$ is crucial to study the behavior of $\varphiVEM$. Since the latter is related to the functionals $\q_0(\cdot)$ and $\q_1(\cdot)$ via \eqref{equ:Qcond}, we introduce the following assumption.

\begin{assumption}[Regularity]\label{cont} 
\begin{itemize}
\item[(i)] For all $q,\l\in\{1,\dots,Q\}$, 
 the functions 
$(t,\theta,\theta_0)\in [0,1]\times \Theta^2 \mapsto \q_0(t,q,\l; \theta_0,\theta)$ and $(t,\theta,\theta_0)\in [0,1]\times \Theta^2 \mapsto \q_1(t,q,\l; \theta_0,\theta)$ are continuous on $[0,1]\times \Theta^2$.
\item[(ii)] For all  $q,\l\in\{1,\dots,Q\}$ the following holds: there exist functions   $\theta\in \Theta\mapsto t_{1,q,\l}(\theta)\in[0,1]$ and $\theta\in \Theta \mapsto t_{2,q,\l}(\theta)\in[0,1]$ with $t_{1,q,\l}(\theta)<t_{2,q,\l}(\theta)$ for all $\theta$, and such that for any $\theta_0\in \Theta,$ the maps $t\in[0,1]\mapsto  \q_0(t,q,\l; \theta_0,\theta)$ (resp.  $t\in[0,1]\mapsto  \q_1(t,q,\l; \theta_0,\theta)$) are continuous on $[0,1]$, with value $0$ on $[0,t_{1,q,\l}(\theta)]$, increasing on $[t_{1,q,\l}(\theta),t_{2,q,\l}(\theta)]$ and value $1$ on $[t_{2,q,\l}(\theta),1]$.
\end{itemize}
\end{assumption}

In the sequel, we assume that Assumption~\ref{cont} is true and we consider for the $t_{1,q,\l}(\theta),t_{2,q,\l}(\theta)$ defined therein, the boundaries
\begin{align}
t_1(\theta)=\min_{1\leq q,\l\leq Q} \{t_{1,q,\l}(\theta)\} \mbox{ and } t_2(\theta)=\max_{1\leq q,\l\leq Q} \{t_{2,q,\l}(\theta)\},\label{deft1t2}
\end{align}
for any $\theta\in\Theta$. 
Lemma~\ref{lem:functionQvalue} states that the function $t\mapsto \Q_{\theta}(\theta,t)$ has the following simple behavior: it is increasing on $[t_1(\theta),t_2(\theta)]$, continuous on $(t_1(\theta),1]$, satisfies $\Q_{\theta}(\theta,t)=0$ for $t\in [0,t_1(\theta)]$, $\Q_{\theta}(\theta,t)=\pizero$  for $t\in [t_2(\theta),1]$ and $\Q_{\theta}(\theta,t)< t$ for $t\in (t_1(\theta),1]$. 
The latter implies in particular that $t\mapsto \Q_{\theta}(\theta,t)$ is always continuous in $0^+$, but may jump in $t_1(\theta)^+$ when $t_1(\theta)>0$. Illustrations are provided in Figure~\ref{fig:QGaussian} in the Gaussian NSBM.

Since $t\in [0,1]\mapsto \Q_{\theta}(\theta,t)$ is always non-decreasing left-continuous, we can define its (generalized) inverse in $\alpha\in [0,\pi_0]$ by
\begin{align}\label{def:optimalthreshold}
T_{\theta}(\alpha)=\max\{t\in[0,1]\::\: \Q_{\theta}(\theta,t) \leq \alpha\} , \:\:\theta\in\Theta.
\end{align}
this entails that  the optimal procedure $\varphi^*$ and our procedure $\varphiVEM$ can be equivalently written as $\ell$-value thresholding procedures, that is, for $(i,j)\in\mathcal{A}$,
\begin{align*}
\varphi^*_{i,j}&=\ind{\ell_{i,j}(X,Z,\theta_0)\leq T_{\theta_0}(\alpha)} =\ind{\Q_{\theta_0}(\theta_0,\ell_{i,j}(X,Z,\theta_0))\leq \alpha};\\
\varphiVEM_{i,j}&=\ind{\ell_{i,j}(X,\wh{Z},\wh{\theta})\leq T_{\hat{\theta}}(\alpha)} =\ind{\Q_{\hat{\theta}}(\wh{\theta},\ell_{i,j}(X,\wh{Z},\wh{\theta}))\leq \alpha},
\end{align*}
for which we recall that $\theta_0$ is the true value of the parameter.

Now, to show that $\varphiVEM$ is close to $\varphi^*$ in terms of FDR and TDR, there are four ingredients, that we now present.

\paragraph{Super criticality}

Let us fix $\theta_0\in\Theta$ the "true" value of the parameter. First, to avoid the regime where $\Q_{\theta_0}(\theta_0,t)$ is zero, we will consider a level $\alpha$ above the critical level $\alpha_*(\theta_0)$, that is defined as follows:
\begin{align}\label{def:alphastar}
\alpha_*(\theta_0)=\lim_{t\to t_1(\theta_0)^+}\{\Q_{\theta_0}(\theta_0,t)\}\in[0,\pizero).
\end{align}
It corresponds to the infimum of the non-zero values of $t\mapsto \Q_{\theta_0}(\theta_0,t)$. 
While  $\alpha_*(\theta_0)=0$ is the typical case, the case for which $\alpha_*(\theta_0)>0$ is possible when $t_1(\theta_0)>0$. 
This is related to the criticality phenomenon introduced in \cite{Chi2007}.
Sometimes, we will denote $\alpha_*(\theta_0)$ by $\alpha_*$ in the sequel for short. 
Throughout this section, we thus fix a "super-critical" nominal level  $\alpha\in (\alpha_*,\pizero)$.

\paragraph{Perfect clustering and appropriate estimation}

Second, our results rely on the fact that the estimator $\wh{\theta}$ and the clustering $\wh{Z}$ used in the procedure $\varphiVEM$ both have an appropriate behavior, that is, $(\wh{\theta}, \wh{Z})$ is close to $(\theta_0, Z)$ in some sense. Obviously, since the clustering can only be made  up to a permutation of the labels, we should define an appropriate distance between these quantities. In the sequel, we let for any $Z,Z'\in\{1,\dots,Q\}^n$, $\theta,\theta'\in\Theta$, 
\begin{equation}\label{normswitch}
\| (\theta', Z')-(\theta, Z)\| = \min_{\sigma} \left\{ \|(\theta')^\sigma-\theta\|_\infty\vee \|Z' - Z^\sigma\|_\infty   \right\},
\end{equation}
where the minimum is taken over all the permutation of $\{1,\dots,Q\}$, where $\|\cdot\|_\infty$ denote the infinite norm (defined each time on the appropriate subspace), and where $Z^\sigma= (\sigma(Z_i))_{1\leq i\leq n}$ and $\theta^\sigma=(\pi^\sigma ,w^\sigma,\nu_0,\nu^\sigma)$ for $\pi^\sigma=(\pi_{\sigma(q)})_{1\leq q\leq Q}$, $w^\sigma=(w_{\sigma(q),\sigma(\l)})_{1\leq q,\l\leq Q}$, $\nu^\sigma=(\nu_{\sigma(q),\sigma(\l)})_{1\leq q,\l\leq Q}$.

As a consequence, on the event where $\| (\wh{\theta},\wh{ Z}) - (\theta_0, Z)\|\leq \eps$ ($\eps\in (0,1)$), there exists a permutation $\sigma$ such that both $\wh{ Z}=Z^\sigma$ and $\|\wh{ \theta}^\sigma-\theta_0\|_\infty\leq \eps$.
Since in that case $\l_{i,j}(X,\wh{Z};\wh{\theta})=\l_{i,j}(X,Z^{\sigma};\wh{\theta})=\l_{i,j}(X,Z;\wh{\theta}^\sigma)$ and $\Q_{\hat{\theta}^\sigma}(\wh{\theta}^\sigma,t)=\Q_{\hat{\theta}}(\wh{\theta},t)$, we have $\varphiVEM=\varphisimple$, where
\begin{equation}\label{varphisimple}
\varphisimple_{i,j}=\ind{\Q_{\hat{\theta}^\sigma}(\wh{\theta}^\sigma,\l_{i,j}(X,Z,\wh{\theta}^\sigma))\leq \alpha}=\mathds{1}\{\l_{i,j}(X,Z,\wh{\theta}^\sigma)\leq T_{\hat{\theta}^\sigma}(\alpha)\}.
\end{equation}
The latter is easier   to study than $\varphiVEM$ because $\|\wh{ \theta}^\sigma-\theta_0\|_\infty\leq \eps$ and $Z$ is the true clustering. As a counterpart, this adds an error term $\P_{\theta_0}( \| (\wh{\theta}, \wh{Z}) - (\theta_0, Z)\|>\eps)$ in the bound.

\paragraph{Concentration of the FDP process}

We show that the FDP process (the process for which the expectation is the FDR, see \eqref{equ:FDP})  concentrates around the MFDR, in an uniform manner. 
This comes from Lemma~\ref{lem:concFDPprocess}, which relies on the independence property \eqref{indep} and classical DKW-type inequalities. Nevertheless, to get uniformity in the decision class, the complexity of the involved events should be appropriately taken into account. For this, for some positive integer $K>0$, let $\mathcal{I}_K$ be the set of all possible unions of $K$ open intervals of $\R$, that is,
\begin{align}\label{equIk}
\mathcal{I}_K =
 \left\{ \bigcup_{k=1}^K (a_k,b_k), -\infty\leq a_k\leq b_k\leq +\infty \mbox{ for } 1\leq k \leq K,  \mbox{ and } b_{k}\leq a_{k+1}  \mbox{ for } 1\leq k\leq K-1 \right\}.
\end{align}
For instance, the set $(-\infty, -1)\cup (5, 7)$ is in $\mathcal{I}_2$, is also in $\mathcal{I}_3$ (because empty intervals are allowed), but is not in $\mathcal{I}_1$.
We then define the following assumption.

\begin{assumption}[Complexity]\label{AssumptionIk}
$\forall (i,j)\in\mathcal{A},$ $\forall t\in[0,1]$, $\forall \theta\in \Theta$, 
$$\{ \l_{i,j}(X,Z,\theta) \leq t\} =\{ \bell(X_{i,j},Z_i,Z_j,\theta) \leq t\} = \{X_{i,j} \in I\},$$ for some $I \in \mathcal{I}_K$ depending only on $Z_i,Z_j,\theta$ and $t$.
\end{assumption}

\paragraph{Smoothness of the functions $\Q$, $\q_1$ and $T$}

First note that bounding the fluctuations of $\theta\mapsto \Q_{\theta}(\theta,t)$ from those of the functionals $\q_0$ and $\q_1$ is possible when the denominator of $\Q_{\theta}(\theta,t)$ is provided to be away from $0$. For this, we consider an (arbitrary) compact interval $\mathcal{K}\subset (\alpha_*,\pizero)$ such that $\alpha$ belongs to the interior of it, and we let, for $\theta_0=(\pi,w,\nu_0,\nu)$ (see \eqref{equ:Qcond2}),
 \begin{align}
\kappa(\theta_0,\alpha)= \sum_{q,\l} \pi_q \pi_\l &[(1-w_{q,\l})  \q_0(t_\mathcal{K} ,q,\l; \theta_0,\theta_0) \nonumber\\
&+w_{q,\l}  \q_1(t_\mathcal{K},q,\l; \theta_0,\theta_0)] , \mbox{ for } t_\mathcal{K}=T_{\theta_0}(\min \mathcal{K}).\label{equ:kappa}
\end{align}
Obviously, we have $\kappa(\theta_0,\alpha)\in (0,1]$.
Finally, we consider the following continuity moduli:  for all $u\in(0,1)$,
\begin{align}
\mathcal{W}_{\alpha,T}(u)&=\sup\{|T_{\theta_0}(y)-T_{\theta_0}(\alpha)|\: : \:  y\in \mathcal{K}, |y-\alpha|\leq u\}\label{equmodulust};\\
\mathcal{W}_{ T,\q_1}(u)&= \sup_{q,\l} \sup \left\{
\left|
\q_1(t,q,\l; \theta_0,\theta_0)-\q_1(T_{\theta_0}(\alpha),q,\l; \theta_0,\theta_0) \right|, t \in T_{\theta_0}(\mathcal{K}) , |t-T_{\theta_0}(\alpha)|\leq u \right\}\label{equmodulusq1}
\\
\mathcal{W}_{\theta_0,\q}(u)&= \sup_{q,\l} \sup_{t\in T_{\theta_0}(\mathcal{K})}\sup_{\delta\in\{0,1\}} \sup \left\{
\left|
\q_\delta(t,q,\l; \theta',\theta)-\q_\delta(t,q,\l; \theta_0,\theta_0) \right| \::\: \right. \nonumber\\
&\hspace{5cm}\left. \theta,\theta'\in \Theta, \|\theta-\theta_0\|_\infty\leq u ,\|\theta'-\theta_0\|_\infty\leq u \right\}\label{equmodulusq};.
\end{align}
Above, we implicitly used the generic notation "$\mathcal{W}_{x,f}$" for the modulus of the function "$f$" in the point "$x$".

\begin{remark}\label{rem:diff}
By Assumption~\ref{cont}, the limits of the functions $\mathcal{W}_{\theta_0,\q}(u)$, $\mathcal{W}_{ T,\q_1}(u)$ are both equal to zero when $u$ goes to zero. Also, as the inverse of a continuous increasing function, the function $y\in \mathcal{K}\mapsto T_{\theta_0}(y)$ is also continuous increasing and thus the limit of $\mathcal{W}_{\alpha,T}(u)$ is $0$ when $u$ goes to zero. 
In addition, when $T_{\theta_0}$ (resp. $\q_1$, $\q$) is differentiable in $y=\alpha$ (resp. $t=T_{\theta_0}(\alpha)$, $(\theta',\theta)=(\theta_0,\theta_0)$), we have that for 
some constants $e(\theta_0,\alpha,Q), C(\theta,\alpha,Q)>0$, for all $u\in (0,e(\theta_0,\alpha,Q))$, 
 $\mathcal{W}_{\alpha,T}(u)\leq C(\theta_0,\alpha,Q) u$ (resp. $\mathcal{W}_{ T,\q_1}(u)\leq C(\theta_0,\alpha,Q) u$, $\mathcal{W}_{\theta_0,\q}(u)\leq C(\theta_0,\alpha,Q) u)$.
\end{remark}

\subsection{Results}\label{sec:results}

As a first result, we provide  the behavior of the FDR of the procedure $\varphiVEM$.
\begin{theorem}\label{FDRcontrol}
There exist universal constants $c_1,c_1',c_2,c_2'>0$ such that the following holds.
Let  Assumptions~\ref{cont}-\ref{AssumptionIk} be true and let $\theta_0=(\pi,w,\nu_0,\nu)\in\Theta$. Consider $\alpha_*=\alpha_*(\theta_0)$ given by \eqref{def:alphastar}, $\alpha\in \mathcal{K} \subset (\alpha_*,\pizero)$ for some compact interval $\mathcal{K}$, $\kappa=\kappa(\theta_0,\alpha)$ given by \eqref{equ:kappa} and the modulus $\mathcal{W}_{\theta_0,\q}$ defined by \eqref{equmodulusq}.
Let $\pimin=\min_{q}\{\pi_q\}$ and $\wmax=\max_{q,\l} \{w_{q,\l}\}.$
Consider the procedure 
 $\varphiVEM$ of Algorithm~\ref{MTVEM} for the VEM estimator $\wh{\theta}$ and clustering $\wh{Z}$. Then there exists $e=e(\theta_0,\alpha,Q)\in(0,1)$  such that for all $\eps\in (0,e)$,
 for all $x>0$ with $x<\pimin^2\wedge (1-\wmax)$,
\begin{align*}
\FDR\left(\theta_0,\varphiVEM\right) \leq\:&  \alpha +x + 16 \kappa^{-1} (\mathcal{W}_{\theta_0,\q}(\eps)+3Q^2\eps) + \P_{\theta_0}( \| (\wh{\theta}, \wh{Z}) - (\theta_0, Z)\|>\eps)\\
&+ c_1 Q^2 \e^{-c_1'\lfloor n/2\rfloor \kappa^2 x^2 / Q^4} + c_2 K  Q^2 \e^{- c_2' m \pimin^2 (1-\wmax) \kappa^2 x^2/K^2}.
\end{align*}
\end{theorem}

Theorem~\ref{FDRcontrol} is proved in Section~\ref{proofofFDRcontrol}.
It shows that the FDR of $\varphiVEM$ is close to the targeted level $\alpha$, up to a remainder term.

We now turn to the optimality result of the procedure $\varphiVEM$, in terms of the TDR, as defined by \eqref{equ:TDR}.

\begin{theorem}\label{Powercontrol}
Consider the setting of Theorem~\ref{FDRcontrol} and additionally let  $\wmin=\min_{q,\l} \{w_{q,\l}\}$. Consider the functions  $\mathcal{W}_{T,\q_1}$, $\mathcal{W}_{\alpha,T}$ given respectively by  \eqref{equmodulusq1},\eqref{equmodulust} and the  optimal procedure $ \varphi^*$ defined by \eqref{def:optimalproc}.
 Then there exists $e=e(\theta_0,\alpha,Q)\in(0,1)$  such that for all $\eps\in (0,e)$,
 for all $x>0$ with $x<\pimin^2\wedge \wmin$,
\begin{align*}
\pione \TDR\left(\theta_0, \varphiVEM\right) 
\geq&\: \pione \TDR\left(\theta_0, \varphi^*\right)- x
-\pione \P_{\theta_0}( \| (\wh{\theta}, \wh{Z}) - (\theta_0, Z)\|>\eps)   \\
&-2\mathcal{W}_{\theta_0,\q}(\eps)-6Q^2\eps- \mathcal{W}_{T,\q_1}\circ  \mathcal{W}_{\alpha,T} \left(8 \kappa^{-1} (\mathcal{W}_{\theta_0,\q}(\eps)+3Q^2\eps)\right)\\
&-2 Q^2 \e^{-2\lfloor n/2\rfloor x^2 /(9 Q^4)} - 6 K  Q^2 \e^{- m \pimin^2 \wmin x^2/(9K^2)}.
\end{align*}
\end{theorem}

Theorem~\ref{Powercontrol} is proved in Section~\ref{proofofPowercontrol}. It shows that the power of $\varphiVEM$ is close to the one of the optimal procedure $ \varphi^*$, up to a remainder term.

Both FDR and TDR bounds are non-asymptotic, and are available for any fixed $n\geq 2$ although $n$ should be large enough in order to make the remainder terms small. 
Our bounds involve several terms: the concentration term (that decreases exponentially fast), the moduli of continuity (that depends on the regularity of the involved functionals) and the quality of $(\hat{\theta}, \hat{Z})$ as estimation/clustering rules.  

As a side result, since the moduli of continuity have all a zero limit in zero, 
 Theorems~\ref{FDRcontrol}~and~\ref{Powercontrol} entail the following consistency result (proof provided in Section~\ref{sec:proofmaincor} for completeness).

\begin{corollary}\label{maincor}
Let  Assumptions~\ref{cont}-\ref{AssumptionIk} be true and let us consider an asymptotic in $n$ (and thus also in $m=n(n+1)/2$) for which the parameter $\theta_0\in\Theta$ is kept fixed (does not depend on $n$) and assume that the VEM estimator $\wh{\theta}$ and clustering $\wh{Z}$ are consistent, that is, such that $\P_{\theta_0}( \|( \wh{\theta}, \wh{Z}) - (\theta_0, Z)\|>\eps)$ converges to $0$ for any $\eps>0$ as $n$ tends to infinity. Consider $\alpha_*=\alpha_*(\theta_0)$ given by \eqref{def:alphastar}, $\alpha\in (\alpha_*,\pizero)$. Then the procedure 
 $\varphiVEM$ of Algorithm~\ref{MTVEM} and the procedure $ \varphi^*$ defined by \eqref{def:optimalproc} satisfy
\begin{align*}
\limsup_n \left\{ \FDR\left(\theta_0,\varphiVEM\right)\right\} \leq\:&  \alpha,\:\:\:\liminf_n \{ \TDR\left(\theta_0, \varphiVEM\right) -\TDR\left(\theta_0, \varphi^*\right)\}\geq 0.
\end{align*}
\end{corollary}

Let us mention that establishing the consistency of the VEM estimator/clustering has been investigated in \cite{Celisse_etal,bickel2013,Maria2017,Maria2019} in different SBM-type models. 

To our knowledge, Theorems~\ref{FDRcontrol}~and~\ref{Powercontrol}  are the first non-asymptotic bounds showing FDR control and TDR optimality in a mixture model context. In comparison, results in \cite{SC2007,CS2009,SC2009,CSW2019} only establish consistency. 
Here, our non-asymptotic bounds provide more informations: for instance,  when the moduli $\mathcal{W}_{\alpha,T}(u)$, $\mathcal{W}_{ T,\q_1}(u)$ and $\mathcal{W}_{\theta_0,\q}(u)$ are all smaller than some constants times $u$, we can choose $x=\eps=\sqrt{(\log n)/n}$ so that for $\eta_n=\P_{\theta_0}\left( \| (\wh{\theta}, \wh{Z}) - (\theta_0, Z)\|>\sqrt{(\log n)/n}\right)$, we have
\begin{align*}
\FDR\left(\theta_0,\varphiVEM\right) &\leq  \alpha +O\left(\sqrt{(\log n)/n}\right)  + \eta_n\\
 \TDR\left(\theta_0, \varphi^*\right)- \TDR\left(\theta_0, \varphiVEM\right)&\leq O\left(\sqrt{(\log n)/n}\right)  + \eta_n.
\end{align*}
Our bounds thus entail convergence rates, although the remainder term $\eta_n$ could in principle deteriorate this rate. 

\begin{remark}
Inspecting the proofs, Theorems~\ref{FDRcontrol}~and~\ref{Powercontrol} (and Corollary~\ref{maincor}) extend to any procedure of the form $\varphi_{i,j}=\ind{q_{i,j}(X,\wh{Z};\wh{\theta})\leq \alpha}$ that uses some estimator  $\wh{\theta}$ and clustering $\wh{Z}$ (possibly different of those coming from a VEM algorithm).
\end{remark}

\subsection{Application to the Gaussian case}\label{sec:exempleGaussian}

Let us illustrate our results in the Gaussian model \eqref{gaussmodel}. The properties will depend on the chosen parameter set $\Theta$, that can take various form. The basic parameter set is
\begin{align}
\Theta&=\:\left\{(\pi,w,\sigma_0,\mu,\sigma)\in (0,1)^Q \times (0,1)^{Q(Q+1)/2}\times (0,\infty) \times \R^{Q(Q+1)/2} \times (0,\infty)^{Q(Q+1)/2}\right. ;\nonumber\\
&\:\:\:\:\:\:\:  \left. \sum_{q=1}^Q \pi_q=1, \mbox{ the elements of }\{(0,\sigma_0),(\mu_{q,\l},\sigma_{q,\l}), 1\leq q\leq \l\leq Q \} \mbox{ are all distincts}\right\}.\label{sec:thetagauss}
\end{align}
We define also the following parameter sets including additional constraints:
\begin{itemize}
\item $\Theta_{\sigma_0}$ with  $\sigma_{q,\l}=\sigma_0, 1\leq q\leq \l\leq Q$ (and thus also $\mu_{q,\l}\neq 0$);
\item $\Theta_{\sigma^+_0}$  with  $\sigma_{q,\l}> \sigma_0, 1\leq q\leq \l\leq Q$ (alternatives with higher variance and possibly zero mean);
\end{itemize}

A detailed study of the Gaussian model is done in Section~\ref{sec:computGaussian} and we report here only some consequences for the above parameter sets.
First, we can check that Assumptions~\ref{cont}~and~\ref{AssumptionIk} both hold:  the complexity assumption holds with $K=2$. The regularity assumption holds with 
$t_{1,q,\l}(\theta)=0$ when $\theta\in \Theta_{\sigma_0}$ or $ \theta\in\Theta_{\sigma^+_0}$ and $t_{2,q,\l}(\theta)=1$ for $\theta\in \Theta_{\sigma_0}$ or
$$
t_{2,q,\l}(\theta)=\left(1+\frac{w_{q,\l}}{1-w_{q,\l}} \frac{\sigma_0}{\sigma_{q,\l}} \exp\left(\frac{\mu_{q,\l}^2}{2(\sigma_0^2-\sigma_{q,\l}^2)}\right)\right)^{-1}
$$
when $\theta=(\pi,w,\sigma_0,\mu,\sigma)\in \Theta_{\sigma^+_0}$. This implies $\alpha_\star(\theta)=0$ both for $\theta\in \Theta_{\sigma_0}$ and $\theta\in \Theta_{\sigma^+_0}$.
An illustration is given in Figure~\ref{fig:QGaussian}.
Now, let $A(\theta)=\pizero(\theta)$ for $\theta\in \Theta_{\sigma_0}$ and 
\begin{equation}\label{equ:Atheta}
A(\theta)=\min_{\substack{(q,\l)\in \{1,\dots,Q\}^2\\ \sigma_{q,\l}\neq \sigma_0}}\left\{ \Q_{\theta}(\theta,t_{2,q,\l}(\theta))\right\}\:\:\: \mbox{ for $\theta\in\Theta_{\sigma^+_0}$.}
\end{equation}
Then the following result holds:

\begin{corollary}\label{corGaussian}
Consider the Gaussian NSBM with parameter set being either $\Theta=\Theta_{\sigma_0}$ or $\Theta=\Theta_{\sigma^+_0}$, the procedure $\varphiVEM$ of Algorithm~\ref{MTVEM} for the VEM estimator $\wh{\theta}$ and clustering $\wh{Z}$ and the optimal procedure $ \varphi^*$ defined by \eqref{def:optimalproc}. 
Let $\theta_0\in \Theta$ and $\alpha\in (0,1)$. Then the following  holds:
\begin{itemize}
\item[(i)] in an asymptotic in $n$ for which the parameter $\theta_0$ is kept fixed (does not depend on $n$) and the VEM estimator $\wh{\theta}$ and clustering $\wh{Z}$ are consistent, that is, such that $\P_{\theta_0}( \|( \wh{\theta}, \wh{Z}) - (\theta_0, Z)\|>\eps)$ converges to $0$ for any $\eps>0$ as $n$ tends to infinity, we have
\begin{align*}
\limsup_n \left\{ \FDR\left(\theta_0,\varphiVEM\right)\right\} \leq\:&  \alpha,\:\:\:\liminf_n \{ \TDR\left(\theta_0, \varphiVEM\right) -\TDR\left(\theta_0, \varphi^*\right)\}\geq 0.
\end{align*}
\item[(ii)]
If $\alpha\in (0,A(\theta_0))$ for the quantity $A(\theta_0)$ defined by \eqref{equ:Atheta}, then there exists a constant $C=C(\theta_0,\alpha,Q)$ and an integer $N=N(\theta_0,\alpha,Q)$ such that if $n\geq N$, 
\begin{align*}
\FDR\left(\theta_0,\varphiVEM\right) \leq\:&  \alpha + C \eps_n  + \P_{\theta_0}\left( \| (\wh{\theta}, \wh{Z}) - (\theta_0, Z)\|> \eps_n \right),
\end{align*}
for any sequence $\eps_n\geq \sqrt{(\log n)/n}$.
\item[(iii)] if in addition $\alpha\in (0,A(\theta_0))\backslash \Lambda$, for some $\Lambda\subset [0,1]$ of Lebesgue measure $0$, we have 
\begin{align*}
\TDR\left(\theta_0, \varphi^*\right)
\leq&\TDR\left(\theta_0, \varphiVEM\right) + C \eps_n  + \P_{\theta_0}\left( \| (\wh{\theta}, \wh{Z}) - (\theta_0, Z)\|> \eps_n \right),
\end{align*}
for any sequence $\eps_n\geq \sqrt{(\log n)/n}$.
\end{itemize}
\end{corollary}

Point (i) is a direct consequence of Corollary~\ref{maincor} above. Point (ii) is a consequence of Theorem~\ref{FDRcontrol} and of the fact that the modulus $\mathcal{W}_{\theta_0,\q}(u)$ is smaller than some constants (depending on $\theta_0,\alpha,Q$) times $u$ in the Gaussian case with a parameter set $\Theta=\Theta_{\sigma_0}$ or $\Theta=\Theta_{\sigma^+_0}$. The latter comes from Section~\ref{sec:computGaussian} (or more precisely Section~\ref{sec:computMgaussian} therein). 
Point (iii) is a consequence of Theorem~\ref{Powercontrol} and of the fact that  $t\mapsto \q_1(t,q,\l; \theta_0,\theta_0)$ is differentiable in $t=T_{\theta_0}(\alpha)$ when $\alpha<A(\theta_0)$, as proved in Section~\ref{sec:studyq}. Also, we use that $T_{\theta_0}$ is continuous increasing so is almost everywhere differentiable on $(0,\pizero(\theta_0))$. It is thus differentiable in $\alpha$, up to remove a subset of Lebesgue measure equal to zero. 

Let us provide some rationale behind Corollary~\ref{corGaussian}: point (i) means that, when inferring consistently the parameter and the clustering, the procedure $\varphiVEM$ consistently mimics the FDR/TDR of the optimal procedure. When the quality of the  parameter/clustering estimation is additionally obtained with a rate, we can deduce a convergence rate on the FDR/TDR, by tuning the rate $\eps_n$ into (ii) and (iii) (the final rate being in any case not faster than $\sqrt{(\log n)/n}$).

\begin{remark}
Dealing with a Gaussian parameter set $\Theta_{\sigma^-_0}$  with $\sigma_{q,\l}< \sigma_0, 1\leq q\leq \l\leq Q$ (alternatives with smaller variance) is also possible up to reduce the $\alpha$ range, see Section~\ref{sec:computGaussian} for more details.
\end{remark}

\begin{figure}[h!]
\begin{center}
\begin{tabular}{cc}
\vspace{-0.5cm}
    Case $1$ &Case $2$\\
\vspace{-0.5cm}
\includegraphics[scale=0.5]{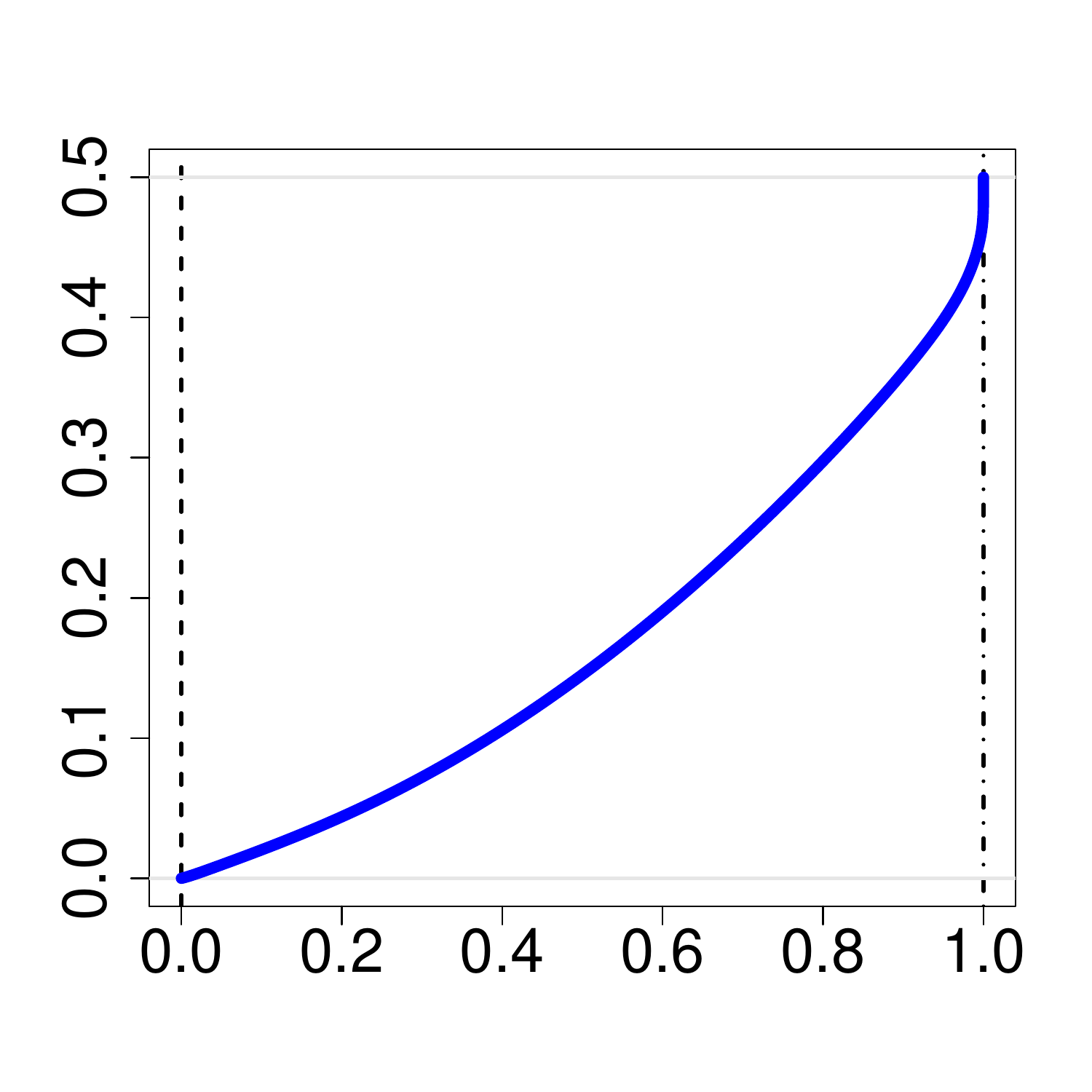}&\includegraphics[scale=0.5]{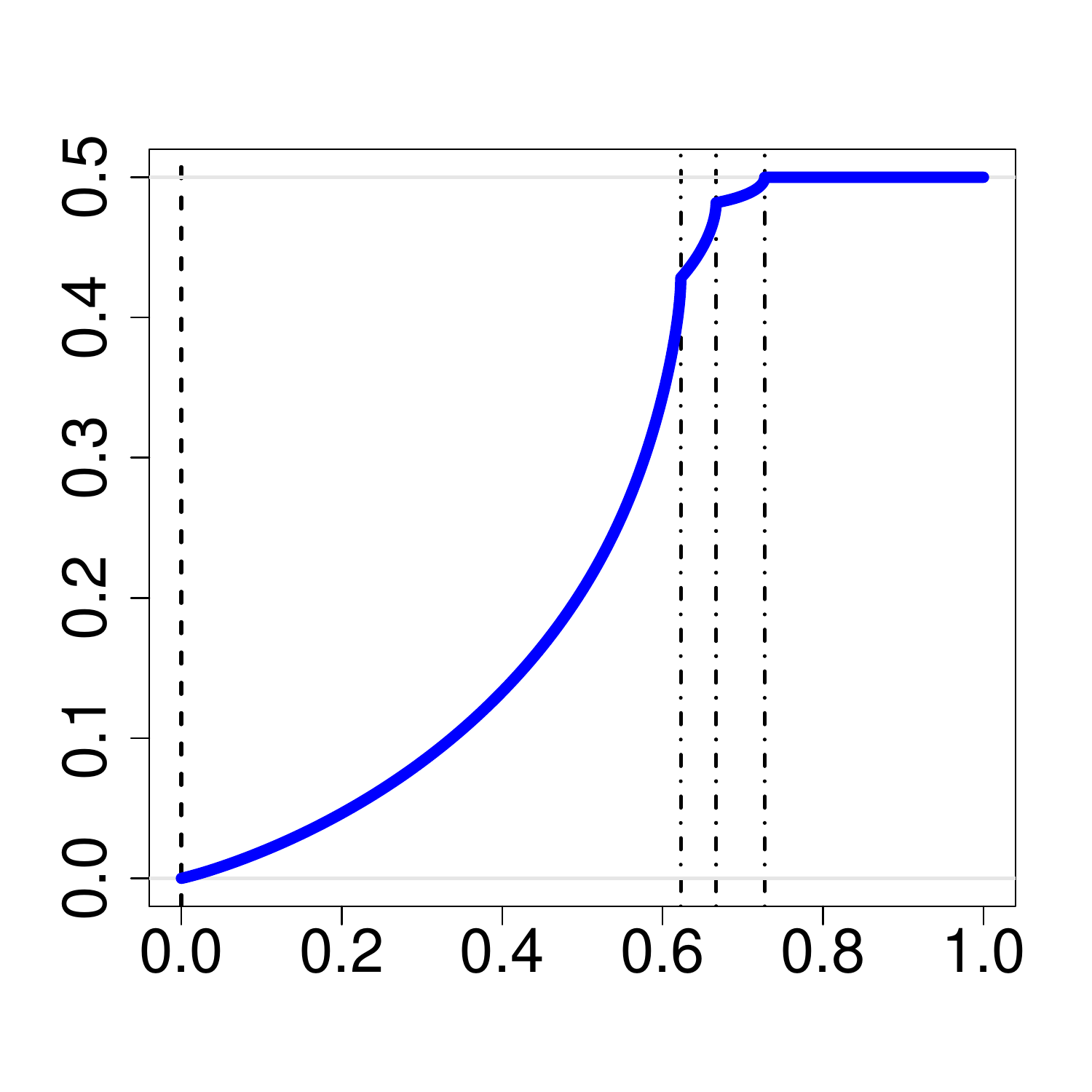}\\
\vspace{-0.5cm}
\end{tabular}
\end{center}
\caption{Plot of $t\mapsto \Q_{\theta}(\theta,t)$ defined by \eqref{equ:Qcond} in the Gaussian case, for $2$ different values of the parameter $\theta$. In each case, the vertical dashed (resp. dashed-dotted) lines correspond to $\{t_{1,q,\l}\}_{q,\l}$ (resp.  $\{t_{2,q,\l}\}_{q,\l}$). In all cases we have $\pi_{q,\l}=0.5$ for all $q,\l$. $Q=2$. $w=(0.4,0.5,0.5,0.6)$. $\sigma_0=(1,1,1,1)$. For case $1$: $\mu=(1,-2,-2,4)$, $\sigma=\sigma_0$. For case $2$: $\mu=(0,0,0,0)$, $\sigma=(1.1,2,2,4)$. 
 \label{fig:QGaussian}}
\end{figure}

\section{Numerical experiments}\label{sec:simu}

In this section, we explore the numerical performances of the new procedure $\varphiVEM$ and compare it to other standard procedures controlling the FDR.

\subsection{Practical issues on the VEM algorithm}\label{subsec:vem_icl}

First, the implementation of the VEM algorithm raises two practical issues. The first is  the choice of the convergence criterion for the algorithm. We may monitor  the value of $J(\theta;\tau;\theta')$ defined in \eqref{def:critJ} and stop the algorithm when its value remains relatively stable. The second and more difficult issue is  initialization. It is well known that the quality of the solution of any EM-type algorithm  heavily  depends on the good the choice of initial value. 
{
For this, we apply the standard $k$-means algorithm to cluster the rows of the (symmetric) observation matrix $(X_{i,j})_{i,j}$ yielding  initial values for $\tau$, which in turn can be used to compute a first parameter value of $\theta$ by using a M-step (or a variation thereof).
In addition, we follow the standard approach of running the  algorithm several times with different initializations and select the best run afterwards.
}

Another considerable problem in practice is the   selection of  the optimal number of latent groups $Q$ in the NSBM. 
Here we use the classical integrated classification likelihood (ICL) approach \citep{Biernacki2000}, which can be interpreted as the penalized  observed likelihood criterion, where the penalty is  the sum of the traditional BIC penalty and of the entropy of the latent variable distribution. The entropy is large when  the uncertainty of the underlying clustering  is high, so that  the quality of the obtained clustering is taken into account in the model selection procedure. 
More precisely, the ICL criterion is given in our model  by
\begin{align}\label{icl}
\mathrm{ICL}(Q) 
&= \tilde\E_{\hat \theta^{[Q]},\hat\tau^{[Q]}}[\log \mathcal{L}(X,A,Z;\hat \theta^{[Q]})\:|\:X] + \mathrm{pen_{BIC}}(Q),
\end{align}
where $\hat \theta^{[Q]}$ and $\tau^{[Q]}$ are the output of the VEM algorithm with $Q$ groups, and
  $\mathrm{pen_{BIC}}(Q)$ denotes the BIC penalty, which  is (roughly) the number of model parameters multiplied with the logarithm of the number of observations. In the NSBM, the parameter $\theta$ splits into 
two parts: for the group proportion vector $\pi$, there are $n$ observations corresponding to the nodes, while for the other parameters $w,\nu_0,\nu$ there are $m$ observations corresponding to  the observed edges, which leads to 
\begin{align*}
\mathrm{pen_{BIC}}(Q)
&= - (Q-1)\log n - \left((1+d_1)\frac{Q(Q+1)}2+  d_0\right)\log m.
\end{align*}
Now,  for some given maximal number $Q_{\max}$ of groups, the number of latent groups $\widehat Q$ chosen by the ICL criterion is given by
$$\widehat Q = \arg\max_{1\leq Q\leq Q_{\max}}\{\mathrm{ICL}(Q)\}.$$

\subsection{Procedures}

{We use the new procedure $\varphiVEM$ with the adjustment described in the previous section ($Q_{\max}=3$).}
As benchmarks, we consider the BH procedure (BH) at level $\alpha$ described in Section \ref{subsec:BH} (with the true sigma $\sigma_0$) and to the so-called adaptive BH procedure (ABH) that corresponds to the BH procedure taken at level $\alpha/ \pizerohat $.
The rationale is that BH controls the FDR at level $\pizero\alpha$ instead of $\alpha$, see \eqref{BHcontrol}. Hence, this correction improves BH by making the achieved FDR closer to $\alpha$, see \cite{BKY2006}. 
We consider two versions of ABH based on the two following estimators:
\begin{enumerate}
\item ABH-Storey: $ \pizerohat = \frac{1 + \sum_{(i,j) \in \mathcal{A}} \ind{p_{i, j}(X) >
      0.5}}{m \times 0.5}$ proposed in \cite{Storey2002} (parameter $\lambda=0.5$), and using the estimator of \cite{SS1982};
\item ABH-VEM: $ \pizerohat=\pizero(\wh{\theta})$, see \eqref{equpi0pi1}, where $\wh{\theta}$ is the VEM estimator coming from Algorithm~\ref{MTVEM}. 
\end{enumerate}

\subsection{Scenario 1: case of an NSBM}

The first setting is as follows: we consider the Gaussian NSBM with $n=100$ nodes, $Q=2$ latent groups and  equal group probabilities
  $\pi_q = 1/Q$ for $q \in \{1, \dots, Q\}$. 
To evaluate the effect of the expected proportion of non-connected
  vertices $\pizero$ defined in \eqref{equpi0pi1}, the parameter $w$ is of
  the form :
  $$w=c_w \times \left(\begin{matrix}
  0.8 & 0.2\\
  0.2 & 0.8 \\
  \end{matrix}
    \right),
  $$ with $c_w \in \{1, 0.5, 0.2\}$ such that the expected proportion of
    non connected edges $\pizero$ equals respectively $0.5, 0.75$ and $ 0.9$. 
  The variances $\sigma_0$ and $\sigma_{q,\l}$,  $1\leq q,\l\leq Q$,
  are all set to $1$. For the alternative means:  we consider both the cases of equal means 
  $\mu_{q,\l}$, $1\leq q,\l\leq Q$, all equal to $0.5$, $1$ or $2$, and the case of different means $\mu=(2,1,1,-3)$ (strong signal when connection probability is high) or $\mu=(1,3,3, -1)$ (strong signal when connection probability is low).
  
For each parameter,  the FDR (mean proportion of the discovered edges that are not in the graph) and the TDR (mean proportion of the edges of the graph that are discovered) of the different procedures $\varphiVEM$, BH, ABH-Storey, ABH-VEM, are estimated with $500$ replications, for a targeted FDR level $\alpha$ taken in the range $\{0.005, 0.025,
0.05, 0.1, 0.15, 0.25\}$. Hence, displaying for each $\alpha$ the point $(\FDR,\TDR)$ provides a ROC-type curve, showing simultaneously if the FDR control is correct and which of the procedures is the most powerful. 
Figure~\ref{fig:sbm_Q_2_w} displays the result for $\mu=(2,2,2,2)$ and $\pizero \in\{0.75,0.9\}$, while Figure~\ref{fig:sbm_Q_2_mu} displays the result for $\pizero= 0.5$ and $\mu=(0.5,0.5,0.5,0.5)$,  $\mu=(2,2,2,2)$, $\mu=(2,1,1,-3)$ and $\mu=(1,3,3, -1)$. 

In all the considered configurations, while $\varphiVEM$ has an FDR close to the target level $\alpha$, it clearly outperforms the other procedures in terms of TDR. 
This is in accordance with the theoretical result, see Corollary~\ref{corGaussian}.
 Markedly, the TDR enhancement can be particularly important. For instance, when  $\mu=(2,2,2,2)$, $\pi_0=0.5$, $\alpha=0.1$, the TDR for ABH-Storey is $\approx 50\%$ while the one of the new procedure is above $85\%$. This supports that coordinate-wise decisions (like (A)BH) are suboptimal and that incorporating the clustering information is essential for inferring the graph.

As a side result, we note that the estimator $\pizerohat$ coming from the VEM algorithm improves the Storey estimator, as ABH-VEM has an FDR much closer to $\alpha$ than ABH-Storey.

\begin{figure}[h!]
\begin{center}
\begin{tabular}{cc}
\vspace{-0.5cm}
$ \pizero=0.75$ & $\pizero= 0.9 $ \\
\vspace{-0.5cm}
\includegraphics[scale=0.4]{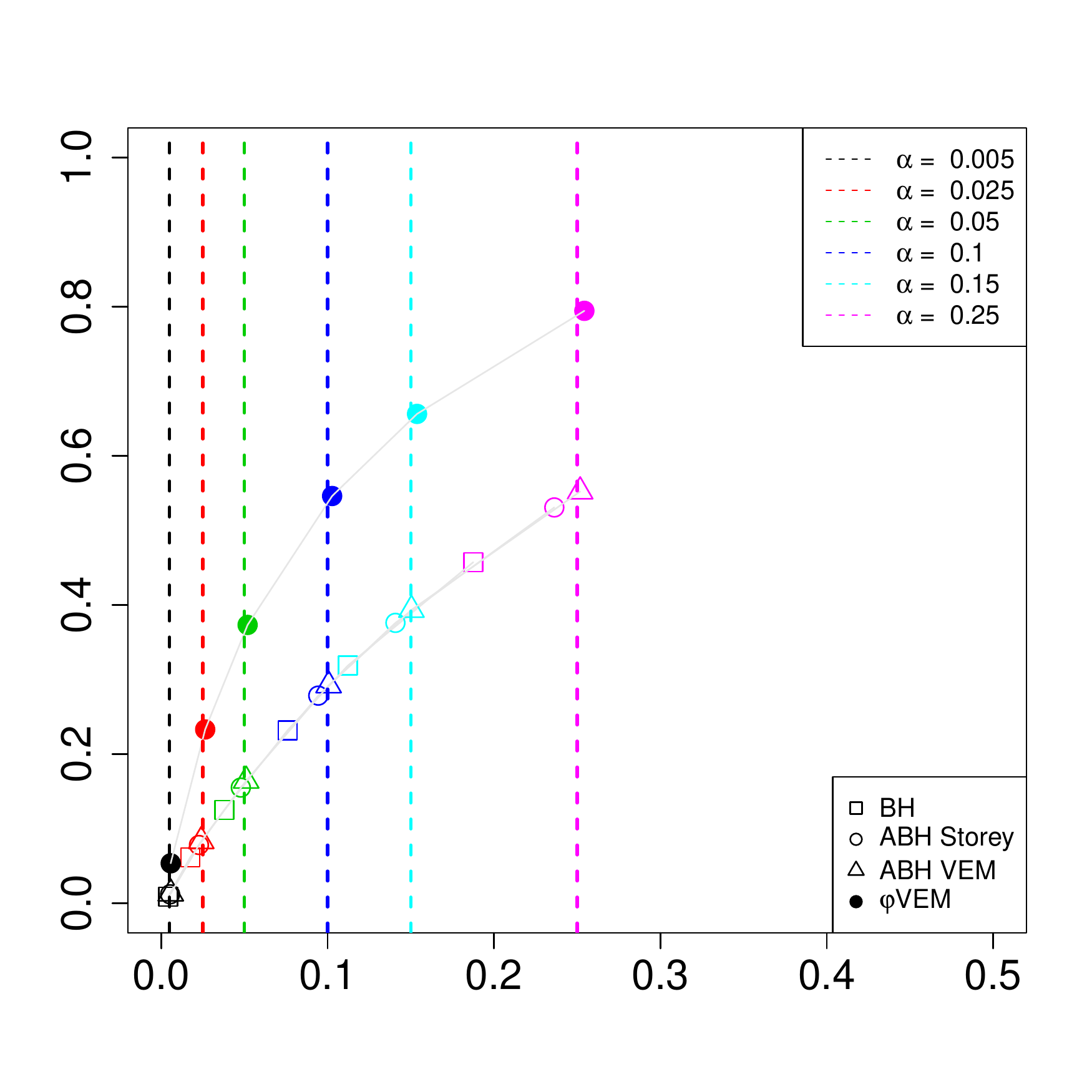}&\includegraphics[scale=0.4]{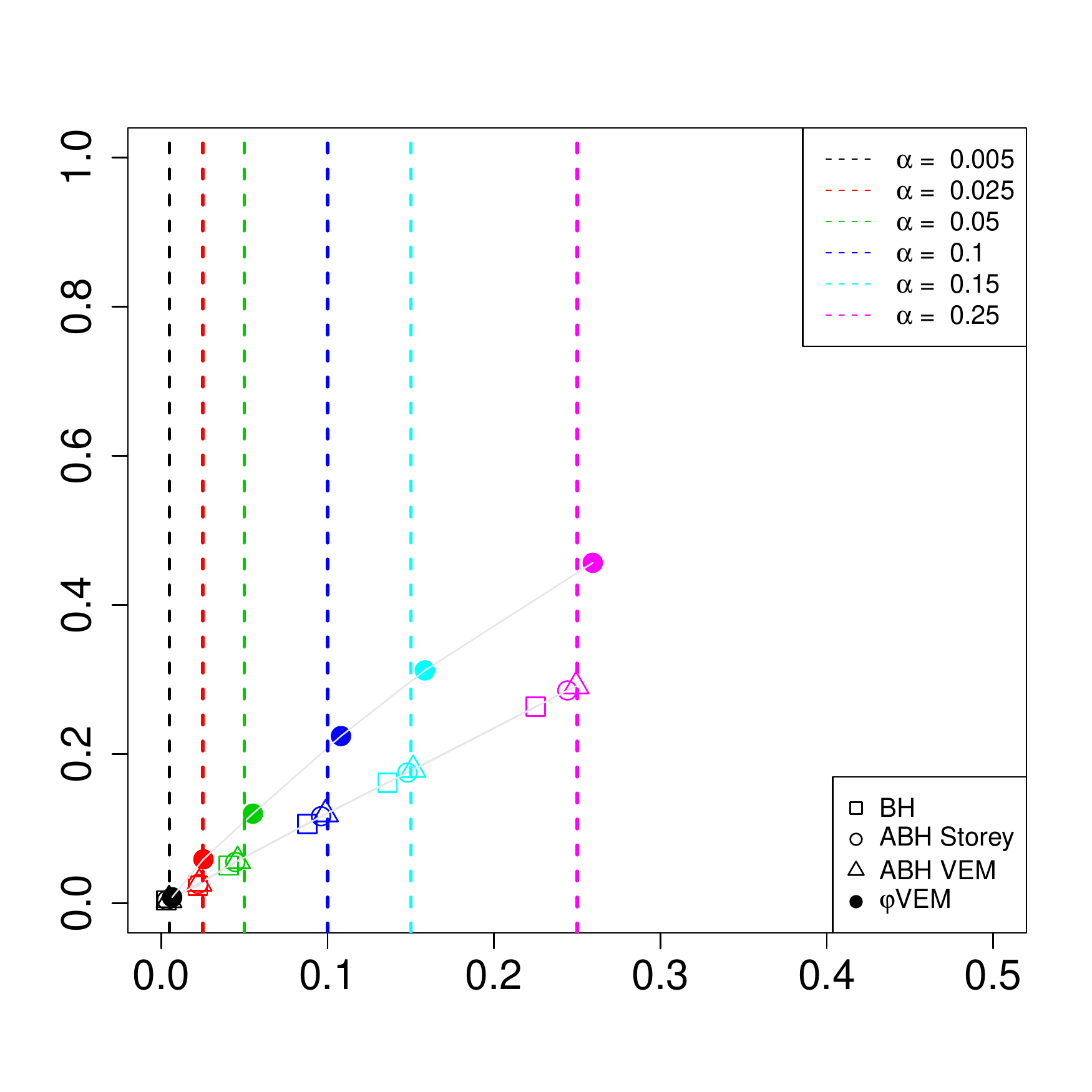}\\
\end{tabular}
\end{center}
\caption{$(\FDR,\TDR)$ for the $4$ different procedures $\varphiVEM$, BH, ABH-Storey, ABH-VEM taken at level $\alpha \in
  \{0.005, 0.025, 0.05, 0.1, 0.15, 0.25 \}$. The vertical dashed lines display the possible values of $\alpha$.
   Scenario 1 with $\sigma_0=1$,
  $\sigma_{q,\l}=1$, $\mu_{q,\l}=2$, $ q,\l   \in \{1,\dots,Q\} $, and $\pi_0\in\{0.75,0.9\}$.
 \label{fig:sbm_Q_2_w}}
\end{figure}

\begin{figure}[h!]
\begin{center}
\begin{tabular}{cc}
\vspace{-0.5cm}
$\mu=(0.5,0.5,0.5,0.5)$ & $\mu=(2,2,2,2)$\\
\vspace{-0.5cm}
\includegraphics[scale=0.4]{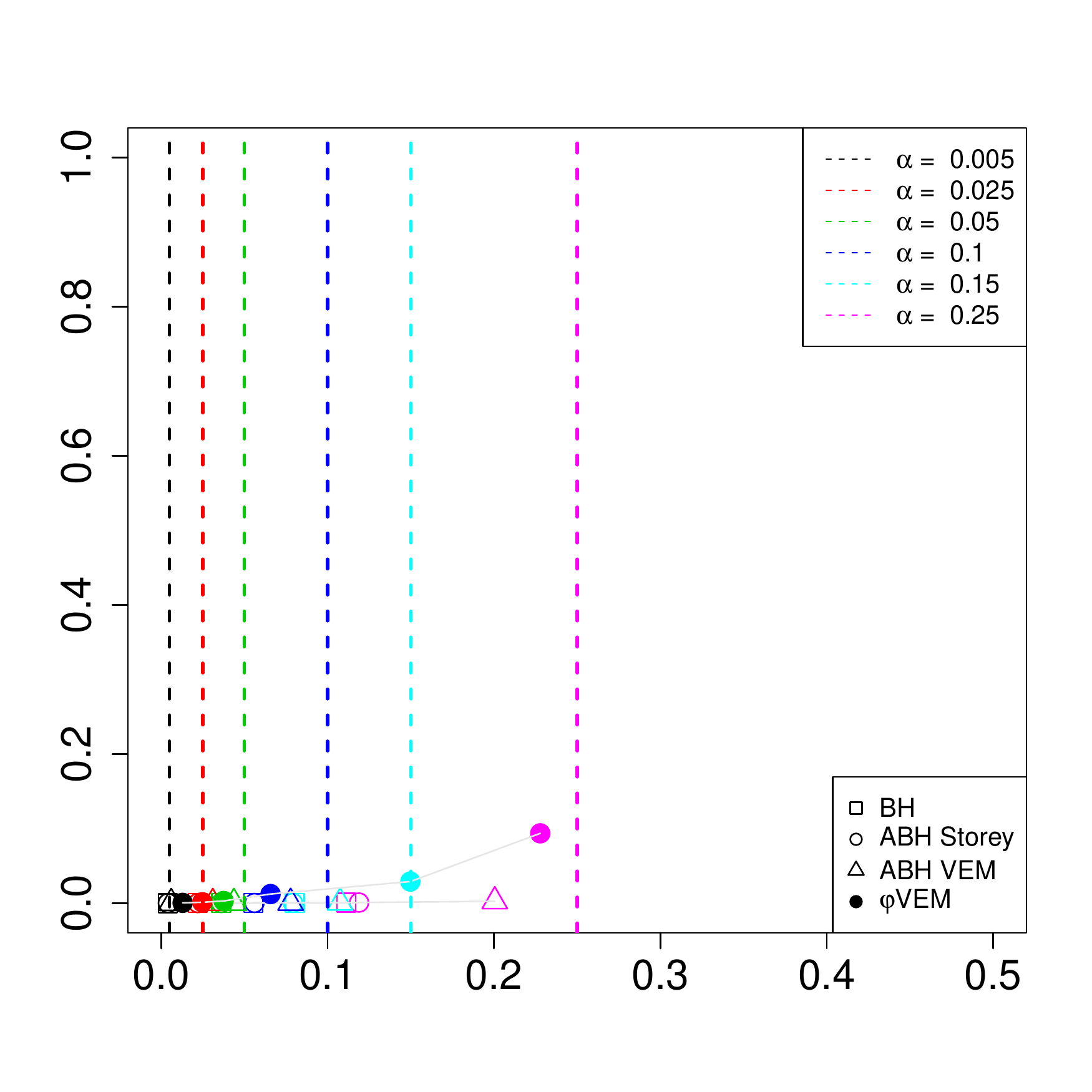}&\includegraphics[scale=0.4]{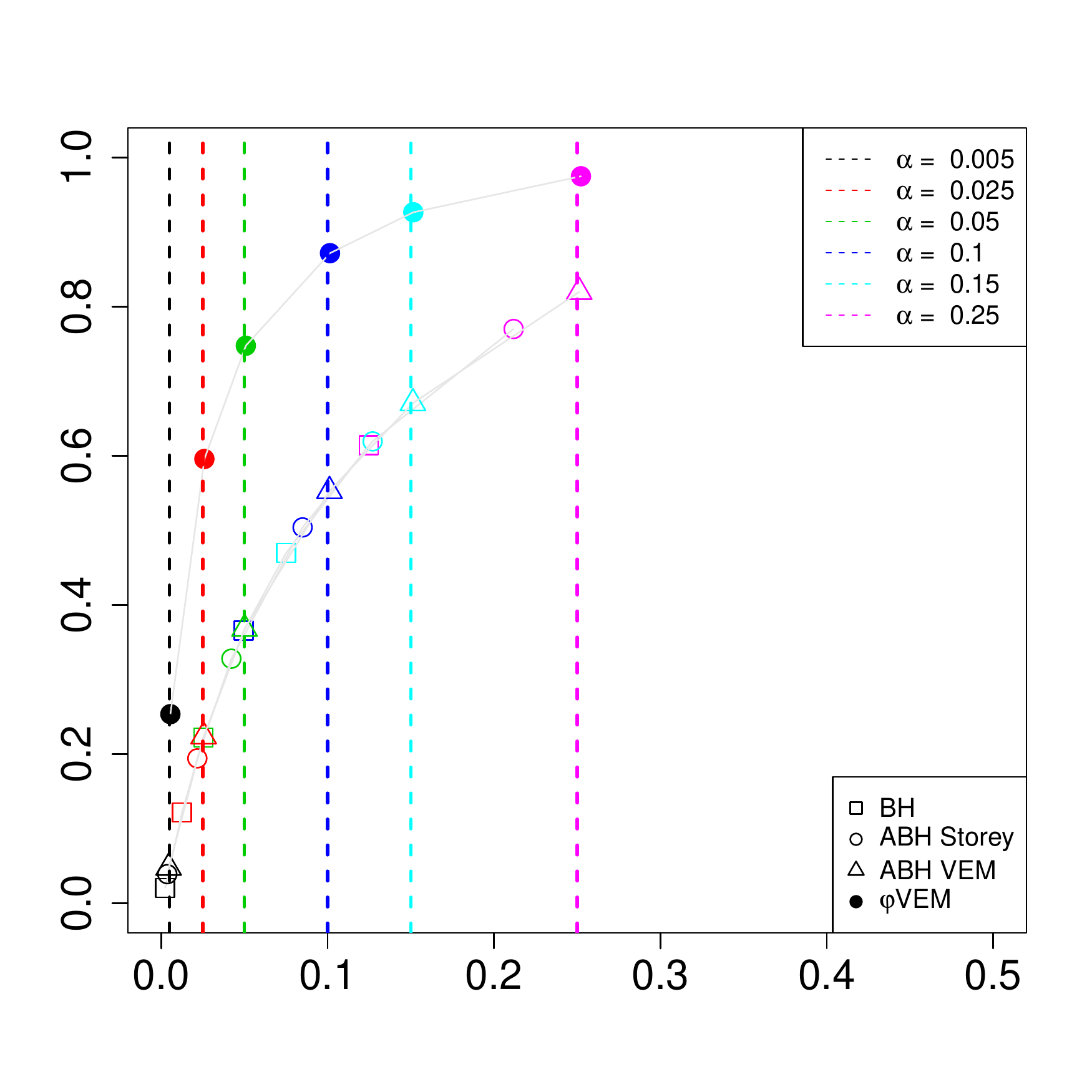}\\
\vspace{-0.5cm}
$\mu=(2,1,1,-3)$& $\mu=(1,3,3, -1)$\\
\vspace{-0.5cm}
\includegraphics[scale=0.4]{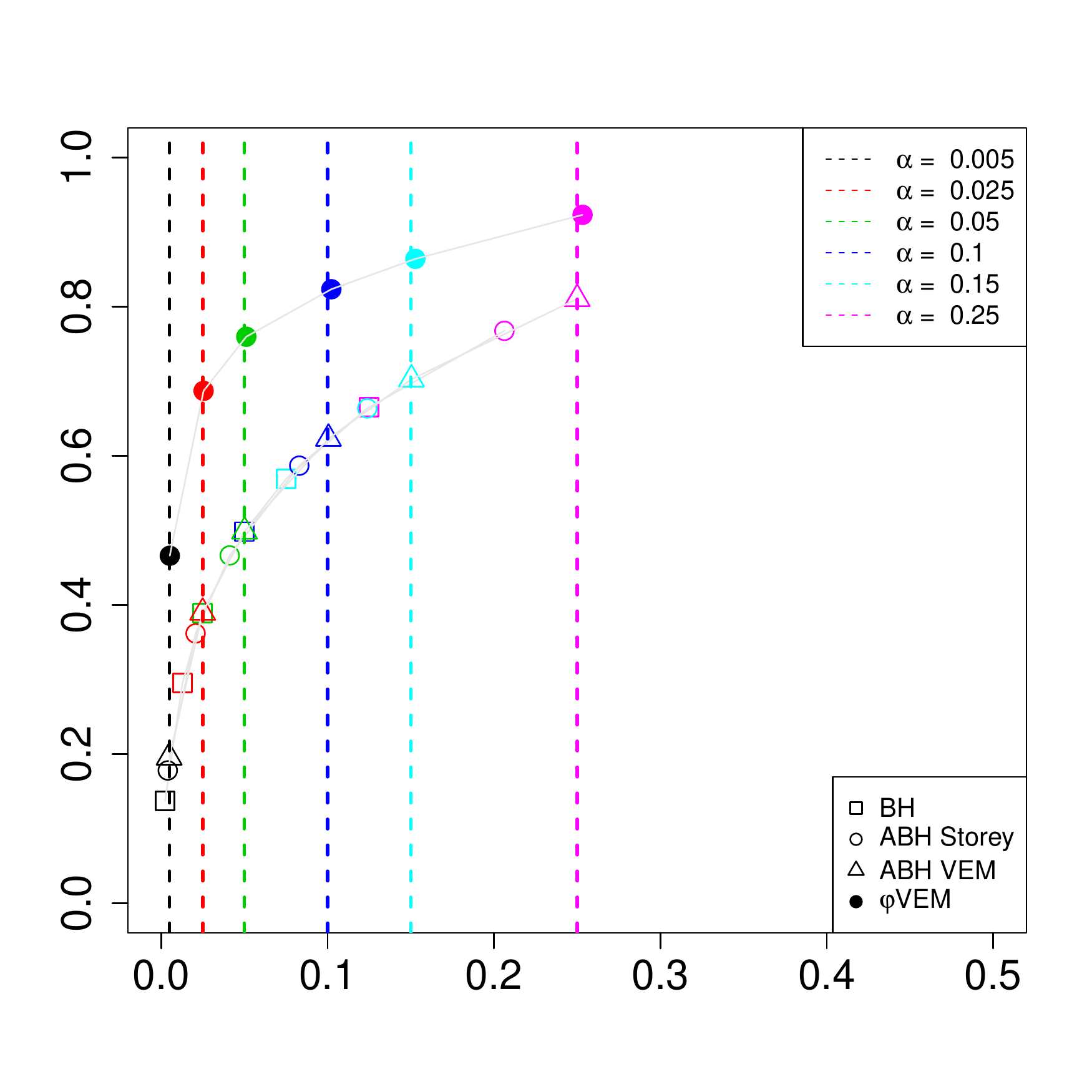}&\includegraphics[scale=0.4]{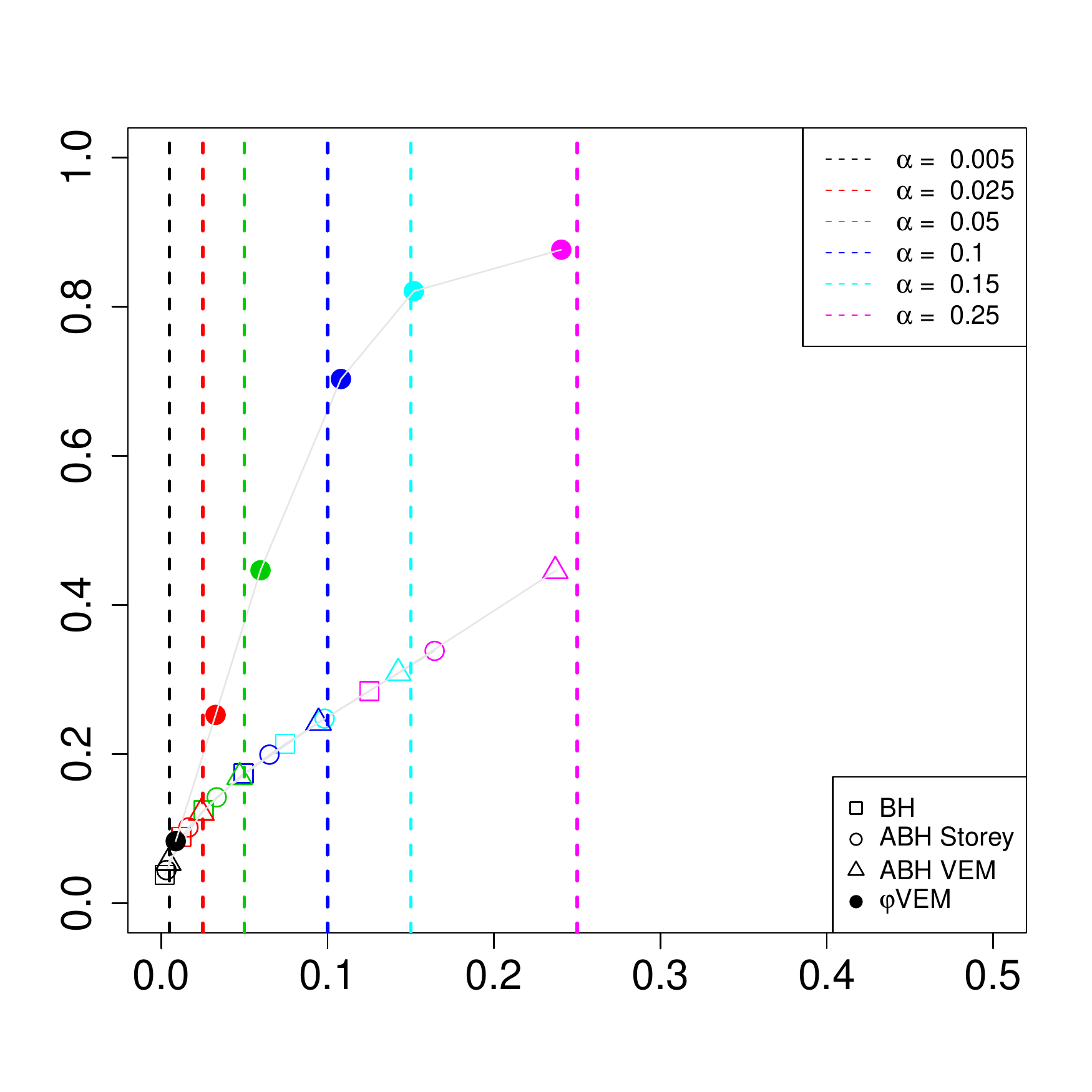}
\end{tabular}
\end{center}
\caption{Same as Figure~\ref{fig:sbm_Q_2_w} for $\pizero= 0.5$ and different $\mu$.
 \label{fig:sbm_Q_2_mu}}
\end{figure}

\subsection{Scenario 2: case of a fixed graphs}\label{sec:scenario2}

In the result of the previous section, one could object that since the new procedure is tailored to the NSBM, it is not surprising that it outperforms procedures that are valid more generally, like BH. It is thus particularly important to complete the study by exploring the robustness of the new procedure, that is,  its behavior  outside the NSBM.

For this, let us consider a {\it deterministic} graph $A =(A_{i,j})_{(i,j)\in\mathcal{A}}$ and generate independently $X_{i,j} \sim \mathcal{N}(0,1)$ when $A_{i,j}=0$ and $X_{i,j} \sim \mathcal{N}(2,1)$ when $A_{i,j}=1$. Here, we underline that $A$ does not change when generating the data $X$ and does not contain any group information, so is different from an NSBM. 
Nevertheless, we can still use the new procedure $\varphiVEM$, that will fit an NSBM on this graph and make a graph inference accordingly.
We consider the two structures displayed in Figure~\ref{fig:structurescenario2} for the graph $A$.

\begin{figure}[h]
\begin{center}
\begin{tabular}{cc}
Star & \hspace{1.5cm}Spider\\
&\\
\begin{tikzpicture}[scale=0.85]
 \tikzstyle{quadri}=[circle,draw,text=black,thick]
 \tikzstyle{estun}=[-,>=latex,very thick]
 \node[quadri] (1) at (0,0) {$1$};
 \node[quadri] (2) at (0,2) {$2$};
  \node[quadri] (3) at (-1.4,1.4) {$3$};
 \node[quadri] (4) at (-2,0) {$4$};
 \node[quadri] (5) at (-1.4,-1.4) {$5$};
 \node[quadri] (6) at (0,-2) {$6$};
 \node[quadri] (7) at (1.4,-1.4) {$7$}; 
 \node[quadri] (8) at (2,0) {$8$}; 
 \node[quadri] (9) at (1.4,1.4) {$9$};  
  \draw[estun] (1)--(2);
  \draw[estun] (1)--(3);
 \draw[estun] (1)--(4);
 \draw[estun] (1)--(5);
 \draw[estun] (1)--(6);
 \draw[estun] (1)--(7);
 \draw[estun] (1)--(8);
 \draw[estun] (1)--(9);
 \end{tikzpicture}
&\hspace{1.5cm}
\begin{tikzpicture}[scale=0.85]
 \tikzstyle{quadri}=[circle,draw,text=black,thick]
 \tikzstyle{estun}=[-,>=latex,very thick]
 \node[quadri] (1) at (0,0) {$1$};
 \node[quadri] (2) at (0,2) {$2$};
  \node[quadri] (3) at (-1.4,1.4) {$3$};
 \node[quadri] (4) at (-2,0) {$4$};
 \node[quadri] (5) at (-1.4,-1.4) {$5$};
 \node[quadri] (6) at (0,-2) {$6$};
 \node[quadri] (7) at (1.4,-1.4) {$7$}; 
 \node[quadri] (8) at (2,0) {$8$}; 
 \node[quadri] (9) at (1.4,1.4) {$9$};  
  \draw[estun] (1)--(2);
  \draw[estun] (1)--(3);
 \draw[estun] (1)--(4);
 \draw[estun] (1)--(5);
 \draw[estun] (1)--(6);
 \draw[estun] (1)--(7);
 \draw[estun] (1)--(8);
 \draw[estun] (1)--(9);
  \draw[estun] (2)--(3);
 \draw[estun] (3)--(4);
 \draw[estun] (4)--(5);
 \draw[estun] (5)--(6);
 \draw[estun] (6)--(7);
 \draw[estun] (7)--(8);
 \draw[estun] (8)--(9);
 \draw[estun] (9)--(2);
 \end{tikzpicture}
\end{tabular}
\end{center}
\caption{Two graph structures used in scenario 2.
 \label{fig:structurescenario2}}
\end{figure}
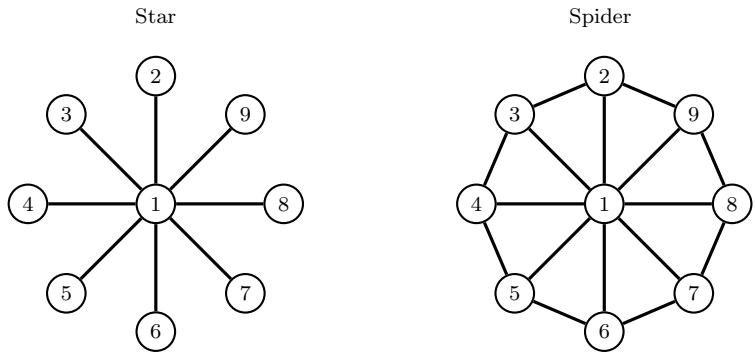

The results are displayed in Figure~\ref{fig:scenario2} for $n=100$ nodes. 
For the star structure, the procedure $\varphiVEM$ is particularly
powerful even if the graph is very sparse (the proportion of connected
nodes equals $2/n=0.02$) and the improvement over BH-like procedures is extreme.
It turns out that fitting an NSBM model to a star structure is particularly beneficial here. 
The VEM-ICL algorithm find $\wh{Q}=2$ groups (with high probability), the center of the star 
 forming the first group and the other nodes the second group.
Since the connection probability between the two groups is $\wh{w}_{1,2}\approx 1$,
the corresponding $\ell$-values are very small and the star can be efficiently recovered.
For the spider structure, 
the procedure $\varphiVEM$ also provides a much higher TDR while still controlling the FDR.
However, the TDR improvement is less extreme in that case because the NSBM fitted is the same as for the Star structure  (with high probability): 
hence, while the edges of the star are still correctly recovered, detecting the other edges is more difficult.

\begin{figure}[h!]
\begin{center}
\begin{tabular}{cc}
\vspace{-0.5cm}
Star  & Spider\\
\vspace{-0.5cm}
\includegraphics[scale=0.4]{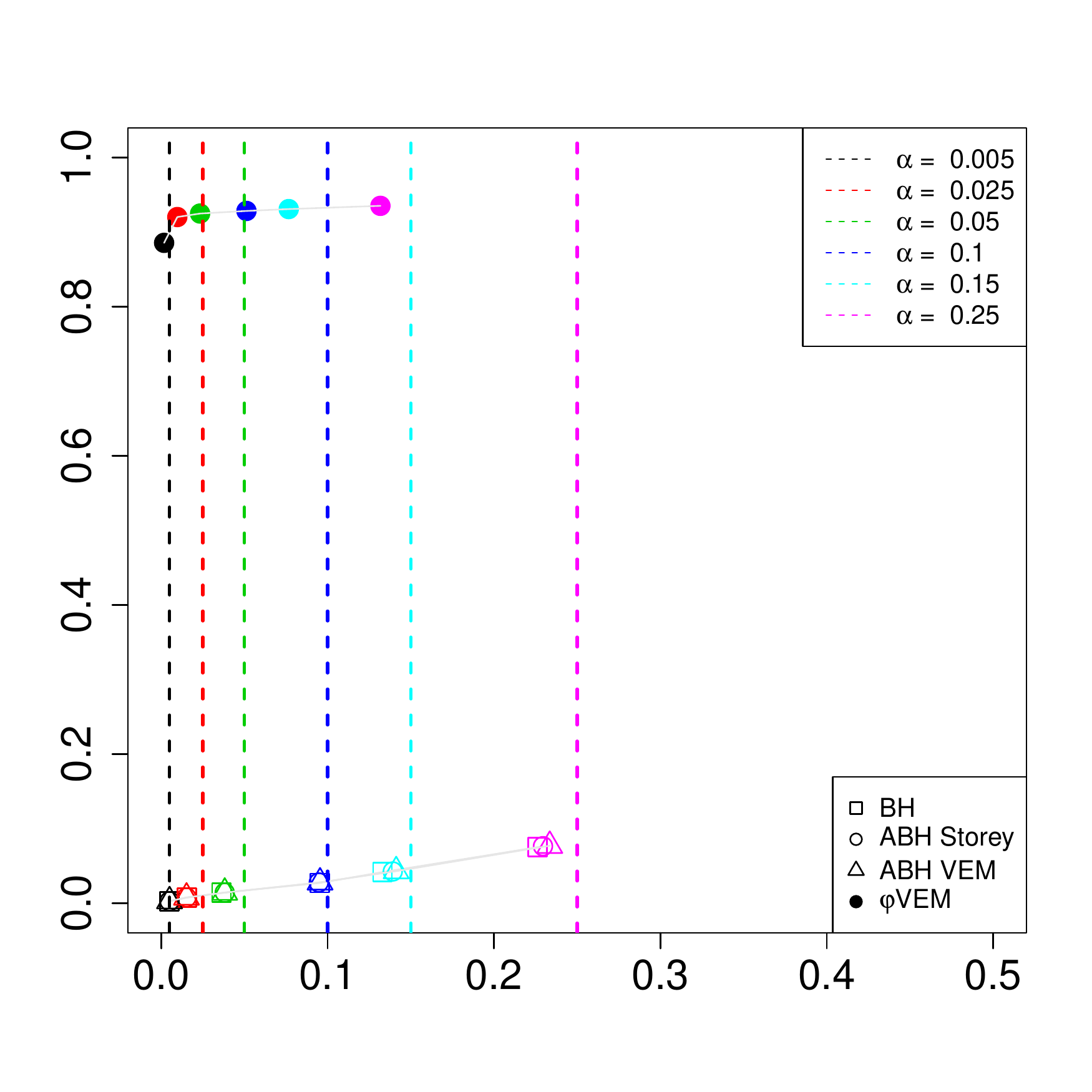}& \includegraphics[scale=0.4]{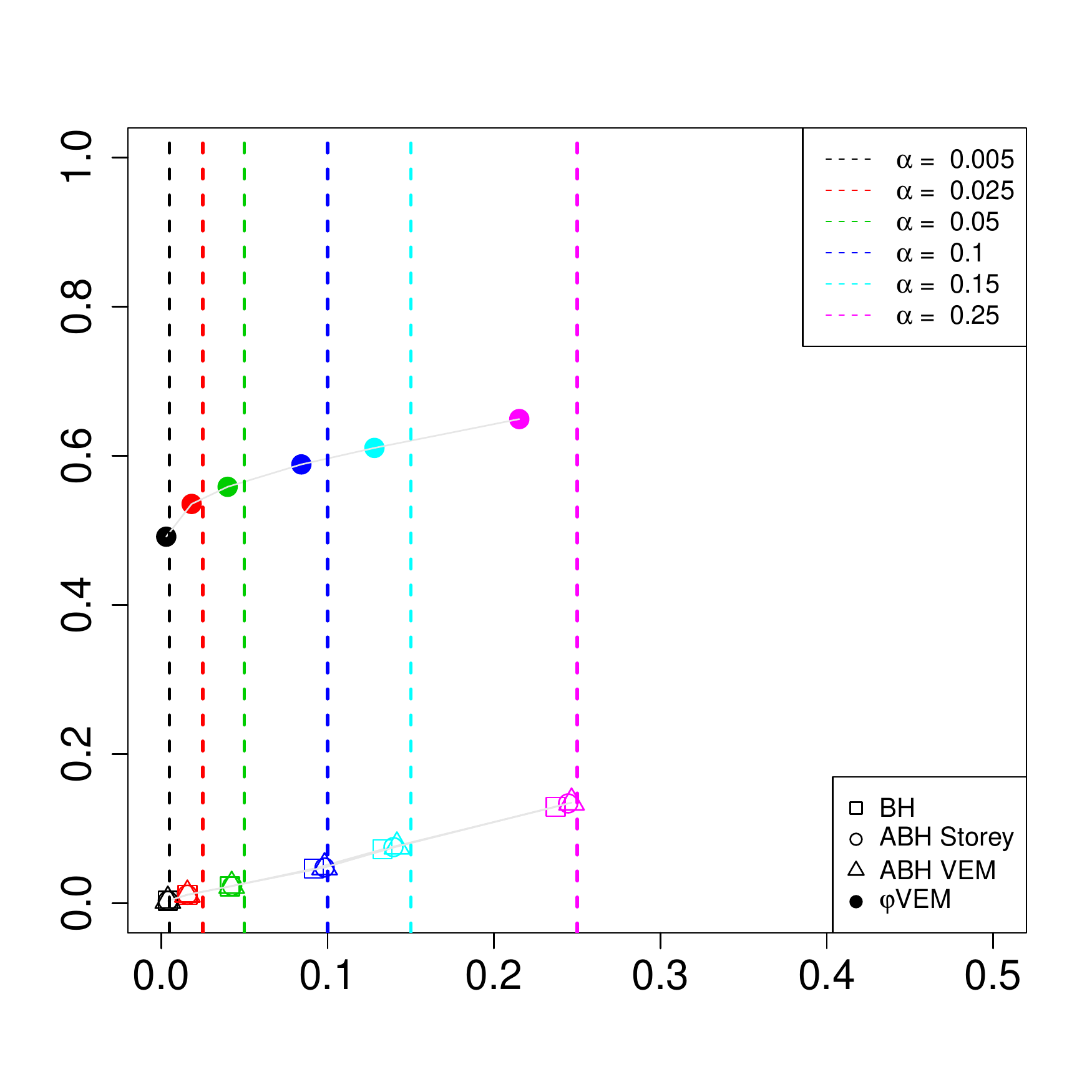}
\vspace{-0.5cm}
\end{tabular}
\end{center}
\caption{Same as Figure~\ref{fig:sbm_Q_2_w}, in scenario 2.
 \label{fig:scenario2}}
\end{figure}

\subsection{Scenario 3: case of non-NSBM random graphs}

To pursue our robustness investigation, we propose to consider the same data-generating process as the previous section, except that $A$ is a priori randomly generated, according to some non-NSBM distribution. We consider three ways to generate $A$ ($n=100$ nodes each time):

 \begin{enumerate}
\item Erdos-R\`enyi $G(n,M)$ model (without replacement). The graph
  has $M$ edges being chosen uniformly
  randomly from the set of all possible edges.  We choose $M=0.2 \times n(n-1)/2$ to obtain a graph whose proportion of
 connected nodes is $0.2$.  
    
\item Bipartite random graph. A bipartite graph is a graph whose
  nodes can be divided into two disjoint sets such that every edge
  connects a node in one set to a node in the other set. We choose
  two sets with the same number of nodes $n/2$. To avoid a too dense graph, nodes between the two sets are
  not always connected here, but only with some probability $p$. Here $p\in\{0.2,0.5\}$.

\item  Preferential attachment model (or the so-called Barab\'asi-Albert model).  A graph of $n$ nodes is built sequentially from a root
  graph by following some growing process, attaching new nodes
  each with a given number ($n/10$ here) of edges that are preferentially attached to existing nodes
  with high degree. Here, we iterate this process until we obtain a graph with $n=100$ nodes and the root graph is generated as an Erdos-R\'enyi graph $G(n_0,p_0)$ (with replacement) with $n_0=n/5$ nodes and a
  probability of connections $p_0=0.5$. This gives a
  graph $A$ with about $20\%$ of edges. 

\end{enumerate}

The results are displayed in Figure~\ref{fig:scenario3}.
The procedure behaves qualitatively as in scenario 1 (FDR control and TDR enhancement), with slightly less improvement in the cases Erdos-R\`enyi. This is well expected because the latter has typically no structure and the VEM-ICL algorithm find no group ($\wh{Q}=1$, with high probability). Nevertheless, even in this case, $\varphiVEM$ provides improvement over (A)BH, because it learns the parameters of the alternative distribution and thus, the optimal decision. This is in line with the findings of the seminal work of \cite{SC2007} in the area of optimal multiple testing for mixture models.
In the bipartite model, the power improvement is better in case $p=0.5$ than $p=0.2$, because the structure in two
sets of nodes is stronger and thus can be easily learned by the algorithm. 
In the preferential attachment model, even if this model is well known to be not of the SBM type, 
the algorithm is still able to learn some part of the structure to increase power.
In that case, the procedure fits an SBM by selecting most of the
time $\wh{Q}=2$ groups with nodes of high degree in one group and nodes with
lower degree in the other one.

\begin{figure}[h!]
\begin{center}
\begin{tabular}{cc}
\vspace{-0.5cm}
Erdos-R\`enyi $G(n,M)$  & Bipartite $p=0.2$ \\
\vspace{-0.5cm}
\includegraphics[scale=0.4]{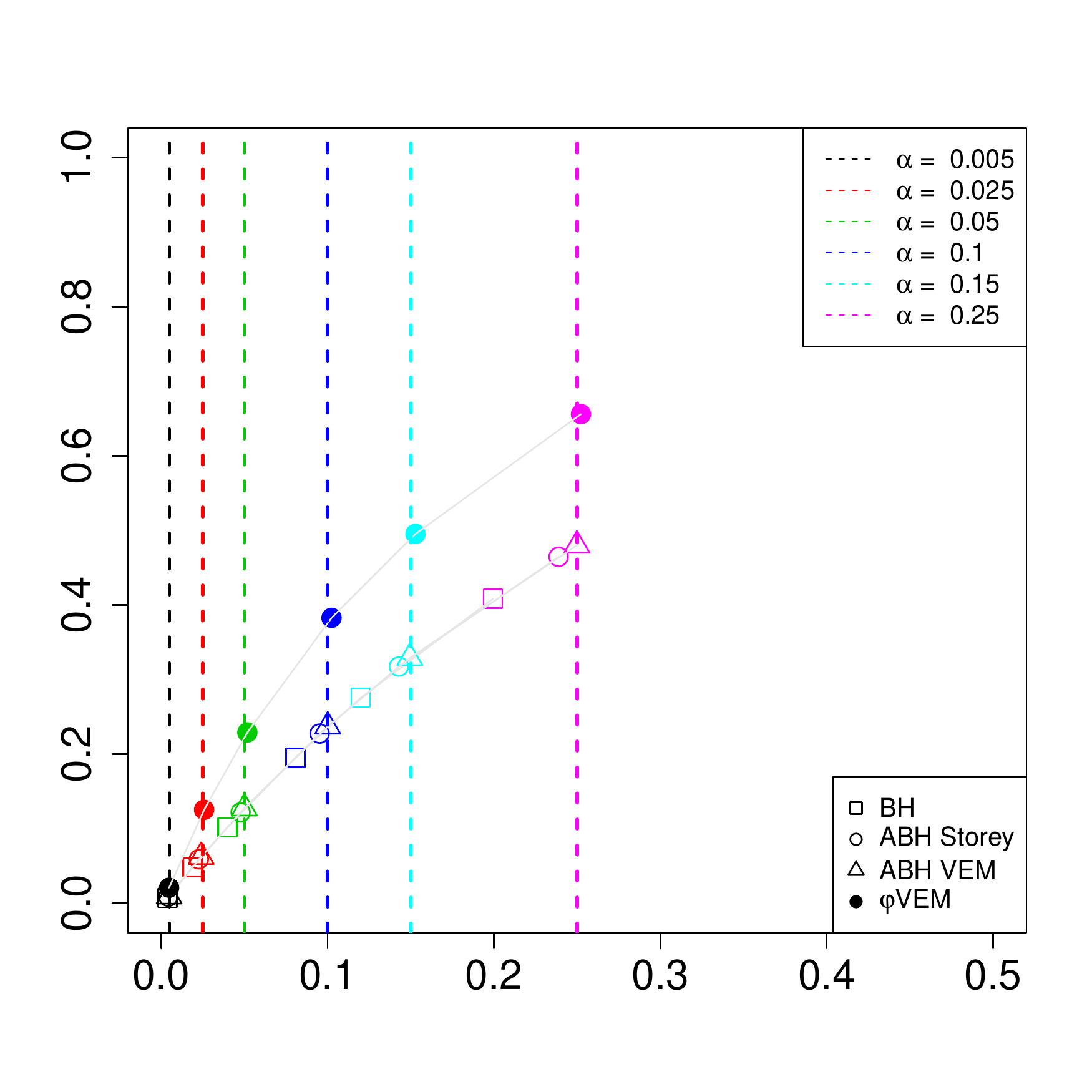}&\includegraphics[scale=0.4]{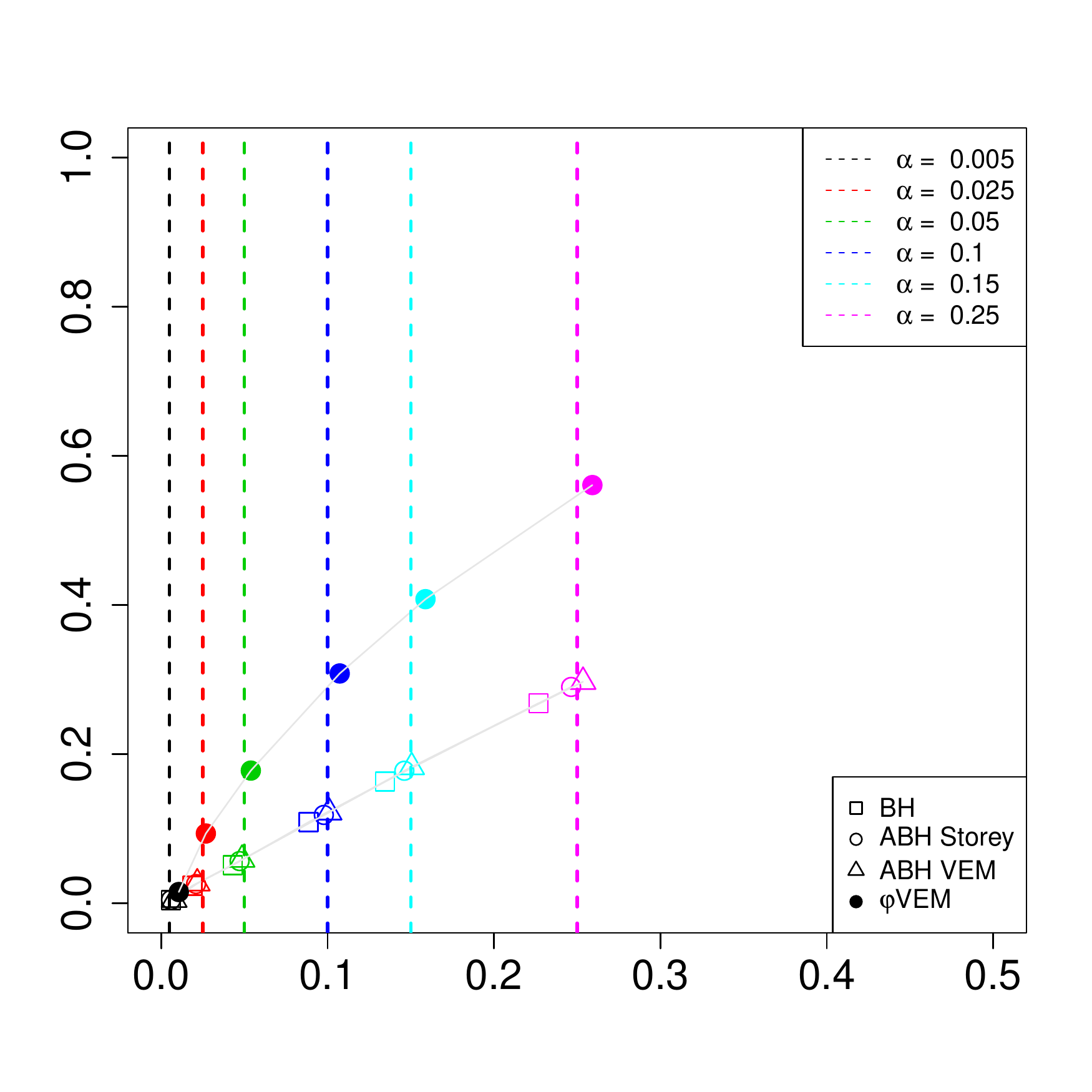}\\
\vspace{-0.5cm}
Bipartite $p=0.5$ &Preferential attachment \\
\vspace{-0.5cm}
\includegraphics[scale=0.4]{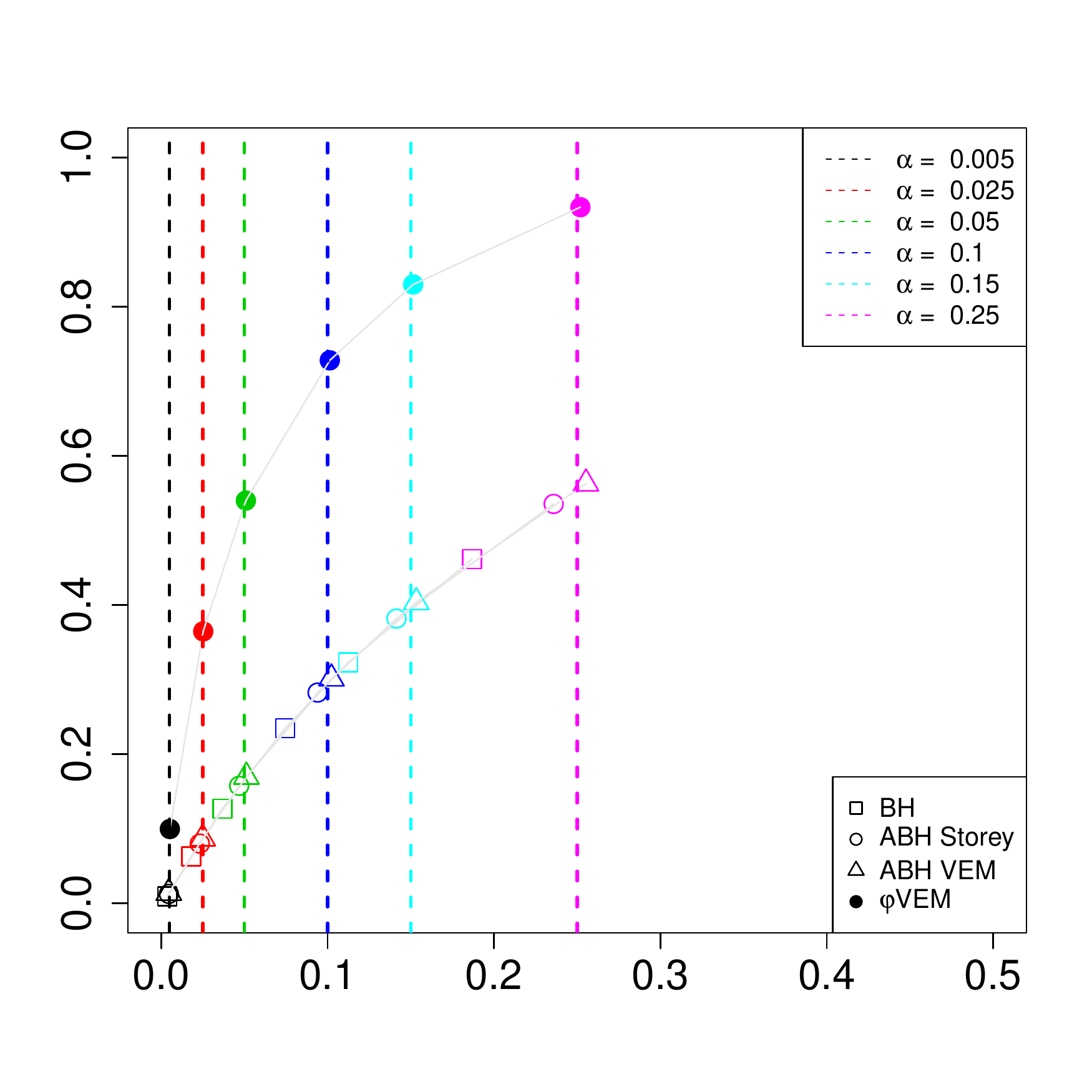}&\includegraphics[scale=0.4]{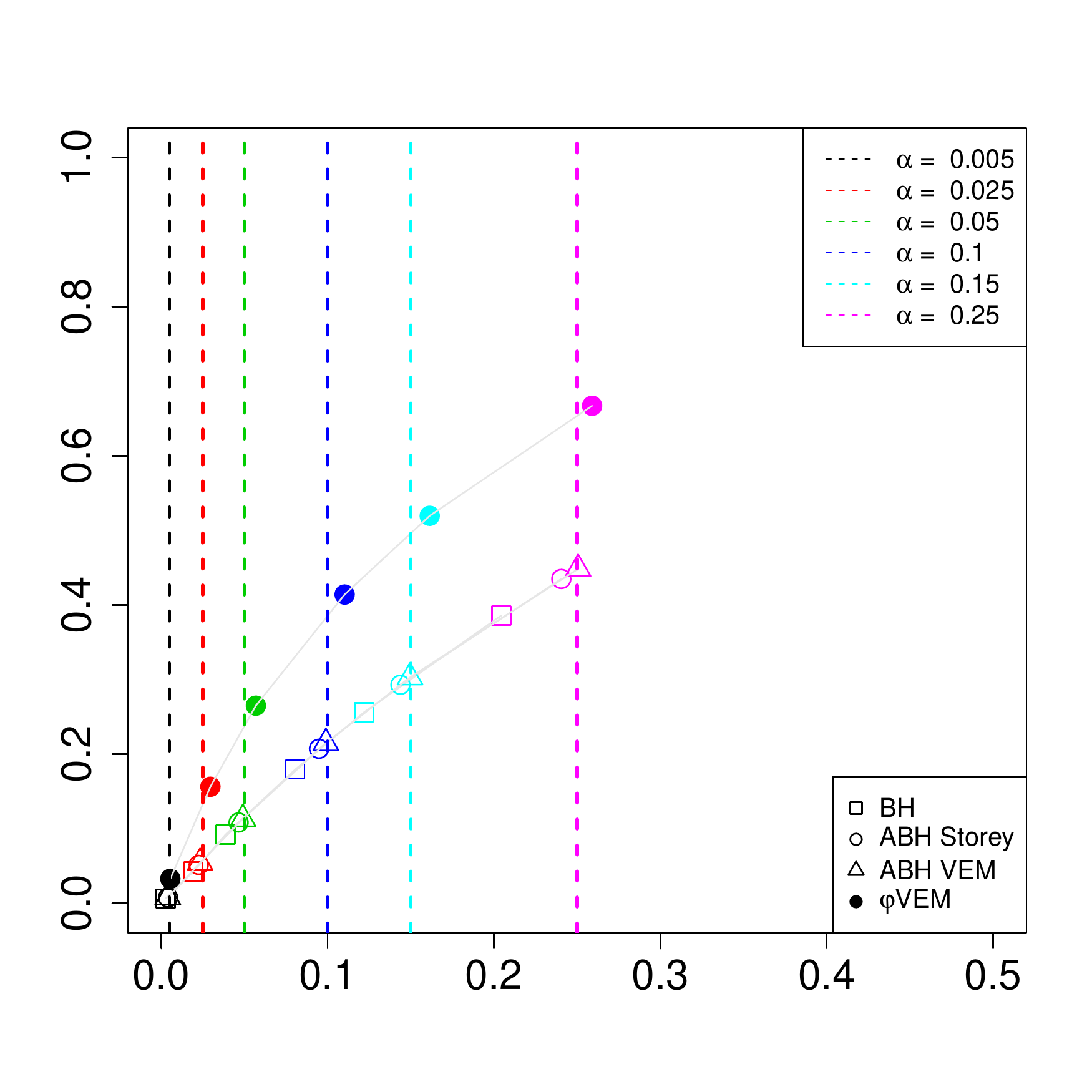}\\
\vspace{-0.5cm}
\end{tabular}
\end{center}
\caption{Same as Figure~\ref{fig:sbm_Q_2_w} in scenario 3.
 \label{fig:scenario3}}
\end{figure}

\section{Discussion}\label{sec:discussion}

The originality of our approach is two-fold. First, we cast the problem of graph inference as the estimation of a "true" latent graph, which can fit a large variety of heterogeneous graph topologies. 
Second, our testing procedure incorporates the learned graph topology to increase the power with respect to standard approaches. That is, the position of the nodes in the entire graph topology is taken into account in the decision to remove or keep edges.

Markedly, the first fold is done by introducing the NSBM with a devoted VEM algorithm, which has an interest on their own in numerous applications where the user may not be interested in graph inference but merely in clustering. Instead of first building a sparse graph by thresholding or a $k$-nearest neighbor approach, before applying some clustering procedure, the input of the VEM algorithm is a dense graph and no troublesome choice of some connectivity parameter, that influences the clustering, is required.

The second fold is done by following a $q$-value-based approach inspired from the standard literature on FDR/TDR in mixture modeling. Our main theoretical results bring a novelty in that area: the FDR/TDR guarantees are non-asymptotic in the size of the graph, with general regularity/complexity conditions on the model, which has an interest in its own right. 

{Let us also mention that there are previous work on building optimal FDR-decision in a group context (see, e.g., \citep{CS2009}), for which a common criticism is as follows: if the groups are previously known, with independent measurements between the groups, why not considering separated FDR control analysis on each group, without combining decisions across the groups? 
Our approach do not encounter such a limitation, because the groups are built on the nodes, while the inference is done on the edges, so our decisions are intrinsically linked and considering separated analysis is in any case not suitable.}

Finally, 
the numerical experiments of Section~\ref{sec:scenario2} suggests that FDR/TDR results can be obtained for a model with a "true", deterministic,  adjacency matrix $A$.
Related to this, it is interesting to adopt a Bayesian point of view on our modeling: the parameter of interest being the adjacency matrix $A$, the SBM can be seen as an a priori distribution on the parameter, while the underlying frequentist model is the one with a deterministic $A$.  In this view, our graph inference $\varphiVEM$ is an empirical Bayes procedure that fits the hyper-parameters $\pi,w$ by a marginal maximum likelihood type approach (via the VEM algorithm). 
Hence, obtaining FDR/TDR results for a model with a fixed adjacency matrix $A$ 
meets the recent literature on frequentist properties of Bayesian procedures (see, e.g., \cite{CR2018} in a multiple testing context) and developing such a methodology in our context is an interesting direction for  future work.

\section{Proofs}\label{sec:proofs}

\subsection{Proofs for Section~\ref{sec:VEMintro}}\label{sec:proofvem}

To solve problem  \eqref{eq:vestep:optim} and prove   Proposition \ref{prop_fixed_point} 
we introduce the function
{
\begin{align}\label{def:critJ}
J(\theta;\tau,\theta')
&=-\tilde\E_{\theta',\tau}[  \log \mathcal{L}(A,Z; \tilde P_{\theta',\tau})\:|\:X] +\tilde\E_{\theta',\tau}[\log \mathcal{L}(X,A,Z;\theta)\:|\:X].
\end{align}
}
Note that
 \eqref{eq:vestep:optim} can be stated as
\begin{align*}
\hat \tau 
 &=\arg\max_{\tau\in\mathcal T} J(\theta;\tau,\theta).
 \end{align*}
 
\begin{lemma}\label{lem_vem_J} 

For any $\theta,\theta'\in\Theta$, $\tau\in \mathcal{T}$, the quantity $J(\theta;\tau,\theta')$ has the following expression
\begin{align*}
J(\theta;\tau,\theta')&=\sum_{q=1}^Q\sum_{i=1}^n \tau_{i,q}\log \frac{\pi_{q}}{\tau_{i,q}}+  \sum_{q=1}^Q\sum_{l=1}^Q\sum_{(i,j)\in\cA}\rho_{q,l}^{i,j}\tau_{i,q}\tau_{j,l} \left\{\log g_{\nu_{q,l}}(X_{i,j}) +\log w_{q,l} -1\right\}\\
&\quad+\sum_{q=1}^Q\sum_{l=1}^Q\sum_{(i,j)\in\cA}(1-\rho_{q,l}^{i,j})\tau_{i,q}\tau_{j,l} \left\{\log g_{0,\nu_0}(X_{i,j})+\log(1- w_{q,l}) \right\}.
 \end{align*}
 where $ \rho_{q,l}^{i,j}=  \rho_{q,l}^{i,j}(\theta')$  is defined by \eqref{eq:def:rho}.
\end{lemma}


\begin{proof}
We have by \eqref{approx:VEM} that
\begin{align*}
J(\theta;\tau,\theta')
&=
-\tilde\E_{\theta',\tau}[ \log \mathcal{L} (A\:|\: Z,X; \theta')\:|\:X] -\tilde\E_{\theta',\tau}[ \log 
\mathcal{L}(Z\:|\:X; \tilde P_{\theta',\tau})
\:|\:X] \\
&+\tilde\E_{\theta',\tau}[\log \mathcal{L}(X,A,Z;\theta)\:|\:X].
\end{align*}
Now, by using \eqref{equvraiscompl}, we have 
\begin{align*}
\tilde\E_{\theta',\tau}[\log \mathcal{L}(X,A,Z;\theta)\:|\:X]
&=
\sum_{(i,j)\in \cA} \tilde \P_{\theta',\tau}\left (A_{i,j}=0  \:\middle|\:X\right)\log g_{0,\nu_0}(X_{i,j})\\
&\quad + \sum_{q=1}^Q\sum_{\l=1}^Q\sum_{(i,j)\in\cA}\tilde\P_{\theta',\tau}\left( A_{i,j}=1, Z_{i,q}Z_{j,\l}=1 \:\middle|\:X\right) \log g_{\nu_{q,\l}}(X_{i,j})\\
&\quad 
+ \sum_{q\le \l}\tilde\E_{\theta',\tau}\left [ M_{q,\l}  \:\middle|\:X\right]\log w_{q,\l} +\tilde\E_{\theta',\tau}\left [ \bar M_{q,\l}  \:\middle|\:X\right]\log (1-w_{q,\l})\\
&\quad + \sum_{q=1}^Q\sum_{i=1}^n \tilde\E_{\theta',\tau}\left [Z_{i,q}\:\middle|\:X\right]\log \pi_{q}.
\end{align*}
This gives
\begin{align*}
J(\theta;\tau,\theta')&=
-\sum_{q=1}^Q\sum_{\l=1}^Q\sum_{(i,j)\in\cA}\rho_{q,\l}^{i,j}\tau_{i,q}\tau_{j,\l}  
+\sum_{q=1}^Q\sum_{i=1}^n \tau_{i,q}\log \frac{\pi_{q}}{\tau_{i,q}}
\\&\quad+
\sum_{(i,j)\in \cA} \log g_{0,\nu_0}(X_{i,j}) \sum_{q=1}^Q\sum_{\l=1}^Q (1-\rho_{q,\l}^{i,j})\tau_{i,q}\tau_{j,\l} \\
&\quad + \sum_{q=1}^Q\sum_{\l=1}^Q\sum_{(i,j)\in\cA}\rho_{q,\l}^{i,j}\tau_{i,q}\tau_{j,\l}  \log g_{\nu_{q,\l}}(X_{i,j})
\\
&\quad 
+\sum_{q=1}^Q\sum_{\l=1}^Q\log w_{q,\l} \sum_{(i,j)\in\cA}\rho_{q,\l}^{i,j}\tau_{i,q}\tau_{j,\l}
+\sum_{q=1}^Q\sum_{\l=1}^Q\log(1- w_{q,\l}) \sum_{(i,j)\in\cA}(1-\rho_{q,\l}^{i,j})\tau_{i,q}\tau_{j,\l}.
\end{align*}
 Rearranging terms yields the result.

\end{proof}

\begin{proof}[Proof of Proposition \ref{prop_fixed_point}] 
From  Lemma \ref{lem_vem_J} we see that
the partial derivative of $J$ with respect to $\tau_{i,q}$ is given by
\begin{align*} 
\frac\partial{\partial\tau_{i,q}} J(\theta;\tau,\theta')
= -\log\tau_{i,q} +\log\pi_q-1 + \sum_{j\neq i}\sum_{\l=1}^Q\tau_{j,\l} d_{q,\l}^{i,j}.
\end{align*} 
And the zero of this derivate satisfies
\begin{align*} 
 \tau_{i,q} =\pi_q\exp\left(
\sum_{j\neq i}\sum_{\l=1}^Q\tau_{j,\l} d_{q,\l}^{i,j}-1\right).
\end{align*} 
Finally, the condition $\sum_{q=1}^Q\tau_{i,q}=1$ yields the result.
\end{proof}

\begin{proof}[Proof of Proposition \ref{prop_mstep_pi_w}]
We see that
\begin{align*} 
\max_{\theta} \tilde\E_{\theta',\tau}\left[\log \mathcal{L}(X,A,Z;\theta)\:|\:X\right]&=
\max_{ \nu_0,\nu} \tilde\E_{\theta',\tau}\left[\log \mathcal{L}(X\:|\:A,Z; \nu_0,\nu)\:|\:X\right]\\
&\quad+ \max_{w}\tilde\E_{\theta',\tau}\left[\log  \mathcal{L}(A \:|\: Z;w)\:|\:X\right]
+ \max_{\pi}\tilde\E_{\theta',\tau}\left[\log \mathcal{L}( Z;\pi)\:|\:X\right].
\end{align*} 
For the term in $\pi$ we have 
\begin{align*} 
\tilde\E_{\theta',\tau}\left[\log \mathcal{L}( Z;\pi)\:|\:X\right]
=\sum_{q=1}^Q \log \pi_q \sum_{i=1}^n\tilde\E_{\theta',\tau}[Z_{i,q}\:|\:X] 
=\sum_{q=1}^Q\log \pi_q \sum_{i=1}^n\tau_{i,q} .
\end{align*} 
Taking into account the condition $\sum_{q=1}^Q\pi_q=1$, we obtain that the maximum is attained at $\hat\pi_q$ given by \eqref{eq:mstep:pi}.
Concerning the optimization with respect to $w$ we have
\begin{align*} 
\tilde\E_{\theta',\tau}\left[\log  \mathcal{L}(A \:|\: Z;w)\:|\:X\right]
&=\sum_{q\le \l}\tilde\E_{\theta',\tau}\left[   M_{q,\l}\:|\:X\right]\log(w_{q,\l})+ \tilde\E_{\theta',\tau}\left[ \bar M_{q,\l}\:|\:X\right]\log(1-w_{q,\l}),
\end{align*} 
which is maximal at
\begin{align*} 
\hat w_{q,\l}= \frac{\tilde\E_{\theta',\tau}\left[   M_{q,\l}\:|\:X\right]}{\tilde\E_{\theta',\tau}\left[ M_{q,\l}+\bar M_{q,\l}\:|\:X\right]}.
\end{align*} 
  Now, for all $q\neq \l$,
\begin{align*} 
\tilde\E_{\theta',\tau}\left[   M_{q,\l}\:|\:X\right]
 &= \sum_{(i,j)\in \cA}  \tilde\P_{\theta',\tau}\left(A_{i,j}=1, Z_{i,q}Z_{j,\l}=1\:|\:X \right) +  \tilde\P_{\theta',\tau}\left(A_{i,j}=1, Z_{i,\l}Z_{j,q}=1\:|\:X \right)\\
 &= \sum_{(i,j)\in \cA}  \P_{\theta'}\left(A_{i,j}=1| Z_{i,q}Z_{j,\l}=1,X \right) \tau_{i,q}\tau_{j,\l}+  \P_{\theta'}\left(A_{i,j}=1| Z_{i,\l}Z_{j,q}=1,X \right)\tau_{i,\l}\tau_{j,q}\\
  &= \sum_{(i,j)\in \cA}  \rho_{q,\l}^{i,j} (\tau_{i,q}\tau_{j,\l}+  \tau_{i,\l}\tau_{j,q})
  =\sum_{(i,j) \in\cA} \kappa_{q,\l}^{i,j}.
  \end{align*} 
Since $M_{q,\l}+\bar   M_{q,\l}=\#\{ (i,j)\in \cA:  Z_{i,q}Z_{j,\l} + Z_{i,\l}Z_{j,q}>0\}$, we obtain
\begin{align*} 
\tilde\E_{\theta',\tau}\left[   M_{q,\l}+\bar   M_{q,\l}\:|\:X\right]
&= \sum_{(i,j)\in \cA} \tilde\P_{\theta',\tau}\left(Z_{i,q}Z_{j,\l}=1\:|\:X \right) +\tilde\P_{\theta',\tau}\left(Z_{i,\l}Z_{j,q}=1\:|\:X \right) \\
& =  \sum_{(i,j)\in \cA} \tau_{i,q}\tau_{j,\l}+  \tau_{i,\l}\tau_{j,q}.
 \end{align*} 
 This yields the solutions given in  \eqref{eq:mstep:wql}--\eqref{eq:mstep:wqq}.  

As for $(\nu_0,\nu)$, we have
\begin{align*}  
\tilde\E_{\theta',\tau}[\log \mathcal{L}(X\:|\:A,Z; \nu_0,\nu)\:|\:X]
&= 
\tilde\E_{\theta',\tau}\left[\sum_{\substack{(i,j)\in \cA\\ A_{i,j}=0}} \log g_{0,\nu_0}(X_{i,j})\:\middle|\:X \right]\nonumber\\
&\quad+\tilde\E_{\theta',\tau}\left[\sum_{q=1}^Q\sum_{\l=1}^Q\sum_{\substack{(i,j): A_{i,j}=1\\ Z_{i,q}Z_{j,\l}=1}} \log(g_{\nu_{q,\l}}(X_{i,j}))\:\middle|\:X\right]
\\
&= 
\sum_{(i,j)\in \cA} \log g_{0,\nu_0}(X_{i,j}) \sum_{q\le \l}  \bar\kappa_{q,\l}^{i,j}+
\sum_{q\le \l} \sum_{(i,j)\in \cA}  \kappa_{q,\l}^{i,j} \log(g_{\nu_{q,\l}}(X_{i,j})),
\end{align*}  
which yields the result.
\end{proof}



\subsection{Proofs for Section~\ref{sec:theory}}\label{sec:proof}\label{proofofFDRcontrol}\label{proofofPowercontrol}\label{sec:proofmaincor}

In the sequel, we denote the $\l$-values $\l_{i,j}(X,Z;\theta)$ \eqref{equ-lvaluescond} by  $\l_{i,j}(\theta)$ for short.
Let us first add some notation, that will be useful in our proof.
Define for any $t\in[0,1]$, $\theta\in\Theta$, 
\begin{align}
\wh{F}_{0}(\theta,t)&=m^{-1}\: \sum_{(i,j)\in\mathcal{A}} (1-A_{i,j}) \ind{\ell_{i,j}(\theta) \leq t};\label{equ:Fhat0}\\
\wh{F}_{1}(\theta,t)&=m^{-1}\: \sum_{(i,j)\in\mathcal{A}} A_{i,j} \ind{\ell_{i,j}(\theta) \leq t};\label{equ:Fhat1}\\
\wh{F}(\theta,t)&=\wh{F}_{0}(\theta,t)+\wh{F}_{1}(\theta,t)=m^{-1}\: \sum_{(i,j)\in\mathcal{A}}  \ind{\ell_{i,j}(\theta) \leq t};\label{equ:Fhat}\\
\FDP(\theta, t) &= 
\frac{\wh{F}_{0}(\theta,t)}{\wh{F}(\theta,t)} \label{equ:FDP}.
 \end{align}
Also let for $\theta' =(\pi',w',\nu'_0,\nu')\in\Theta$,
\begin{align}
F_{0,\theta'}(\theta,t)&=\E_{\theta'} [\wh{F}_{0}(\theta,t)]=
\sum_{q,\l} \pi'_q\pi'_\l (1-w'_{q,\l}) \: \q_0(t,q,\l; \theta',\theta);
\label{equ:F0}\\
F_{1,\theta'}(\theta,t)&=\E_{\theta'} [\wh{F}_{1}(\theta,t)]=
\sum_{q,\l} \pi'_q\pi'_\l w'_{q,\l} \: \q_1(t,q,\l; \theta',\theta);
\label{equ:mF1}\\
F_{\theta'}(\theta,t)&= F_{0,\theta'}(\theta,t)+F_{1,\theta'}(\theta,t).
\label{equ:F}
 \end{align}

Note that $\Q_{\theta'}(\theta,t)$ defined in \eqref{equ:Qcond} is thus such that
\begin{align}
\Q_{\theta'}(\theta,t) &=\frac{F_{0,\theta'}(\theta,t)}{F_{\theta'}(\theta,t)}\label{equ:mFDR}.
 \end{align}

\begin{proof}[Proof of Theorem~\ref{FDRcontrol}]

First observe that  by \eqref{normswitch} and \eqref{varphisimple}, we have
$$
\FDR\left(\theta_0, \varphiVEM\right) \leq 
\E_{\theta_0}\left[\FDP(\hat{\theta}^\sigma,T_{\hat{\theta}^\sigma}(\alpha))  \mathds{1}_{\|\hat{\theta}^\sigma-\theta_0\|_\infty\leq \eps} \right]
+ \P_{\theta_0}( \| (\wh{\theta}, \wh{Z}) - (\theta_0, Z)\|>\eps),
$$
where $\sigma$ denotes any permutation of $\{1,\dots,Q\}$ that minimizes  $\sigma\mapsto  \|\wh{\theta}^\sigma-\theta\|_\infty\vee \|\wh{Z} - Z^\sigma\|_\infty$.
Now, let  $x>0$ with $x<\pimin^2\wedge (1-\wmax)$ and $y>0$ and consider the  event 
$$
\Omega=\left\{\sup_{\substack{\theta\in\Theta, t \in[0,1]\\ {F}_{\theta_0}(\theta,t)\geq y }}\left|\FDP(\theta,t)-\Q_{\theta_0}(\theta,t) \right|\leq x\right\}.
$$
We have
\begin{align*}
\E_{\theta_0}\left[\FDP(\hat{\theta}^\sigma,T_{\hat{\theta}^\sigma}(\alpha))  \mathds{1}_{\|\hat{\theta}^\sigma-\theta_0\|_\infty\leq \eps} \right]
&\leq \E_{\theta_0}\left[\FDP(\hat{\theta}^\sigma,T_{\hat{\theta}^\sigma}(\alpha)) \mathds{1}_{\Omega} \mathds{1}_{\|\hat{\theta}^\sigma-\theta_0\|_\infty\leq \eps} \right]+ \P_{\theta_0}(\Omega^c).
\end{align*}
Now, applying Lemma~\ref{lemqvaluedef} \eqref{majT} (with the definition of $\eta(\eps)$ therein), 
there exists $e=e(\theta_0,\alpha,Q)\in(0,1)$  such that for all $\eps\leq e$,
if $\|\hat{\theta}^\sigma-\theta_0\|_\infty\leq \eps$ then 
$ T_{{\theta_0}}(\min  \mathcal{K}) \leq  T_{{\theta_0}}(\alpha-\eta(\eps)) \leq T_{\hat{\theta}^\sigma}(\alpha)\leq  T_{{\theta_0}}(\alpha+\eta(\eps))\leq T_{{\theta_0}}(\max  \mathcal{K}).
$
In particular, ${F}_{\theta_0}(\hat{\theta}^\sigma,T_{\hat{\theta}^\sigma}(\alpha))\geq {F}_{\theta_0}(\hat{\theta}^\sigma, T_{{\theta_0}}(\min\mathcal{K})) \geq \kappa(1- \eta(\eps)/4) $, by applying Lemma~\ref{lemqvaluedef} \eqref{majF01}. 
Hence, choosing $y\leq \kappa/2\leq\kappa(1- \eta(\eps)/4) $ (which holds by choosing $e$ small enough), we get by definition of $\Omega$,
\begin{align*}
\E_{\theta_0}\left[\FDP(\hat{\theta}^\sigma,T_{\hat{\theta}^\sigma}(\alpha)) \mathds{1}_{\Omega} \mathds{1}_{\|\hat{\theta}^\sigma-\theta_0\|_\infty\leq \eps} \right] &\leq x+\E_{\theta_0}\left[\Q_{\theta_0}(\hat{\theta}^\sigma,T_{\hat{\theta}^\sigma}(\alpha)) \mathds{1}_{\|\hat{\theta}^\sigma-\theta_0\|_\infty\leq \eps} \right] \\
&\leq x+\E_{\theta_0}\left[\Q_{\theta_0}(\hat{\theta}^\sigma,T_{{\theta_0}}(\alpha+\eta(\eps))) \mathds{1}_{\|\hat{\theta}^\sigma-\theta_0\|_\infty\leq \eps} \right] \\
&\leq x+\Q_{\theta_0}(\theta_0,T_{{\theta_0}}(\alpha+\eta(\eps))) + \eta(\eps) = x + \alpha + 2\eta(\eps) ,
\end{align*}
by applying \eqref{majQ}.
This gives
\begin{align*}
\FDR\left(\theta_0, \varphiVEM\right)&\leq  \alpha +x + 2\eta(\eps)
 + \P_{\theta_0}(\Omega^c) +  \P_{\theta_0}( \| (\wh{\theta}, \wh{Z}) - (\theta_0, Z)\|>\eps).
\end{align*}
We conclude by upper bounding $\P_{\theta_0}(\Omega^c)$ according to Lemma~\ref{lem:concFDPprocess}.
\end{proof}

\begin{proof}[Proof of Theorem~\ref{Powercontrol}]

First observe that, similarly to the proof of Theorem~\ref{FDRcontrol} (and using the same notation for the permutation $\sigma$), we have
$$
\pione\TDR\left(\theta_0, \varphiVEM\right) \geq 
\E_{\theta_0}\left[  \wh{F}_{1}(\wh{\theta}^\sigma,T_{\hat{\theta}^\sigma}(\alpha))\mathds{1}_{\| (\wh{\theta}, \wh{Z}) - (\theta_0, Z)\|\leq \eps}\right].
$$

Also observe that for the procedure $\varphi^*$ given by \eqref{def:optimalproc}, we have  
\begin{align*}
\pione \TDR\left(\theta_0, \varphi^*\right) 
&=\E_{\theta_0}\left[  \wh{F}_{1}({\theta_0},T_{{\theta_0}}(\alpha))\right] = {F}_{1,\theta_0}(\theta_0,T_{{\theta_0}}(\alpha)).
\end{align*}
For all $x>0$ with $x<\pimin^2\wedge \wmin$, consider the event
$$
\Omega_1= \left\{\sup_{\theta\in\Theta, t \in[0,1]}\left|\wh{F}_{1}(\theta,t)-{F}_{1,\theta_0}(\theta,t) \right|\leq x\right\}.
$$
We obviously have 
\begin{align*}
\E_{\theta_0}\left[  \wh{F}_{1}(\wh{\theta}^\sigma,T_{\hat{\theta}^\sigma}(\alpha))\mathds{1}_{\| (\wh{\theta}, \wh{Z}) - (\theta_0, Z)\|\leq \eps}\right]&\geq \E_{\theta_0}\left[  \wh{F}_{1}(\wh{\theta}^\sigma,T_{\hat{\theta}^\sigma}(\alpha))\mathds{1}_{\Omega_1}\mathds{1}_{\| (\wh{\theta}, \wh{Z}) - (\theta_0, Z)\|\leq \eps}\right] \\
&\geq  \E_{\theta_0}\left[  {F}_{1,\theta_0}(\wh{\theta}^\sigma,T_{\hat{\theta}^\sigma}(\alpha))\mathds{1}_{\Omega_1}\mathds{1}_{\| (\wh{\theta}, \wh{Z}) - (\theta_0, Z)\|\leq \eps}\right] -x  .
\end{align*}
Now, by applying Lemma~\ref{lemqvaluedef} 
(with the definition of $\eta(\eps)$ therein), there exists  $e=e(\theta_0,\alpha,Q)\in(0,1)$ such that for all $\eps\leq e$,  if $\|\hat{\theta}^\sigma-\theta_0\|_\infty\leq \eps$ then 
$ T_{{\theta_0}}(\min  \mathcal{K}) \leq  T_{{\theta_0}}(\alpha-\eta(\eps)) \leq T_{\hat{\theta}^\sigma}(\alpha)\leq  T_{{\theta_0}}(\alpha+\eta(\eps))\leq T_{{\theta_0}}(\max  \mathcal{K})
$ and
\begin{align*}
&\E_{\theta_0}\left[  {F}_{1,\theta_0}(\wh{\theta}^\sigma,T_{\hat{\theta}^\sigma}(\alpha))\mathds{1}_{\Omega_1}\mathds{1}_{\| (\wh{\theta}, \wh{Z}) - (\theta_0, Z)\|\leq \eps}\right]\\
&\geq    {F}_{1,\theta_0}(\theta_0,T_{{\theta_0}}(\alpha-\eta(\eps)))- \kappa\eta(\eps)/4 -\pione \P_{\theta_0}(\Omega_1^c) - \pione - \P_{\theta_0}( \| (\wh{\theta}, \wh{Z}) - (\theta_0, Z)\|>\eps)   .
\end{align*}
Now using the functions  $\mathcal{W}_{T,\q_1}$ and $\mathcal{W}_{\alpha,T}$ defined by \eqref{equmodulusq1} and  \eqref{equmodulust}, respectively, we have by \eqref{majF1},
\begin{align*}
 {F}_{1,\theta_0}(\theta_0,T_{{\theta_0}}(\alpha-\eta(\eps)))&\geq  {F}_{1,\theta_0}(\theta_0,T_{{\theta_0}}(\alpha)) - 
 \mathcal{W}_{T,\q_1}\left(T_{{\theta_0}}(\alpha)-T_{{\theta_0}}(\alpha-\eta(\eps)\right)\\
 &\geq {F}_{1,\theta_0}(\theta_0,T_{{\theta_0}}(\alpha)) - 
 \mathcal{W}_{T,\q_1}\circ  \mathcal{W}_{\alpha,T} \left(\eta(\eps)\right).
 \end{align*}
Using Lemma~\ref{lem:concF01} to upper-bound $\P_{\theta_0}(\Omega_1^c)$ concludes the proof.
\end{proof}

\begin{proof}[Proof of Corollary~\ref{maincor}]
By using Theorems~\ref{FDRcontrol}~and~\ref{Powercontrol}, and since $\wh{\theta}$, $\wh{Z}$ are consistent, 
we have for all $\eps\in (0,e)$, and $x\in (0,\pimin^2\wedge (1-\wmax))$,
$$
\limsup_n  \FDR\left(\theta_0,\varphiVEM\right) \leq\:  \alpha +x + 16 \kappa^{-1} (\mathcal{W}_{\theta_0,\q}(\eps)+3Q^2\eps)  
$$
and 
\begin{align*}
&\liminf_n \{ \TDR\left(\theta_0, \varphiVEM\right) -\TDR\left(\theta_0, \varphi^*\right)\}\\
&\hspace{2cm}\geq  - \pione^{-1} x-2\mathcal{W}_{\theta_0,\q}(\eps)-6Q^2\eps- \mathcal{W}_{T,\q_1}\circ  \mathcal{W}_{\alpha,T} \left(8 \kappa^{-1} (\mathcal{W}_{\theta_0,\q}(\eps)+3Q^2\eps)\right).
\end{align*}
Taking now $\eps$ and $x$ tending to $0$ gives the result.
\end{proof}

\section*{Acknowledgments} 

We would like to thank Isma\"el Castillo, Antoine Chambaz, Catherine Matias and St\'ephane Robin for interesting discussions.
This work has been supported by the grants
ANR-16-CE40-0019 (SansSouci), ANR-17-CE40-0001 (BASICS) and  ANR-18-CE02-0010-01(EcoNet) of the French National Research Agency ANR.

\bibliographystyle{apalike}
\bibliography{biblio}

\section{Supplementary material}\label{sec:supp}

\subsection{Main lemmas for Section~\ref{sec:theory}}

\begin{lemma} \label{lem:functionQvalue}
 Let  Assumption~\ref{cont} be true and consider any $\theta\in\Theta$ with the corresponding quantities $t_{1,q,\l}(\theta),t_{2,q,\l}(\theta)$, $q,\l\in\{1,\dots,Q\}^2$, $t_1(\theta)=\min_{q,\l} t_{1,q,\l}(\theta)$ and $t_2(\theta)=\max_{q,\l} t_{2,q,\l}(\theta)$.
 Then  the function $t\mapsto \Q_{\theta}(\theta,t)$ is increasing on $[t_1(\theta),t_2(\theta)]$, continuous on $(t_1(\theta),1]$, satisfies $\Q_{\theta}(\theta,t)=0$ for $t\in [0,t_1(\theta)]$, $\Q_{\theta}(\theta,t)=\pizero$  for $t\in [t_2(\theta),1]$ and $\Q_{\theta}(\theta,t)< t$ for $t\in (t_1(\theta),1]$. 
\end{lemma}

\begin{proof}
First note that the following relation holds (coming from \eqref{equ-lvaluescond}, \eqref{equ:Qcond} and Fubini's theorem): for all $\theta\in\Theta$, $t\in[0,1]$,
\begin{equation}
\label{equ:relationfubini}
\Q_\theta(\theta,t)
= \frac{ \E_\theta\left[\sum_{(i,j)\in\mathcal{A}} \ell_{i,j}(\theta)  \ind{\ell_{i,j}(\theta) \leq t} \right]}{\E_\theta\left[\sum_{(i,j)\in\mathcal{A}}  \ind{\ell_{i,j}(\theta) \leq t} \right]}.
\end{equation}
In the sequel, denote respectively $\Q_{\theta}(\theta,t)$ by $f(t)$ and $\ell_{i,j}(\theta)$ by $\ell_{i,j}$  for short.
Note that by Assumption~\ref{cont}, we have $f(t) = 0$ for $t\in[0,t_1]$, $f(t)=\pizero$ for $t\in [t_2,1]$. Notice that, by \eqref{equ:relationfubini} and by Assumption~\ref{cont}, we have
$$
f(t)= \frac{F_{0,\theta}(\theta,t)}{F_{\theta}(\theta,t)} = \frac{ \E_\theta\left[\sum_{(i,j)\in\mathcal{A}} \ell_{i,j} \ind{\ell_{i,j}< t} \right]}{\E_\theta\left[\sum_{(i,j)\in\mathcal{A}}  \ind{\ell_{i,j}< t} \right]}= \frac{ \E_\theta\left[\sum_{(i,j)\in\mathcal{A}} \ell_{i,j} \ind{\ell_{i,j}\leq t} \right]}{\E_\theta\left[\sum_{(i,j)\in\mathcal{A}}  \ind{\ell_{i,j}\leq t} \right]},\:\mbox{ for } t\in[0,1] .
$$
The left-continuity and right-continuity of $f$ in any $t$ such that $F_{\theta}(\theta,t)>0$ is straightforward by the dominated convergence theorem. 

Now prove that $f$ is increasing on $[t_1,t_2].$ For this, let $t_1\leq t< t'\leq t_2$ and prove $f(t')>f(t)$. If $F_{\theta}(\theta,t)=0$, then $f(t)=0$. Since $F_{0,\theta}(\theta,t')>0$ and $F_{\theta}(\theta,t')>0$, we have $f(t')>0=f(t)$. Now assume $F_{\theta}(\theta,t)>0$, so that $F_{\theta}(\theta,t')>0$ also holds. We let 
$$
\delta=\frac{F_{\theta}(\theta,t')}{F_{\theta}(\theta,t)}-1
= \frac{\E_\theta\left[\sum_{(i,j)\in\mathcal{A}}  \ind{t<\ell_{i,j}\leq t'}\right]}{\E_\theta\left[\sum_{(i,j)\in\mathcal{A}}  \ind{\ell_{i,j}\leq t}\right]}.
$$
Now, we have
\begin{align*}
&\left(f(t')-f(t)\right) F_{\theta}(\theta,t') \\
&=\E_\theta\left[\sum_{(i,j)\in\mathcal{A}} \ell_{i,j} \ind{\ell_{i,j}\leq t'} \right] - (1+\delta) \:\E_\theta\left[\sum_{(i,j)\in\mathcal{A}} \ell_{i,j} \ind{\ell_{i,j}\leq t} \right]\\
&=\E_\theta\left[\sum_{(i,j)\in\mathcal{A}} \ell_{i,j} \ind{t<\ell_{i,j}\leq t'} \right]- \delta \:\E_\theta\left[\sum_{(i,j)\in\mathcal{A}} \ell_{i,j} \ind{\ell_{i,j}\leq t} \right]\\
&\geq t\: \E_\theta\left[\sum_{(i,j)\in\mathcal{A}}  \ind{t<\ell_{i,j}\leq t'} \right]  - \delta \:\E_\theta\left[\sum_{(i,j)\in\mathcal{A}} \ell_{i,j} \ind{\ell_{i,j}\leq t} \right]\\
&\geq t \left(\E_\theta\left[\sum_{(i,j)\in\mathcal{A}}  \ind{t<\ell_{i,j}\leq t'} \right] 
- \delta\: \E_\theta\left[\sum_{(i,j)\in\mathcal{A}}\ind{\ell_{i,j}\leq t}) \right] \right)=0.
\end{align*}
Now, since
$
F_{\theta}(\theta,t') >0,
$
this entails 
$f(t')\geq f(t)$. Also, if $f(t')= f(t)$, the inequalities above are all equalities and we have
\begin{align*}
\E_\theta\left[\sum_{(i,j)\in\mathcal{A}} \ell_{i,j} \ind{t<\ell_{i,j}\leq t'} \right]
= t\: \E_\theta\left[\sum_{(i,j)\in\mathcal{A}}  \ind{t<\ell_{i,j}\leq t'} \right]
\end{align*}
and thus
$$
\E_\theta\left[\sum_{(i,j)\in\mathcal{A}} (\ell_{i,j}-t) \ind{t<\ell_{i,j}\leq t'} \right]
=0
$$
which gives $(\ell_{i,j}-t) \ind{t<\ell_{i,j}\leq t'}=0$ $\P_\theta$-a.s. for all $(i,j)$, which is impossible by Assumption~\ref{cont}. Hence, $f(t')> f(t)$ and the increasingness of $f$ is proved.
 
Finally, let $t\in(t_1,1]$ and prove that $t>f(t)$. 
$$
(t-f(t))F_{\theta}(\theta,t)=\E_\theta\left[\sum_{(i,j)\in\mathcal{A}} (t-\ell_{i,j}) \ind{\ell_{i,j}< t} \right] \geq 0
$$
and thus $t\geq f(t)$. Moreover, $t = f(t)$ entails $(t-\ell_{i,j}) \ind{\ell_{i,j}\leq t}=0$ $\P_\theta$-a.s. for all $(i,j)$, and thus 
$\ell_{i,j}\geq t$ $\P_\theta$-a.s. for all $(i,j)$.
This is excluded by Assumption~\ref{cont}. Hence $t> f(t)$. 

\end{proof}

\begin{lemma}[Optimality of $\varphi^*$]\label{lem:optimalproc}
 Let  Assumption~\ref{cont} be true and let $\theta_0\in\Theta$. Consider $\alpha_*$ given by \eqref{def:alphastar}, $\alpha\in(\alpha_*,\pizero)$ and $\varphi^*$ given by \eqref{def:optimalproc}. Then $
\maFDR(\theta_0,\varphi^*) \leq \alpha $ and for any multiple testing procedure $\varphi$ such that $
\maFDR(\theta_0,\varphi) \leq \alpha $, we have $\TDR(\theta_0,\varphi^*)\geq \TDR(\theta_0,\varphi)$.
\end{lemma}

\begin{proof}
We follow an argument inspired from the proof of Theorem~1 in \cite{CSW2019}.
Denote $\ell_{i,j}(\theta_0)$ by $\ell_{i,j}$ for short.
First, note that by Lemma~\ref{lem:functionQvalue}, $\maFDR(\theta_0,\varphi^*)= \Q_{\theta_0}(\theta_0,T_{\theta_0}(\alpha)) = \alpha$ by definition of $T_{\theta_0}(\alpha)\in (t_1(\theta_0),t_2(\theta_0))$, see \eqref{def:optimalthreshold}.
This gives
$$
\E_{\theta_0}\left[ \sum_{(i,j)\in\mathcal{A}}  (\ell_{i,j}-\alpha) \varphi^*_{i,j} \right] = \E_{\theta_0}\left[ \sum_{(i,j)\in\mathcal{A}}  \ell_{i,j} \varphi^*_{i,j}\right] - \alpha \E_{\theta_0}\left[ \sum_{(i,j)\in\mathcal{A}}  \varphi^*_{i,j}\right] = 0.
$$
Also, since $\Psi:x \in [0,1) \mapsto (x-\alpha)/(1-x)$ is continuous increasing and defines a one to one map from $[0,1)$ to $[-\alpha,+\infty)$, we have, almost surely,
$$
\varphi^*_{i,j} = \ind{\Psi(\ell_{i,j})\leq \Psi(T_{\theta_0}(\alpha)) } = \mathds{1}\left\{\ell_{i,j}-\alpha\leq \frac{T_{\theta_0}(\alpha)-\alpha}{1-T_{\theta_0}(\alpha)} (1-\ell_{i,j}) \right\}.
$$
As a result,  it can then be checked that for any procedure $\varphi$,
\begin{equation}\label{equ:SC}
(\ell_{i,j}-\alpha) (\varphi_{i,j}^*-\varphi_{i,j}) \leq \frac{T_{\theta_0}(\alpha)-\alpha}{1-T_{\theta_0}(\alpha)}  (1-\ell_{i,j}) (\varphi_{i,j}^*-\varphi_{i,j}).
\end{equation}
Indeed, this is true if $\varphi_{i,j}^*=1 $ and $\varphi_{i,j}=0$. If $\varphi_{i,j}^*=\varphi_{i,j}$ this obviously holds. If  $\varphi_{i,j}^*=0 $ and $\varphi_{i,j}=1$, then $\ell_{i,j}-\alpha\geq \frac{T_{\theta_0}(\alpha)-\alpha}{1-T_{\theta_0}(\alpha)} (1-\ell_{i,j})$ and the relation is also true.
Hence, 
provided that $\maFDR(\theta_0,\varphi)\leq \alpha$,
\begin{align*}
 \E_{\theta_0}\left[ \sum_{(i,j)\in\mathcal{A}}  (\ell_{i,j}-\alpha) \varphi_{i,j}\right]\leq 0 = \E_{\theta_0}\left[ \sum_{(i,j)\in\mathcal{A}}  (\ell_{i,j}-\alpha) \varphi^*_{i,j}\right]
\end{align*}
This implies
\begin{align*}
0\leq \E_{\theta_0}\left[\sum_{(i,j)\in\mathcal{A}} (\ell_{i,j}-\alpha) (\varphi_{i,j}^*-\varphi_{i,j}) \right]. 
\end{align*}
Hence, by \eqref{equ:SC}, we get
\begin{align*}
0\leq\frac{T_{\theta_0}(\alpha)-\alpha}{1-T_{\theta_0}(\alpha)} \: \E_{\theta_0}\left[\sum_{(i,j)\in\mathcal{A}} (1-\ell_{i,j}) (\varphi_{i,j}^*-\varphi_{i,j})  \right],
\end{align*}
which in turn gives $\TDR({\theta_0},\varphi^*) \geq \TDR({\theta_0},\varphi)$, because $T_{\theta_0}(\alpha)\in (\alpha,1)$.
\end{proof}

\begin{lemma}\label{lemqvaluedef}
Let  Assumption~\ref{cont} be true, $\theta_0\in\Theta$, 
$\alpha_*=\alpha_*(\theta_0)$ given by \eqref{def:alphastar}, $\alpha\in \mathcal{K} \subset (\alpha_*,\pizero)$ for some compact interval $\mathcal{K}$,  $\kappa=\kappa(\theta_0,\alpha)$ given by \eqref{equ:kappa}, the modulus $\mathcal{W}_{\theta_0,\q}$ defined by \eqref{equmodulusq} and the modulus $\mathcal{W}_{T,\q_1}$ defined by \eqref{equmodulusq1}.
Then there exists $e=e(\theta_0,\alpha,Q)\in(0,1)$ such that for all $\eps\leq e$, letting $\eta(\eps)=8 \kappa^{-1} (\mathcal{W}_{\theta_0,\q}(\eps)+3Q^2\eps)$, we have $\min \mathcal{K} \leq \alpha-\eta(\eps)$ and $\alpha+\eta(\eps) \leq \max \mathcal{K}$ and for any $\theta,\theta'\in\Theta$ with $\|{\theta}-\theta_0\|_\infty\leq \eps$ and $\|{\theta'}-\theta_0\|_\infty\leq \eps$,  we have 
\begin{align}
& \sup_{t\in T_{\theta_0}(\mathcal K)}\left|F_{0,\theta'}(\theta,t)-F_{0,\theta_0}(\theta_0,t) \right|\vee \left|F_{1,\theta'}(\theta,t)-F_{1,\theta_0}(\theta_0,t) \right|\vee\left|F_{\theta'}(\theta,t)-F_{\theta_0}(\theta_0,t) \right| \nonumber\\
&\hspace{5cm}\leq 2\mathcal{W}_{\theta_0,\q}(\eps)+6Q^2\eps;\label{majF01}\\
& \sup_{t\in T_{\theta_0}(\mathcal K)}\left|\Q_{\theta'}(\theta,t)-\Q_{\theta_0}(\theta_0,t) \right|\leq \eta(\eps);\label{majQ}\\
&T_{\theta_0}(\min \mathcal{K}) \leq  T_{\theta_0}(\alpha-\eta(\eps)) \leq T_{\theta}(\alpha)\leq  T_{\theta_0}(\alpha+\eta(\eps))\leq T_{\theta_0}(\max \mathcal{K})\label{majT}.
\end{align}
Moreover, for $t\in T_{\theta_0}(\mathcal{K})$ with $|t-T_{\theta_0}(\alpha)|\leq \eps$, we have
\begin{align}
 \left|F_{1,\theta_0}(\theta_0,t)-F_{1,\theta_0}(\theta_0, T_{\theta_0}(\alpha)) \right| \leq \mathcal{W}_{T,\q_1}(\eps)\label{majF1}.
\end{align}
\end{lemma}

\begin{proof}
To prove \eqref{majF01}, we have by \eqref{equmodulusq} and \eqref{equ:mF1}, for all $t\in T_{\theta_0}(\mathcal K)$,
\begin{align*}
&\left|F_{1,\theta'}(\theta,t)-F_{1,\theta_0}(\theta_0,t) \right|\nonumber\\
&=\left| \sum_{q,\l} \pi_q(\theta') \pi_\l(\theta') w_{q,\l}(\theta') \: \q_1(t,q,\l; \theta',\theta) - 
\sum_{q,\l} \pi_q(\theta_0) \pi_\l(\theta_0) w_{q,\l}(\theta_0) \: \q_1(t,q,\l; \theta_0,\theta_0) \right|\\
&\leq  \sum_{q,\l} \pi_q(\theta') \pi_\l(\theta') w_{q,\l}(\theta') \: \left| \q_1(t,q,\l; \theta',\theta) - \frac{ \pi_q(\theta_0) \pi_\l(\theta_0) w_{q,\l}(\theta_0)}{\pi_q(\theta') \pi_\l(\theta') w_{q,\l}(\theta')}
 \: \q_1(t,q,\l; \theta_0,\theta_0) \right|\\
 &\leq \sup_{q,\l} \left| \q_1(t,q,\l; \theta',\theta) - \q_1(t,q,\l; \theta_0,\theta_0) \right| +  \sum_{q,\l} \left|\pi_q(\theta') \pi_\l(\theta') w_{q,\l}(\theta') \: -  \pi_q(\theta_0) \pi_\l(\theta_0) w_{q,\l}(\theta_0) \right|\\
 &\leq \mathcal{W}_{\theta_0,\q}(\eps)+3Q^2\eps.
\end{align*}
Similarly, the latter bound is also valid for $\left|F_{1,\theta'}(\theta,t)-F_{1,\theta_0}(\theta_0,t) \right|$. Since $F_{\theta'}(\theta,t)=F_{0,\theta'}(\theta,t)+F_{1,\theta'}(\theta,t)$, this proves \eqref{majF01}. The proof of \eqref{majF1} follows similarly.

Let us now establish \eqref{majQ}. First, by \eqref{majF01}, and since ${F}_{\theta_0}(\theta_0,\min T_{\theta_0}(\mathcal K))=\kappa(\theta_0,\alpha)>0$, we have for all $t\in T_{\theta_0}(\mathcal K)$, $F_{\theta'}(\theta,t)\geq \kappa/2$ by choosing $e=e(\theta_0,\alpha)$ small enough. Hence,
\begin{align*}
\left|\Q_{\theta'}(\theta,t)-\Q_{\theta_0}(\theta_0,t) \right|&= \left| \frac{{F}_{0,\theta'}(\theta,t)}{F_{\theta'}(\theta,t)}-\frac{{F}_{0,\theta_0}(\theta_0,t)}{F_{\theta_0}(\theta_0,t)}\right|\\
&\leq \left| \frac{{F}_{0,\theta'}(\theta,t)}{F_{\theta'}(\theta,t)}- \frac{{F}_{0,\theta_0}(\theta_0,t)}{F_{\theta'}(\theta,t)}\right|+{F}_{0,\theta_0}(\theta_0,t)  \left|\frac{1}{F_{\theta'}(\theta,t)} -\frac{1}{F_{\theta_0}(\theta_0,t)}\right|\\
&\leq \frac{ \left|{F}_{0,\theta'}(\theta,t)-{F}_{0,\theta_0}(\theta_0,t)\right|}{\kappa/2}+ \frac{{F}_{0,\theta_0}(\theta_0,t)}{{F}_{\theta_0}(\theta_0,t)} \frac{\left|{F}_{\theta'}(\theta,t)-{F}_{\theta_0}(\theta_0,t)\right|}{\kappa/2}\\
&\leq 4 \kappa^{-1} (2\mathcal{W}_{\theta_0,\q}(\eps)+6Q^2\eps),
\end{align*}
which proves \eqref{majQ} by using again  \eqref{majF01}.

Let us finally prove \eqref{majT}. The relation \eqref{majQ}, used with $\theta=\theta'=\theta_0$ gives   for all $t\in T_{\theta_0}(\mathcal K)$,
\begin{align}\label{equ-intermlemqvalue}
\Q_{\theta_0}(\theta_0,t)-\eta(\eps) \leq \Q_{{\theta}}({\theta}, t) \leq \Q_{\theta_0}(\theta_0,t)+\eta(\eps) .
\end{align}
Furthermore, applying \eqref{equ-intermlemqvalue} for $t=T_{\theta_0}(\alpha+\eta(\eps))$ and $t=T_{\theta_0}(\alpha-\eta(\eps))$ yields
$$
\alpha\leq \Q_{{\theta}}({\theta}, T_{\theta_0}(\alpha+\eta(\eps)))\mbox{ and }  \Q_{{\theta}}({\theta}, T_{\theta_0}(\alpha-\eta(\eps)))\leq \alpha,
$$
which gives the result by definition of the pseudo-inverse \eqref{def:optimalthreshold}. 
\end{proof}

\begin{lemma}[Concentration of the FDP process]\label{lem:concFDPprocess}
There exists universal constants $c_1,c_1',c_2,c_2'>0$ such that the following holds. Let Assumption~\ref{AssumptionIk} be true for some $K>0$.
Let $\theta_0=(\pi,w,\nu_0,\nu)\in\Theta$, $\pimin=\min_{q}\{\pi_q(\theta_0)\}$, $\wmin=\min_{q,\l} \{w_{q,\l}(\theta_0)\}$ and $\wmax=\max_{q,\l} \{w_{q,\l}(\theta_0)\}.$
Then for all $x>0$ with $x<\pimin^2\wedge (1-\wmax)$ and $y>0$,
\begin{align*}
&\P_{\theta_0}\left(\sup_{\substack{\theta\in\Theta, t \in[0,1]\\ {F}_{\theta_0}(\theta,t)\geq y }}\left|\FDP(\theta,t)-\Q_{\theta_0}(\theta,t) \right|> x \right)\\&\leq  c_1 Q^2 \e^{-c_1'\lfloor n/2\rfloor y^2 x^2 / Q^4} + c_2 K  Q^2 \e^{- c_2' m \pimin^2 (1-\wmax) y^2 x^2/K^2}.
\end{align*}
\end{lemma}

\begin{proof}
This is a direct application of Lemma~\ref{lem:concF01} (used with $x$ replaced by $xy/4$) and of the relation
\begin{align*}
&\left|\frac{\wh{F}_0(\theta,t) }{ \wh{F}(\theta,t)}-  \frac{{F}_{0,\theta_0}(\theta,t) }{{F}_{\theta_0}(\theta,t) } \right|\\
&\leq \left|\frac{\wh{F}_0(\theta,t) - {F}_{0,\theta_0}(\theta,t) }{ \wh{F}(\theta,t)} \right| + \frac{{F}_{0,\theta_0}(\theta,t) }{\wh{F}(\theta,t) {F}_{\theta_0}(\theta,t)} |\wh{F}(\theta,t)-{F}_{\theta_0}(\theta,t)|\\
&\leq \left|\frac{\wh{F}_0(\theta,t) - {F}_{0,\theta_0}(\theta,t) }{ \wh{F}(\theta,t)} \right| + \frac{1}{\wh{F}(\theta,t) } |\wh{F}(\theta,t)-{F}_{\theta_0}(\theta,t)|\\
&\leq \frac{2}{{F}_{\theta_0}(\theta,t)}\left( \left|\wh{F}_0(\theta,t) - {F}_{0,\theta_0}(\theta,t) \right| + \left|\wh{F}(\theta,t)-{F}_{\theta_0}(\theta,t)\right|\right)\\
&\leq \frac{4}{y}\left( \left|\wh{F}_0(\theta,t) - {F}_{0,\theta_0}(\theta,t) \right| + \left|\wh{F}(\theta,t)-{F}_{\theta_0}(\theta,t)\right|\right),
\end{align*}
which holds provided that ${F}_{\theta_0}(\theta,t)/2\geq y/2$ and $\wh{F}(\theta,t)\geq {F}_{\theta_0}(\theta,t)/2$. The first relation comes by assumption, the second relation comes from the concentration of $\wh{F}(\theta,t)$, because $xy/4\leq {F}_{\theta_0}(\theta,t)/2$.
\end{proof}

\subsection{Auxiliary results}

\begin{lemma}\label{lem:concF01}
In the setting of Lemma~\ref{lem:concFDPprocess}, 
for all $x>0$ with $x<\pimin^2\wedge (1-\wmax)$,
\begin{align*}
&\P_{\theta_0}\left(\sup_{\theta\in\Theta, t \in[0,1]}\left|\wh{F}_{0}(\theta,t)-{F}_{0,\theta_0}(\theta,t) \right|> x \right)\\
&\hspace{2cm}\leq  2 Q^2 \e^{-2\lfloor n/2\rfloor x^2 /(9 Q^4)} + 6 K  Q^2 \e^{- m \pimin^2 (1-\wmax) x^2/(9K^2)}.
\end{align*}
For all $x>0$ with $x<\pimin^2\wedge \wmin$,
\begin{align*}
\P_{\theta_0}\left(\sup_{\theta\in\Theta, t \in[0,1]}\left|\wh{F}_{1}(\theta,t)-{F}_{1,\theta_0}(\theta,t) \right|> x \right)&\leq  2 Q^2 \e^{-2\lfloor n/2\rfloor x^2 /(9 Q^4)} + 6 K  Q^2 \e^{- m \pimin^2 \wmin x^2/(9K^2)}.
\end{align*}
Finally, for all $x>0$ with $x<\pimin^2$,
\begin{align*}
\P_{\theta_0}\left(\sup_{\theta\in\Theta, t \in[0,1]}\left|\wh{F}(\theta,t)-{F}_{\theta_0}(\theta,t) \right|> x \right)&\leq  2 Q^2 \e^{-2\lfloor n/2\rfloor x^2 /(9 Q^4)} + 4 K  Q^2  \e^{- m \pimin^2 x^2/4}.
\end{align*}

\end{lemma}

\begin{proof}
Let us denote 
\begin{align*}
m_{q,\l}(A,Z)&=\sum_{(i,j)\in\mathcal{A}}  \mathds{1}\{Z_i=q,Z_j=\l\}\\
m_{0,q,\l}(A,Z)&=\sum_{(i,j)\in\mathcal{A}}  \mathds{1}\{Z_i=q,Z_j=\l\}(1-A_{i,j}) ,
\end{align*}
 for $q,\l\in\{1,\dots,Q\}$. For all $\theta\in\Theta$ and $t\in[0,1]$, we have
\begin{align*}
\wh{F}_{0}(\theta,t)&=m^{-1}\: \sum_{(i,j)\in\mathcal{A}} (1-A_{i,j})  \mathds{1}\{\ell_{i,j}(\theta) \leq t\}=  \sum_{1\leq q,\l\leq Q}\frac{m_{0,q,\l}(A,Z)}{m}\wh{F}_{0,q,\l}(\theta,t),
\end{align*}
where we let
\begin{align*}
\wh{F}_{0,q,\l}(\theta,t)&= (m_{0,q,\l}(A,Z))^{-1}\: \sum_{(i,j)\in\mathcal{A}} \mathds{1}\{Z_i=q,Z_j=\l\} (1-A_{i,j})  \mathds{1}\{\bell(X_{i,j},q,\ell;\theta)  \leq t\},
\end{align*}
Note that  
$$
\E_{\theta_0} (\wh{F}_{0,q,\l}(\theta,t)\:|\:A,Z) = {F}_{0,\theta_0,q,\l}(\theta,t)
,
$$
where we denote ${F}_{0,\theta_0,q,\l}(\theta,t)=\P_{\theta_0}(\bell(X_{i,j},q,\ell;\theta)  \leq t\:|\: Z_i=q,Z_j=\l,A_{i,j}=0)$.
As a consequence, we have
\begin{align*}
&|\wh{F}_{0}(\theta,t)-{F}_{0,\theta_0}(\theta,t)|\\
&=\left| \sum_{1\leq q,\l\leq Q}\frac{m_{0,q,\l}(A,Z)}{m} \wh{F}_{0,q,\l}(\theta,t) - \sum_{1\leq q,\l\leq Q} \pi_q\pi_\l (1-w_{q,\l}){F}_{0,\theta_0,q,\l}(\theta,t) \right|\\
&\leq\sum_{1\leq q,\l\leq Q} \left| \frac{m_{0,q,\l}(A,Z)}{m} \wh{F}_{0,q,\l}(\theta,t) -  \pi_q\pi_\l (1-w_{q,\l}){F}_{0,\theta_0,q,\l}(\theta,t) \right|\\
&\leq\sum_{1\leq q,\l\leq Q} \left| \frac{m_{0,q,\l}(A,Z)}{m} - \pi_q\pi_\l (1-w_{q,\l}) \right|  
+\sum_{1\leq q,\l\leq Q}  \pi_q\pi_\l (1-w_{q,\l})\left|  \wh{F}_{0,q,\l}(\theta,t)-  {F}_{0,\theta_0,q,\l}(\theta,t) \right|\\
&\leq \sum_{1\leq q,\l\leq Q}  \frac{m_{q,\l}(Z)}{m} \left| \frac{m_{0,q,\l}(A,Z)}{m_{q,\l}(Z)} - (1-w_{q,\l})  \right| \\
&+\sum_{1\leq q,\l\leq Q} (1-w_{q,\l}) \left| \frac{m_{q,\l}(Z)}{m} -  \pi_q\pi_\l  \right|\\  
&+\sum_{1\leq q,\l\leq Q}  \pi_q\pi_\l (1-w_{q,\l})\left|  \wh{F}_{0,q,\l}(\theta,t)-  {F}_{0,\theta_0,q,\l}(\theta,t) \right|.
\end{align*}
The latter is smaller than or equal to $x$ on the  event
\begin{align*}
\Omega&=\left\{\forall q,\l\in \{1,\dots,Q\}\::\: \left| \frac{m_{0,q,\l}(A,Z)}{m_{q,\l}(Z)} - (1-w_{q,\l})  \right| \leq x/3,\left| \frac{m_{q,\l}(Z)}{m} -  \pi_q\pi_\l  \right|\leq x/(3Q^2), \right.\\
&\left. \sup_{\theta\in\Theta, t \in[0,1]}\left|\wh{F}_{0,q,\l}(\theta,t)-{F}_{0,\theta_0,q,\l}(\theta,t)\right| \leq x/3\right\}.
\end{align*}
Let us now provide an upper-bound for $\P_{\theta_0}(\Omega^c)$. 
We have
\begin{align*}
\P_{\theta_0}(\Omega^c)\leq &\:\sum_{q,\l=1}^Q \P_{\theta_0}\left(\left| \frac{m_{0,q,\l}(A,Z)}{m_{q,\l}(Z)} - (1-w_{q,\l})  \right| > x/3, \left| \frac{m_{q,\l}(Z)}{m} -  \pi_q\pi_\l  \right|\leq x/(3Q^2)\right) \\
&+ \sum_{q,\l=1}^Q\P_{\theta_0}\left(\left| \frac{m_{q,\l}(Z)}{m} -  \pi_q\pi_\l  \right|> x/(3Q^2)\right)\\
&+ \sum_{q,\l=1}^Q\P_{\theta_0}\left(\sup_{\theta\in\Theta, t \in[0,1]}\left|\wh{F}_{0,q,\l}(\theta,t)-{F}_{0,\theta_0,q,\l}(\theta,t)\right| > x/3,\left| \frac{m_{0,q,\l}(A,Z)}{m_{q,\l}(Z)} - (1-w_{q,\l})  \right| \leq x/3, \right.\\
&\hspace{4cm}\left.\left| \frac{m_{q,\l}(Z)}{m} -  \pi_q\pi_\l  \right|\leq x/(3Q^2)\right).\\
=&\:(I) + (II) + (III).
\end{align*}
To bound (I), we note that, conditionally on $Z$, $m_{0,q,\l}(A,Z)$ is the sum of $m_{q,\l}(Z)$ i.i.d. $\mathcal{B}(1-w_{q,\l})$, which gives by applying (i) of Lemma~\ref{lem:binom} ($p=1-w_{q,\l}$, $n=m_{q,\l}(Z)$), that
\begin{align*}
(I)&\leq \sum_{q,\l=1}^Q \E_{\theta_0}\left( 2 \e^{-  2 m_{q,\l}(Z) (x/3)^2} \mathds{1}\left\{\left| \frac{m_{q,\l}(Z)}{m} -  \pi_q\pi_\l  \right|\leq x/(3Q^2)\right\}\right)\\
& \leq 2 Q^2 \e^{-  m(2/9)  \left(\pi_q\pi_\l - x/(3Q^2)\right)_+ x^2}.
\end{align*}
For bounding (II), we use readily (ii) of Lemma~\ref{lem:binom} to obtain
\begin{align*}
(II)&\leq 2 Q^2 \e^{-2\lfloor n/2\rfloor x^2 /(9 Q^4)} .
\end{align*}
For bounding (III), note that
\begin{align*}
\wh{F}_{0,q,\l}(\theta,t)&= (m_{0,q,\l}(A,Z))^{-1}\: \sum_{(i,j)\in\mathcal{A}} \mathds{1}\{Z_i=q,Z_j=\l\} (1-A_{i,j})  \mathds{1}\{\bell(X_{i,j},q,\ell;\theta)  \leq t\}\\
&= (m_{0,q,\l}(A,Z))^{-1}\: \sum_{(i,j)\in\mathcal{A}} \mathds{1}\{Z_i=q,Z_j=\l\} (1-A_{i,j})  \mathds{1}\{X_{i,j}\in I(q,\l,\theta,t)\},
\end{align*}
for some $I(q,\l,\theta,t) \in \mathcal{I}_K$, by using Assumption~\ref{AssumptionIk}. 
Note that, conditionally on $A,Z$, 
 the variables of $(X_{i,j}, (i,j)\in\mathcal{A}, A_{i,j}=0,Z_j=q,Z_j=\l)$ are i.i.d.
Hence, we can apply Lemma~\ref{lem:donskerI} ($n=m_{0,q,\l}(A,Z)$), to get for all $x>0$,
\begin{align*}
(III) &\leq  Q^2\E_{\theta_0}\left( 4 K \e^{- m_{0,q,\l}(A,Z) x^2/(2K^2)} \mathds{1}\left\{ m_{0,q,\l}(A,Z) \geq  m (\pi_q\pi_\l -x/(3Q^2) ) ((1-w_{q,\l}) - x/3)\right\}\right)\\
&\leq  4 K  Q^2 \e^{- m (\pi_q\pi_\l -x/(3Q^2) )_+ ((1-w_{q,\l}) - x/3)_+ x^2/(2K^2)} .
\end{align*}
Finally, note that 
$\pi_q\pi_\l -x/3\geq \pimin^2/2$ provided that $x\leq 3\pimin^2/2$
and $(1-w_{q,\l}) - x/3\geq (1-\wmax)/2$, provided that $x\leq 3  (1-\wmax)/2$, so that $\P_{\theta_0}(\Omega^c)$ is smaller than
$ 2 Q^2 \e^{-  m  \pimin^2 x^2/9} + 2 Q^2 \e^{-2\lfloor n/2\rfloor x^2 /(9 Q^4)} + 4 K  Q^2 \e^{- m \pimin^2 (1-\wmax) x^2/(8K^2)}$. This concludes the proof of the first inequality. The second inequality is similar, by replacing $1-w_{q,\l}$ (resp. $1-A_{i,j}$) by $w_{q,\l}$ (resp. $A_{i,j}$). The third inequality is obtained similarly.
\end{proof}

\begin{lemma}\label{lem:donskerI}
Let $U_1,\dots,U_n$ be $n$ i.i.d. continuous real random variables and $K$ a positive integer. Then we have for all $x>0$,
\begin{align*}
\P\left(\sup_{I \in \mathcal{I}_K}\left|n^{-1} \sum_{i=1}^n \mathds{1}\{U_i \in I\} - \P(U_1 \in I)\right|\geq x\right)\leq 4 K \e^{- n x^2/(2K^2)},
\end{align*}
where $\mathcal{I}_K$ is defined by \eqref{equIk}.
\end{lemma}

\begin{proof} 
Denote 
$$\mathcal{C}=\{ (a_k, b_k)_{1\leq k\leq K}\mbox{ such that }  -\infty\leq a_k\leq b_k\leq +\infty \mbox{ for } 1\leq k \leq K,  \mbox{ and } b_{k}\leq a_{k+1}  \mbox{ for } 1\leq k\leq K-1 \}.$$
For all $y>0$, we have
\begin{align*}
&\P\left(\sup_{I \in \mathcal{A}}\left|n^{-1} \sum_{i=1}^n \mathds{1}\{U_i \in I\} - \P(U_1 \in I)\right|\geq 2 K y\right)\\
&\leq \P\left(\sup_{(a_k, b_k)_{k} \in \mathcal{C}} \sum_{k=1}^K \left|n^{-1} \sum_{i=1}^n \mathds{1}\{a_k <U_i<b_k \} - \P(a_k<U_1<b_k)\right|\geq 2 K y\right)\\
&\leq \P\left(\sup_{(a_k, b_k)_{k} \in \mathcal{C}} \sum_{k=1}^K \left|n^{-1} \sum_{i=1}^n \mathds{1}\{U_i<b_k \} - \P(U_1<b_k)\right|\geq K y\right)\\
&+ \P\left(\sup_{(a_k, b_k)_{k} \in \mathcal{C}} \sum_{k=1}^K \left|n^{-1} \sum_{i=1}^n \mathds{1}\{U_i\leq a_k \} - \P( U_1\leq a_k)\right|\geq K y\right)
\end{align*}
because $ \mathds{1}\{a_k <U_i<b_k \}= \mathds{1}\{U_i<b_k \}- \mathds{1}\{U_i\leq a_k \}$.
Now,  we have
\begin{align*}
&\P\left(\sup_{(a_k, b_k)_{k} \in \mathcal{C}} \sum_{k=1}^K \left|n^{-1} \sum_{i=1}^n \mathds{1}\{U_i\leq a_k \} - \P( U_1\leq a_k)\right|\geq K y\right)\\
&\leq \P\left( \sum_{k=1}^K \sup_{t\in\R }\left|n^{-1} \sum_{i=1}^n \mathds{1}\{U_i\leq t \} - \P( U_1\leq t)\right|\geq K y\right)\\
&\leq \sum_{k=1}^K  \P\left( \sup_{t\in\R }\left|n^{-1} \sum_{i=1}^n \mathds{1}\{U_i\leq t \} - \P( U_1\leq t)\right|\geq y\right)\\
&\leq 2 K \e^{-2 n y^2},
\end{align*}
by using DKW inequality with Massart's constant, see \cite{Mass1990}. Since, almost surely,  $\forall i\in\{1,\dots,n\},$ $\forall t\in\Q$, $U_i\neq t$, we have almost surely,
\begin{align*}
\sup_{t\in\R }\left|n^{-1} \sum_{i=1}^n \mathds{1}\{U_i\leq t \} - \P( U_1\leq t)\right| &= \sup_{t\in\Q }\left|n^{-1} \sum_{i=1}^n \mathds{1}\{U_i\leq t \} - \P( U_1\leq t)\right|\\
&=\sup_{t\in\Q }\left|n^{-1} \sum_{i=1}^n \mathds{1}\{U_i< t \} - \P( U_1< t)\right|\\
&=\sup_{t\in\R }\left|n^{-1} \sum_{i=1}^n \mathds{1}\{U_i< t \} - \P( U_1< t)\right|,
\end{align*}
and the proof is finished.
\end{proof}

\begin{lemma}\label{lem:binom}
Let $n\geq 2$ be an integer. Then
\begin{itemize}
\item[(i)] For $Y\sim \mathcal{B}(n,p)$,  $p\in (0,1)$, we have for all $x>0$,
\begin{align*}
\P\left(|Y/n -p|\geq x\right)&\leq 2 \e^{-  2 n x^2}.
\end{align*}
\item[(ii)] For $Z_i$ for $1\leq i\leq n$ i.i.d.  where $\pi_q=\P(Z_1=q)\in (0,1)$, $q=1,\dots,Q$, $\sum_{q=1}^Q \pi_q=1$, we have  for all $x>0$,
\begin{align*}
\P\left(\left|m^{-1}\sum_{1\leq i< j\leq n} \mathds{1}\{Z_i=q,Z_j=\l\} -  \pi_q\pi_\l\right|\geq x\right)&\leq 2\e^{-2 \lfloor n/2\rfloor x^2}
. 
\end{align*}
\end{itemize}
\end{lemma}

\begin{proof} 
Both inequalities are applications of versions of Hoeffding inequalities : (i) is the classical version, while (ii) is the one devoted to  $U$-statistics, see e.g. \cite{Pitcan2017}.

\end{proof}

\subsection{Gaussian model computations}\label{sec:assumptiongaussian}\label{sec:completmentcalcul}\label{sec:computGaussian}

\subsubsection{Checking assumptions in Gaussian NSBM}

In this section, we consider the Gaussian model \eqref{gaussmodel} with parameter set \eqref{sec:thetagauss} (here, we do not restrict the parameter set to be either $\Theta_{\sigma_0}$ or $\Theta_{\sigma^+_0}$). 
Recall in this case that the parameter is 
$\theta=(\pi,w,\sigma_0,\mu=(\mu_{q,\l})_{1\leq q\leq \l\leq Q},\sigma=(\sigma_{q,\l})_{1\leq q\leq \l\leq Q})$  and the set $\Theta$ is such that for each $q,\l$, either $\sigma_{q,\l}\neq  \sigma_0$ or $\mu_{q,\l}\neq 0$, which means that the distribution under the null is always different from the distribution under the alternative. We now provide a short summary of the results obtained throughout this section.

First, we can check that Assumptions~\ref{cont}~and~\ref{AssumptionIk} both hold, see Sections~\ref{sec:computlqvalgaussian} and~\ref{sec:studyq}. The complexity assumption holds with $K=2$. The regularity assumption holds with, for all $q,\l,$ $t_{1,q,\l}(\theta_0)=0$ if and only if $\sigma_{q,\l}\geq \sigma_0$ and $t_{2,q,\l}(\theta_0)=1$ if and only if $\sigma_{q,\l}\leq \sigma_0$.  An illustration is given in Figure~\ref{fig:QGaussian2}.
Note however that the function $t\mapsto \Q_{\theta_0}(\theta_0,t)$ might jump in $t_1(\theta_0)$ (see case 4), so is not necessarily continuous on $[0,1]$. Also, this function might have infinite derivative at the boundary points $\{t_{1,q,\l}(\theta_0)\}_{q,\l}$, $\{t_{2,q,\l}(\theta_0)\}_{q,\l}$ (see cases 2-3-4).

Second, the critical level $\alpha_*(\theta_0)$ given by \eqref{def:alphastar} follows the following simple rule: $\alpha_*(\theta_0)=0$ if and only if $\theta_0=(\pi,w,\sigma_0,\mu,\sigma)$ is such that $\max_{q,\l}\sigma_{q,\l}\geq \sigma_0$, that is, there exists a variance under the alternative that is at least equal to the variance under the null, see Lemma~\ref{lem:alphastargauss} in Section~\ref{sec:computalphastargaussian}.
For instance, in the context of Figures~\ref{fig:QGaussian}~and~\ref{fig:QGaussian2}, $\alpha_*=0$ in cases $1$-$2$-$3$ while $\alpha_*>0$ in case $4$.

Third, in the Gaussian case, we can say more about the order of $\mathcal{W}_{\theta_0,\q}(u)$ and $\mathcal{W}_{ T,\q_1}(u)$ as $u\to 0$. 
For this, we should however avoid the non regular behavior of $t\mapsto \Q_{\theta_0}(\theta_0,t)$ occurring at the boundary point $\{t_{1,q,\l}(\theta_0)\}_{q,\l}\cup \{t_{2,q,\l}(\theta_0)\}_{q,\l}$. 
This is possible by considering a compact set $\mathcal{K}$ containing $\alpha$ and such that $\mathcal{K}\subset (\alpha_*,\alpha^*)$ where $\alpha^*=\alpha^*(\theta_0)$ is given by
\begin{align}
\alpha^*(\theta_0)&= \Q_{\theta_0}(\theta_0,t'_2(\theta_0)) \in (\alpha_*(\theta),\pizero]\label{def:alphastar2}\\
t'_2(\theta_0)&= \min\{t_{1,q,\l}(\theta_0),t_{2,q,\l}(\theta_0),  \mbox{ for $q,\l$ s.t. }t_{1,q,\l}(\theta_0)\neq t_1(\theta_0) \}\in (t_1(\theta_0),1].\label{def:tprime2}
\end{align}
For instance, in Figures~\ref{fig:QGaussian}~and~\ref{fig:QGaussian2}, $t'_2(\theta_0)$ is equal to $1$, $\approx 0.62$, $\approx 0.27$ and $\approx 0.31$ in case 1-2-3-4, respectively. 
As a matter of fact, $t'_2(\theta_0)$ is often fairly away from zero. For instance, we establish in Section~\ref{sec:studytprime2}, that $t'_2(\theta_0)\geq 1/2$ when $\theta_0=(\pi,w,\sigma_0,\mu,\sigma)$ is such that $w_{q,\l}\leq 1/2$ and $\sigma_{q,\l}\geq \sigma_0$ for all $q,\l$.

Choosing a compact $\mathcal{K}$ as above, we are able to state that there exists $C(\theta_0,\alpha,Q)$ with for $u$ small enough, $\mathcal{W}_{T,\q_1}(u)\leq u C(\theta_0,\alpha,Q)$, because $\q_1$ is differentiable in $T_{\theta_0}(\alpha)$, see Section~\ref{sec:studyq}.

Dealing with $\mathcal{W}_{\theta_0,\q}(u)$ needs an additional assumption:
\begin{equation}\label{equ:forMgaussian}
\mbox{  $\mathcal{C}(\theta)=\{(q,\l)\in  \{1,\dots,Q\}^2\::\: \sigma_{q,\l}=\sigma_0\}$ does not depend on $\theta=(\pi,w,\sigma_0,\mu,\sigma)\in\Theta$.}
\end{equation}
Equivalently, \eqref{equ:forMgaussian} means that there exists some $\mathcal{C}\subset \{1,\dots,Q\}^2$ such that 
$$
\Theta\subset  \{\theta=(\pi,w,\sigma_0,\mu,\sigma)\in \Theta\::\:  \forall (q,\l)\in \mathcal{C}, \sigma_{q,\l}=\sigma_0,\: \forall (q,\l)\notin \mathcal{C}, \sigma_{q,\l}\neq \sigma_0 \}.
$$
In Section~\ref{sec:computMgaussian}, we show that under \eqref{equ:forMgaussian} there exists $C(\theta_0,\alpha,Q)>0$ with for $u$ small enough, $\mathcal{W}_{\theta_0,\q}(u)\leq u C(\theta_0,\alpha,Q)$.

\begin{figure}[h!]
\begin{center}
\begin{tabular}{cc}
\vspace{-0.5cm}
Case $3$ &Case $4$\\
\vspace{-0.5cm}
\includegraphics[scale=0.4]{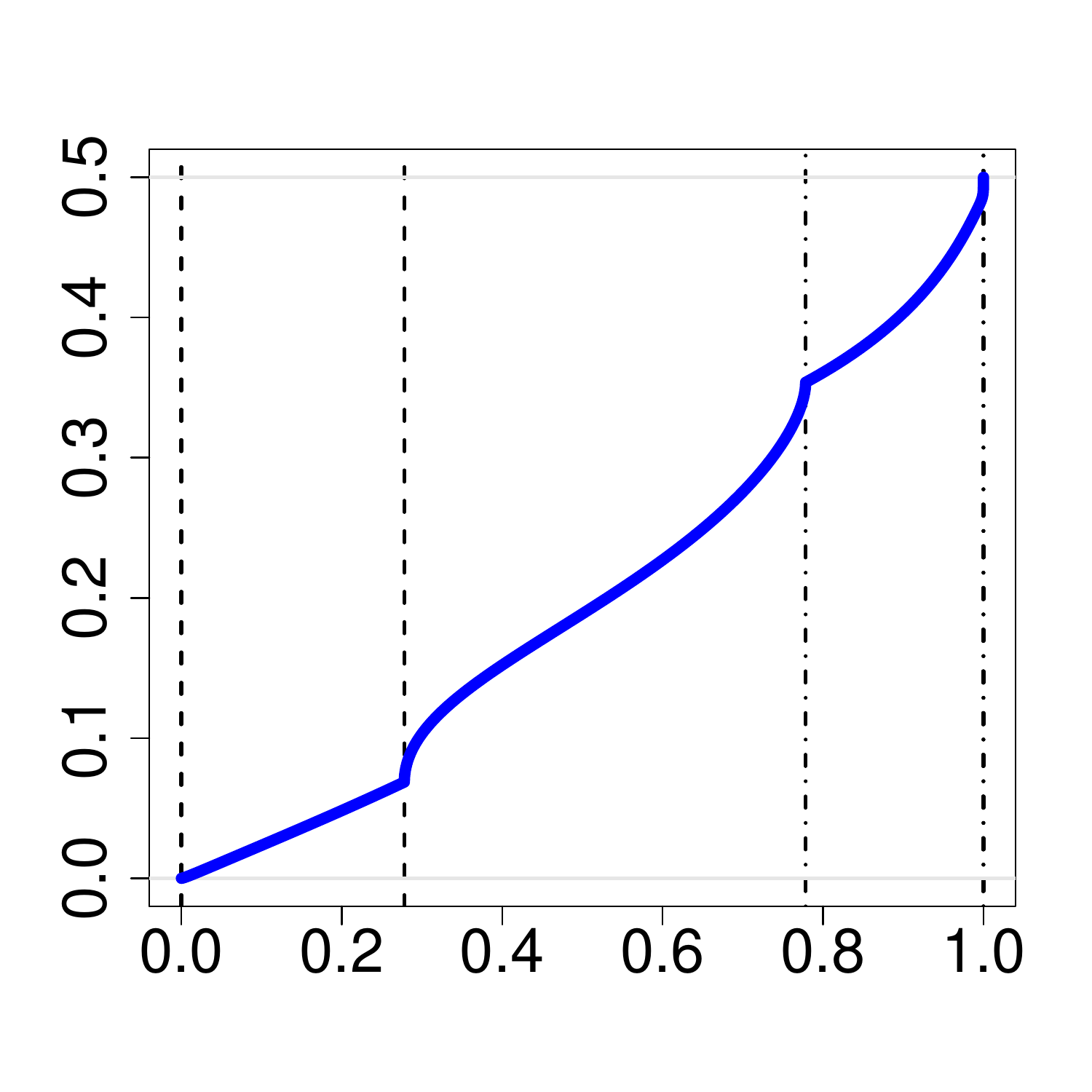}&\includegraphics[scale=0.4]{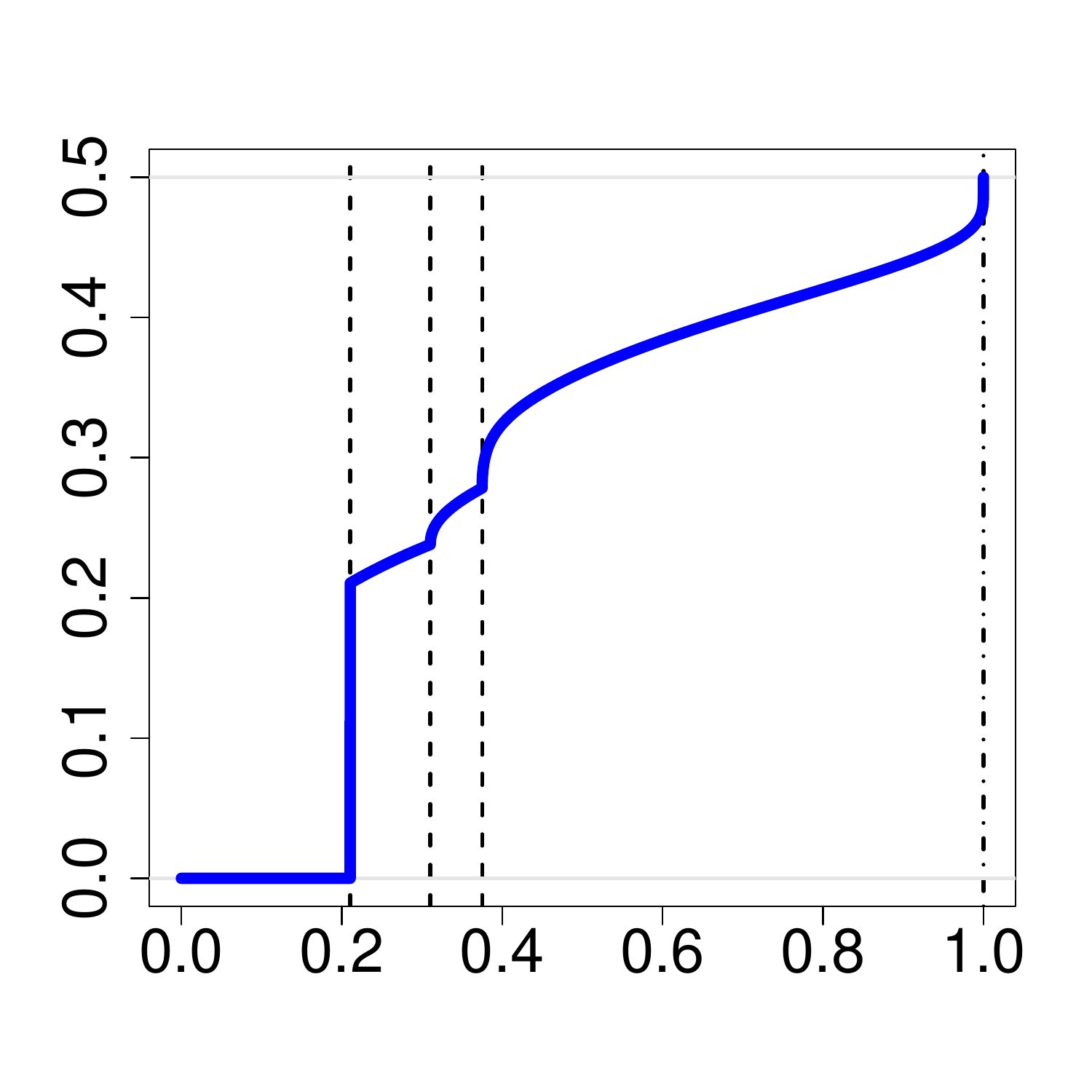}
\end{tabular}
\end{center}
\caption{Plot of $t\mapsto \Q_{\theta}(\theta,t)$ defined by \eqref{equ:Qcond} in the Gaussian case, for $2$ different values of the parameter $\theta$. In each case, the vertical dashed (resp. dashed-dotted) lines correspond to $\{t_{1,q,\l}\}_{q,\l}$ (resp.  $\{t_{2,q,\l}\}_{q,\l}$). In all cases we have $\pi_{q,\l}=0.5$ for all $q,\l$. $Q=2$. $w=(0.4,0.5,0.5,0.6)$. $\sigma_0=(1,1,1,1)$.
For case $3$: $\mu=(1,-2,-2,4)$, $\sigma=(0.5,1.1,1.1,3)$. For case $4$: $\mu=(0,0,0,0)$, $\sigma=(0.3,0.9,0.9,0.4)$. 
 \label{fig:QGaussian2}}
\end{figure}

\subsubsection{Computing $\l$-value and $q$-value functionals}\label{sec:computlqvalgaussian}

The $\l$-value functional \eqref{formulalvalfunction} is clearly given by
\begin{align}
\bell(x,q,\l;\theta)&= 
 \frac{ (1-w_{q,\l}) \phi(x/\sigma_0)/\sigma_0 }{(1-w_{q,\l}) \phi(x/\sigma_0)/\sigma_0 +w_{q,\l} \phi((x-\mu_{q,\l})/\sigma_{q,\l})/\sigma_{q,\l}}.\label{equ:lfonctiongauss}
 \end{align}
 Now,  let us fix $q,\l  \in \{1,\dots,Q\}$, $\theta',\theta\in\Theta$, $\delta\in\{0,1\}$ and let us compute the $q$-value functional
 \begin{align*}
\q_\delta(t,q,\l; \theta',\theta)=\P_{\theta'}( \bell(X_{i,j},q,\l;\theta)\leq t\:|\:
 Z_i=q,Z_j=\l, A_{i,j}=\delta) .
 \end{align*}
First, for $t=0$, $\q_\delta(t,q,\l; \theta',\theta)=0$.
Second, for all $t\in(0,1]$, if
$\theta=(\pi,w,\sigma_0,\mu,\sigma)$, 
 observe that 
\begin{align*}
\{x\in\R\::\:\bell(x,q,\l;\theta)\leq t\}
=&\left\{x\in\R\::\:
 \frac{ (1-w_{q,\l}) \phi(x/\sigma_0)/\sigma_0 }{(1-w_{q,\l}) \phi(x/\sigma_0)/\sigma_0 +w_{q,\l} \phi((x-\mu_{q,\l})/\sigma_{q,\l})/\sigma_{q,\l}}\leq t\right\}\\
&=
\left\{x\in\R\::\:
 \frac{\phi((x-\mu_{q,\l})/\sigma_{q,\l})}{\phi(x/\sigma_0)} \geq  (\sigma_{q,\l}/\sigma_0)(1/w_{q,\l}-1)(1/t-1)\right\}.
\end{align*}
Now, since 
$-2\log\left(\frac{\phi((x-\mu_{q,\l})/\sigma_{q,\l})}{\phi(x/\sigma_0)}\right) = [\sigma_{q,\l}^{-2}-\sigma_0^{-2}] x^2 - 2\mu_{q,\l}\sigma_{q,\l}^{-2} x + \mu_{q,\l}^2 \sigma_{q,\l}^{-2} ,
$
we have
\begin{align}
 \{ \bell(X_{i,j},q,\l;\theta) \leq t\} = \left\{ aX_{i,j}^2+bX_{i,j}+c\leq  0 \right\},
\label{equunioninterval}
\end{align}
for the values $a=a(q,\l,\theta)$, $b=b(q,\l,\theta)$, $c=c(q,\l,\theta)$ given by
\begin{align}\label{equabc}
\left\{\begin{array}{ll}
a &= \sigma_{q,\l}^{-2}-\sigma_0^{-2};\\
b &=  -2\mu_{q,\l}\sigma_{q,\l}^{-2}; \\
c &=\mu_{q,\l}^2 \sigma_{q,\l}^{-2} +2\log\left((\sigma_{q,\l}/\sigma_0)(1/w_{q,\l}-1)(1/t-1)\right).
\end{array}\right.
\end{align}

 As a result, we have
\begin{align}
 \q_\delta(t,q,\l; \theta',\theta)&=\P_{U\sim \mathcal{N}(\mu(\delta),\sigma(\delta)^2)}(aU^2+bU+c\leq 0) ,\label{equ:qfonctiongauss}
 \end{align}
for  $\mu(\delta)=\mu(\delta,q,\l,\theta')$, $\sigma(\delta)=\sigma(\delta,q,\l,\theta')$ given by
\begin{align}\label{equmusigma}
\mbox{$\mu(0)=0, \sigma(0)=\sigma_0'$, and $\mu(1)=\mu'_{q,\l}, \sigma(1)=\sigma'_{q,\l}$.}
\end{align}

In addition, expression \eqref{equ:qfonctiongauss} can be made explicit by an elementary  inversion of $aU^2+bU+c<0$ in $U$.
More precisely, denoting $\Phi_{\delta}$  the cumulative distribution function of the distribution $\mathcal{N}(\mu(\delta),\sigma(\delta)^2)$, we obtain

\begin{itemize}
\item[(i)] if $a<0$ (that is, $\sigma_{q,\l}>\sigma_0$),
\begin{align*}
&\q_\delta(t,q,\l; \theta',\theta)=\mathds{1}\{b^2<4ac\}  \\
&+ \left[ \Phi_\delta\left(\frac{b-(b^2-4ac)^{1/2}_+}{2|a|}\right)  + 1-\Phi_\delta\left( \frac{b+(b^2-4ac)^{1/2}_+}{2|a|}\right) \right]\mathds{1}\{b^2\geq 4ac\};
\end{align*}
\item[(ii)] if $a>0$ (that is, $\sigma_{q,\l}<\sigma_0$),
\begin{equation*}
\q_\delta(t,q,\l; \theta',\theta)=  \left[\Phi_\delta\left(  \frac{-b+(b^2-4ac)^{1/2}_+}{2a} \right) -\Phi_\delta\left(  \frac{-b-(b^2-4ac)^{1/2}_+}{2a}\right)\right]\mathds{1}\{b^2> 4ac\};
\end{equation*}
\item[(iii)] if $a=0$  and $b\neq 0$ (that is, $\sigma_{q,\l}=\sigma_0$ and $\mu_{q,\l}\neq 0$),
\begin{equation*}
\q_\delta(t,q,\l; \theta',\theta)=
 (1-\Phi_\delta(-c/b))  \mathds{1}\{b< 0\} + \Phi_\delta(-c/b)  \mathds{1}\{b> 0\}.
\end{equation*}
\end{itemize}

Note that \eqref{equunioninterval} shows that Assumption~\ref{AssumptionIk} holds with $K=2$.

\subsubsection{Study of the $q$-value functional}\label{sec:studyq}

By Lemma~\ref{lem:proba}, and expression \eqref{equ:qfonctiongauss}, we get that, for all $q,\l\in\{1,\dots,Q\}$, 
 the function 
$$(t,\theta,\theta')\in [0,1]\times \Theta^2 \mapsto \q_\delta(t,q,\l; \theta',\theta)$$ is continuous on $[0,1]\times \Theta^2$.
 
By the explicit expressions of the previous section, we have a more accurate idea of the behavior of the continuous function  $t\in [0,1] \mapsto \q_\delta(t,q,\l; \theta',\theta)$, for fixed values of $q,\l\in\{1,\dots,Q\}$, $\theta,\theta' \in\Theta$ and any $\delta\in\{0,1\}$. Let us define $t_1=t_{1,q,\l}(\theta)$ and $t_2=t_{2,q,\l}(\theta)$ as follows:
 \begin{align}\label{deft1t2}
(t_1, t_2) = 
\left\{\begin{array}{cl}
(0,t_0) &\mbox{ if $\sigma_{q,\l}>\sigma_0$};\\
(t_0,1) &\mbox{ if $\sigma_{q,\l}<\sigma_0$};\\
(0,1) &\mbox{ if $\sigma_{q,\l}=\sigma_0$};
\end{array}
\right.
\:\: t_0 = \left(1+\frac{w_{q,\l}}{1-w_{q,\l}} \frac{\sigma_0}{\sigma_{q,\l}} \exp\left(\frac{\mu_{q,\l}^2}{2(\sigma_0^2-\sigma_{q,\l}^2)}\right)\right)^{-1}. 
\end{align}
Then we have the following result.
\begin{lemma}\label{lem:forassumptioncont}
The function
$t\in [0,1] \mapsto \q_\delta(t,q,\l; \theta',\theta)$ is constant equal to $0$ on $[0,t_1]$,  
continuous increasing on $[t_1,t_2]$ from $t=t_1$ (value $0$) to $t=t_2$ (value $1$) and then is constant equal to $1$ on $[t_0,1]$. It is also infinitely differentiable on $(t_1,t_2)$, but not differentiable in $t_1$ when $t_1> 0$ and in $t_2$ when $t_2< 1$.
\end{lemma}

Indeed, if $a=0$ the result is obvious. If $a\neq 0$, $t_0$ is the only value of $t$ such that $b^2=4 a c$; 
 if $a<0$, the quantity $ac$, as a function of $t$, is continuous increasing with limits $-\infty$ and $+\infty$ in $t=0^+$ and $t=1^-$.
If $a>0$, the quantity $ac$, as a function of $t$, is decreasing with limits $+\infty$ and $-\infty$ in $t=0^+$ and $t=1^-$, respectively. 
In addition, 
 the derivative in $t \notin\{t_1,t_2\}$ is equal to, when $a\neq 0$,  
$$
\frac{2}{t(1-t)(b^2-4ac)^{1/2}} \left(\phi_{\delta}\left(\frac{-b+(b^2-4ac)^{1/2}_+}{2a}\right)  +\phi_{\delta}\left(\frac{-b-(b^2-4ac)^{1/2}_+}{2 a}\right)\right)\mathds{1}\{b^2> 4ac\} ,
$$
and when $a=0$ (and thus $b\neq 0$),
$$
\frac{2}{t(1-t)|b|} \phi_{\delta}\left(\frac{-c}{b}\right)  ,
$$
where $\phi_{\delta}$ denotes  the density of the distribution $\mathcal{N}(\mu(\delta),\sigma(\delta)^2)$.
This entails Lemma~\ref{lem:forassumptioncont}.

In particular, the results of this section imply that Assumption~\ref{cont} holds
and that $t\mapsto \q_1(t,q,\l; \theta_0,\theta_0)$ is differentiable in $t=T_{\theta_0}(\alpha)$ when $T_{\theta_0}(\alpha)\notin\{t_1,t_2\}$.

\subsubsection{Studying $\alpha_*$}\label{sec:computalphastargaussian}

Let $\theta=(\pi,w,\sigma_0,\mu,\sigma)\in\Theta$. Recall $$\alpha_*(\theta)=\lim_{t\to t_1(\theta)^+}\{\Q_{\theta}(\theta,t)\} = \lim_{t\to t_1(\theta)^+}\left\{\frac{1 }{1+ \frac{\sum_{q,\l} \pi_q\pi_\l w_{q,\l}  \q_1(t,q,\l;\theta,\theta)}{  \sum_{q,\l} \pi_q\pi_\l (1-w_{q,\l})  \q_0(t,q,\l;\theta,\theta) }} \right\}.$$
We prove in this section the following result.
\begin{lemma}\label{lem:alphastargauss}
For all $\theta=(\pi,w,\sigma_0,\mu,\sigma)\in\Theta$, we have $\alpha_*(\theta)=0$ if and only if  there exists $q,\l\in\{1,\dots,Q\}$, such that $\sigma_{q,\l}\geq \sigma_0$.
\end{lemma}

To prove this, first note that  for all $t\in[0,1]$, we have 
\begin{align*}
 \min_{q,\l}\left\{\frac{w_{q,\l}}{1-w_{q,\l}} \frac{  \q_1(t,q,\l;\theta,\theta)}{  \q_0(t,q,\l;\theta,\theta) }\right\}&\leq  \frac{\sum_{q,\l} \pi_q\pi_\l w_{q,\l}  \q_1(t,q,\l;\theta,\theta)}{  \sum_{q,\l} \pi_q\pi_\l (1-w_{q,\l})  \q_0(t,q,\l;\theta,\theta) }\\
 &\leq \max_{q,\l}\left\{\frac{w_{q,\l}}{1-w_{q,\l}} \frac{  \q_1(t,q,\l;\theta,\theta)}{  \q_0(t,q,\l;\theta,\theta) }\right\}.
\end{align*}
 We distinguish among the three following cases:
\begin{itemize}
\item if  $\theta$ is such that for all $q,\l$, we have $\sigma_{q,\l}< \sigma_0$. Then $t_1(\theta)=\min_{q,\l} t_{1,q,\l}(\theta)  >0$. In that case, 
$$\alpha_*(\theta)= \frac{1 }{1+ \frac{\sum_{q,\l} \pi_q\pi_\l w_{q,\l}  \q_1(t_1(\theta)^+,q,\l;\theta,\theta)}{  \sum_{q,\l} \pi_q\pi_\l (1-w_{q,\l})  \q_0(t_1(\theta)^+,q,\l;\theta,\theta) }}>0 ,$$
because in the sums, $\q_0(t_1(\theta)^+,q,\l;\theta,\theta)$ and $\q_1(t_1(\theta)^+,q,\l;\theta,\theta)$ are non-zero for $q,\l$ such that $t_{1,q,\l}(\theta)=t_1(\theta)$ (and are zero otherwise).

\item if $\theta$ is such that for all $q,\l$, we have $\sigma_{q,\l} \leq \sigma_0$ and there exists $q,\l$ such that $\sigma_{q,\l} = \sigma_0$. Then  $t_1(\theta)=\min_{q,\l} t_{1,q,\l}(\theta)=0$. 
Also, for $t$ close  enough to $0$, we have 
$$
\frac{\sum_{q,\l} \pi_q\pi_\l w_{q,\l}  \q_1(t,q,\l;\theta,\theta)}{  \sum_{q,\l} \pi_q\pi_\l (1-w_{q,\l})  \q_0(t,q,\l;\theta,\theta) } = 
\frac{\sum_{q,\l : \sigma_{q,\l} = \sigma_0} \pi_q\pi_\l w_{q,\l}  \q_1(t,q,\l;\theta,\theta)}{  \sum_{q,\l: \sigma_{q,\l} = \sigma_0} \pi_q\pi_\l (1-w_{q,\l})  \q_0(t,q,\l;\theta,\theta) } ,
$$
because by Section~\ref{sec:studyq}, $\q_\delta(t,q,\l;\theta,\theta)=0$ for $\sigma_{q,\l} < \sigma_0$ when $t$ is close enough to $0$.
Next, by the computations of Section~\ref{sec:computlqvalgaussian}, for all $\delta\in\{0,1\}$ (see case $a=0$ therein), for all $q,\l$, with $\sigma_{q,\l}=\sigma_0$, we have
\begin{align*}
\q_\delta(t,q,\l; \theta,\theta)&=
 \left(1-\Phi_\delta\left(\frac{\mu_{q,\l}^2 \sigma_{q,\l}^{-2} +2\log\left((\sigma_{q,\l}/\sigma_0)(1/w_{q,\l}-1)(1/t-1)\right)}{2\mu_{q,\l}\sigma_{q,\l}^{-2}}\right)\right)\mathds{1}\{\mu_{q,\l}>0\}\\
 &+
 \Phi_\delta\left(\frac{\mu_{q,\l}^2 \sigma_{q,\l}^{-2} +2\log\left((\sigma_{q,\l}/\sigma_0)(1/w_{q,\l}-1)(1/t-1)\right)}{2\mu_{q,\l}\sigma_{q,\l}^{-2}}\right)\mathds{1}\{\mu_{q,\l}<0\}\\
 &=1-\Phi_\delta\left(\frac{\mu_{q,\l}}{2}  +\sigma_0^{2} \frac{\log\left((1/w_{q,\l}-1)(1/t-1)\right)}{\mu_{q,\l}}\right)\mathds{1}\{\mu_{q,\l}>0\}\\
 &+
 \Phi_\delta\left(\frac{\mu_{q,\l}}{2}  +\sigma_0^{2} \frac{\log\left((1/w_{q,\l}-1)(1/t-1)\right)}{\mu_{q,\l}}\right)\mathds{1}\{\mu_{q,\l}<0\}
 .
\end{align*}
Recall that $1-\Phi(x)\sim \phi(x)/x$ when $x\to \infty$. Hence, for all $q,\l$, when $t\to 0^+$,
\begin{align*}
\frac{\q_1(t,q,\l; \theta,\theta)}{\q_0(t,q,\l; \theta,\theta)}&\sim 
 \frac{\frac{|\mu_{q,\l}|}{2\sigma_0}  +\sigma_0 \frac{\log\left((1/w_{q,\l}-1)(1/t-1)\right)}{|\mu_{q,\l}|}}
 {\frac{-|\mu_{q,\l}|}{2\sigma_0}  +\sigma_0 \frac{\log\left((1/w_{q,\l}-1)(1/t-1)\right)}{|\mu_{q,\l}|}}
 \frac{\phi\left( \frac{-|\mu_{q,\l}|}{2\sigma_0}  +\sigma_0 \frac{\log\left((1/w_{q,\l}-1)(1/t-1)\right)}{|\mu_{q,\l}|}\right)}{\phi\left(\frac{|\mu_{q,\l}|}{2\sigma_0}  +\sigma_0 \frac{\log\left((1/w_{q,\l}-1)(1/t-1)\right)}{|\mu_{q,\l}|}\right)} \\
 &\sim \exp\left\{\log\left((1/w_{q,\l}-1)(1/t-1)\right)\right\} = \frac{1-w_{q,\l}}{w_{q,\l} } (1/t-1).
\end{align*}
because $\phi(x-y)/\phi(x+y)=e^{2xy}$ for all $x,y\in\R$. Hence, in that case, when $t\to 0^+$,
$$
\frac{\sum_{q,\l} \pi_q\pi_\l w_{q,\l}  \q_1(t,q,\l;\theta,\theta)}{  \sum_{q,\l} \pi_q\pi_\l (1-w_{q,\l})  \q_0(t,q,\l;\theta,\theta) } \sim 1/t.
$$
and $\alpha_*(\theta)=0$.

\item if $\theta$ is such that there exist $q,\l$ with $\sigma_{q,\l}> \sigma_0$, then $t_1(\theta)=\min_{q,\l} t_{1,q,\l}(\theta)=0$. 
Also, for $t$ close  enough to $0$, we have 
$$
\frac{\sum_{q,\l} \pi_q\pi_\l w_{q,\l}  \q_1(t,q,\l;\theta,\theta)}{  \sum_{q,\l} \pi_q\pi_\l (1-w_{q,\l})  \q_0(t,q,\l;\theta,\theta) } = 
\frac{\sum_{q,\l : \sigma_{q,\l} \geq \sigma_0} \pi_q\pi_\l w_{q,\l}  \q_1(t,q,\l;\theta,\theta)}{  \sum_{q,\l: \sigma_{q,\l} \geq \sigma_0} \pi_q\pi_\l (1-w_{q,\l})  \q_0(t,q,\l;\theta,\theta) } ,
$$
because by Section~\ref{sec:studyq}, $\q_\delta(t,q,\l;\theta,\theta)=0$ for $\sigma_{q,\l} < \sigma_0$ when $t$ is close enough to $0$.
Hence, we have for $t$ close  enough to $0$,
\begin{align*}
 \min_{q,\l: \sigma_{q,\l} \geq \sigma_0}\left\{\frac{w_{q,\l}}{1-w_{q,\l}} \frac{  \q_1(t,q,\l;\theta,\theta)}{  \q_0(t,q,\l;\theta,\theta) }\right\}&\leq  \frac{\sum_{q,\l} \pi_q\pi_\l w_{q,\l}  \q_1(t,q,\l;\theta,\theta)}{  \sum_{q,\l} \pi_q\pi_\l (1-w_{q,\l})  \q_0(t,q,\l;\theta,\theta) }\\
 &\leq \max_{q,\l: \sigma_{q,\l} \geq \sigma_0}\left\{\frac{w_{q,\l}}{1-w_{q,\l}} \frac{  \q_1(t,q,\l;\theta,\theta)}{  \q_0(t,q,\l;\theta,\theta) }\right\}.
\end{align*}
We know by above that when $t\to0^+$,
$$
 \min_{q,\l: \sigma_{q,\l} = \sigma_0}\left\{\frac{w_{q,\l}}{1-w_{q,\l}} \frac{  \q_1(t,q,\l;\theta,\theta)}{  \q_0(t,q,\l;\theta,\theta) }\right\}\sim \max_{q,\l: \sigma_{q,\l} = \sigma_0}\left\{\frac{w_{q,\l}}{1-w_{q,\l}} \frac{  \q_1(t,q,\l;\theta,\theta)}{  \q_0(t,q,\l;\theta,\theta) }\right\}\sim 1/t.
$$
Now, for $q,\l$ such that $\sigma_{q,\l} > \sigma_0$, we have for $t$ small enough ($a<0$, $c>0$)
\begin{align*}
&\q_\delta(t,q,\l; \theta,\theta)\\
&=  \Phi_{\delta}\left(\frac{b-(b^2-4ac)^{1/2}_+}{2|a|}\right)  + 1-\Phi_{\delta}\left( \frac{b+(b^2-4ac)^{1/2}_+}{2|a|}\right) \\
&=  1-\Phi\left(\frac{-b+2|a|\mu(\delta)+(b^2-4ac)^{1/2}_+}{2|a|\sigma(\delta)}\right)  + 1-\Phi\left( \frac{b-2|a|\mu(\delta)+(b^2-4ac)^{1/2}_+}{2|a|\sigma(\delta)}\right) \\
&=  1-\Phi\left(\frac{-|(b-2|a|\mu(\delta))|+(b^2-4ac)^{1/2}_+}{2|a|\sigma(\delta)}\right)  + 1-\Phi\left( \frac{|(b-2|a|\mu(\delta))|+(b^2-4ac)^{1/2}_+}{2|a|\sigma(\delta)}\right)
\end{align*}
Now use for all $z>0$, for $x\to \infty$,
$$
\frac{1-\Phi(-z+  x)}{ 1-\Phi(z+  x)} \sim  \frac{\phi(-z+  x)}{\phi(z+x)} = e^{2xz} \to \infty
$$
so that $1-\Phi(-z+  x)+ 1-\Phi(z+  x)\sim 1-\Phi(-z+  x)$.
As a result, since  $|(b-2|a|\mu(1))|= 2|\mu_{q,\l}| |\sigma_{q,\l}^{-2} + |\sigma_{q,\l}^{-2}-\sigma_0^{-2}| | = 2|\mu_{q,\l}| \sigma_0^{-2}$, when $t\to 0^+$,
\begin{align*}
\q_0(t,q,\l; \theta,\theta)&\sim 1-\Phi\left(\frac{-2|\mu_{q,\l}| \sigma_{q,\l}^{-2} +(b^2-4ac)^{1/2}_+}{2|a|\sigma_0}\right)\\
\q_1(t,q,\l; \theta,\theta)&\sim 1-\Phi\left(\frac{-2|\mu_{q,\l}| \sigma_0^{-2}+(b^2-4ac)^{1/2}_+}{2|a|\sigma_{q,\l}}\right)\\
&= 1-\Phi\left(
\frac{-|\mu_{q,\l}| }{\sigma_{q,\l}} +\frac{\sigma_0}{\sigma_{q,\l}} \frac{-2|\mu_{q,\l}| \sigma_{q,\l}^{-2}+(b^2-4ac)^{1/2}_+}{2|a|\sigma_0}
\right).
\end{align*}
Now use for all $z>0$, $u\in(0,1)$, for $x\to \infty$,
$$
\frac{1-\Phi(-z+ u x)}{1-\Phi(x)} \sim u^{-1} \frac{\phi(-z+ u x)}{\phi(x)} =
u e^{0.5 ( (1-u^2)x^2 + 2uz x -z^2)}\to \infty,
$$
to conclude that 
$$
 \min_{q,\l: \sigma_{q,\l} \geq \sigma_0}\left\{\frac{w_{q,\l}}{1-w_{q,\l}} \frac{  \q_1(t,q,\l;\theta,\theta)}{  \q_0(t,q,\l;\theta,\theta) }\right\},\:\:\: \max_{q,\l: \sigma_{q,\l} \geq \sigma_0}\left\{\frac{w_{q,\l}}{1-w_{q,\l}} \frac{  \q_1(t,q,\l;\theta,\theta)}{  \q_0(t,q,\l;\theta,\theta) }\right\}
$$
both tends to infinity when $t\to 0^+$. 
Hence, $\alpha^\star(\theta)=0$ in that case.

\end{itemize}

\subsubsection{Studying $t'_2(\theta_0)$}\label{sec:studytprime2}

In the case where $\theta_0=(\pi,w,\sigma_0,\mu,\sigma)$ is such that  $\forall q,\l$, $\sigma_{q,\l}\geq \sigma_0$, we have by Section~\ref{sec:studyq} that $t_{1,q,\l}(\theta_0)=0$ for all $q,\l$ and thus
\begin{align*}
t'_2(\theta_0)&= \min\{t_{1,q,\l}(\theta_0),t_{2,q,\l}(\theta_0),  \mbox{ for $q,\l$ s.t. }t_{1,q,\l}(\theta_0)\neq t_1(\theta_0) \}\\
&=\min_{q,\l}\{t_{2,q,\l}(\theta_0)\}
\end{align*}
If $\forall q,\l$, $\sigma_{q,\l}= \sigma_0$, the latter is simply $t'_2(\theta_0)=1$. Otherwise, we have
\begin{align*}
t'_2(\theta_0)
&= \min\left\{ \left(1+\frac{w_{q,\l}}{1-w_{q,\l}} \frac{\sigma_0}{\sigma_{q,\l}} \exp\left(\frac{\mu_{q,\l}^2}{2(\sigma_0^2-\sigma_{q,\l}^2)}\right)\right)^{-1},  \mbox{ for $q,\l$ s.t. } \sigma_{q,\l}> \sigma_0\right\}\\
&\geq \min_{q,\l}\left(1+\frac{w_{q,\l}}{1-w_{q,\l}}\right)^{-1} = 1-\max_{q,\l}\{ w_{q,\l}\} .
\end{align*}

\subsubsection{Bounding the continuity modulus of $q$-value fonctionals} \label{sec:computMgaussian}

Denote $\mathcal{C}$ the set \eqref{equ:forMgaussian}.
Recall that by \eqref{equmodulusq}, we have
\begin{align*} 
\mathcal{W}_{\theta_0,\q}(u)&= \sup_{q,\l} \sup_{t\in T_{\theta_0}(\mathcal{K})}\sup_{\delta\in\{0,1\}} \sup \left\{
\left|
\q_\delta(t,q,\l; \theta',\theta)-\q_\delta(t,q,\l; \theta_0,\theta_0) \right| \::\: \right. \nonumber\\
&\hspace{5cm}\left. \theta,\theta'\in \Theta, \|\theta-\theta_0\|_\infty\leq u ,\|\theta'-\theta_0\|_\infty\leq u \right\}.
\end{align*}
For short, denote $d(q,\l,\theta,t)=b^2(q,\l,\theta)-4a(q,\l,\theta)c(q,\l,\theta,t)$ for any $\theta\in \Theta$ (and $a$, $b$ and $c$ being the quantities defined by \eqref{equabc}), and consider
\begin{align*}
\mathcal{B}(\theta_0,\alpha) 
= \big\{\theta\in \Theta\::\: &\forall (q,\l)\notin \mathcal{C},\:  |a(q,\l,\theta)-a(q,\l,\theta_0)|\leq |a(q,\l,\theta_0)|/2, \\
& \:\:\:\:\forall t\in T_{\theta_0}(\mathcal{K}),\:   |d(q,\l,\theta,t)-d(q,\l,\theta_0,t)|\leq |d(q,\l,\theta_0,t)|/2, \\
&\forall (q,\l)\in \mathcal{C},\:  |b(q,\l,\theta)-b(q,\l,\theta_0)|\leq |b(q,\l,\theta_0)|/2\big\}.
\end{align*}
We check that there exists $e(\theta_0,\alpha,Q)$ such that for all $\eps\leq e(\theta_0,\alpha,Q)$, we have 
\begin{equation}\label{trucBtheta0alpha}
\{\theta\in \Theta\::\:\|\theta-\theta_0\|_\infty\leq \eps\}\subset \mathcal{B}(\theta_0,\alpha) .
\end{equation}
To see this, let $\theta\in \Theta$ with $\|\theta-\theta_0\|_\infty\leq \eps$. 
For $(q,\l)\notin \mathcal{C}$, we have $a(q,\l,\theta_0)\neq 0$ so that, since $\theta\in\Theta\mapsto a(q,\l,\theta)$ is continuous, for $\eps$ smaller than some positive number $e_1(q,\l,\theta_0,\alpha)$, we have $|a(q,\l,\theta)-a(q,\l,\theta_0)|\leq |a(q,\l,\theta_0)|/2$. In addition, by definition of $T_{\theta_0}(\mathcal{K})$, we have 
$t_{1,q,\l}(\theta_0),t_{2,q,\l}(\theta_0) \notin T_{\theta_0}(\mathcal{K})$, see \eqref{deft1t2}, and thus the sign of $d(q,\l,\theta,t)$ does not depend on $t\in T_{\theta_0}(\mathcal{K})$. Assume without loss of generality that it is positive, so that $\inf_{t\in T_{\theta_0}(\mathcal{K})} d(q,\l,\theta_0,t)>0$. Now, since $T_{\theta_0}(\mathcal{K})$ is a compact set and   $(\theta,t)\in\Theta\times T_{\theta_0}(\mathcal{K}) \mapsto d(q,\l,\theta,t)$ is continuous, we have that
$
\lim_{\theta\to\theta_0} \sup_{t\in T_{\theta_0}(\mathcal{K})}|d(q,\l,\theta,t)-d(q,\l,\theta_0,t)|=0
$
(otherwise, there exists $c>0,\theta_n\to \theta_0$ and $t_n$ such that for $n$ large, $|d(q,\l,\theta_n,t_n)-d(q,\l,\theta_0,t_n)|>c$ and we obtain a contradiction by considering any limit $t_0$ of $t_n$). As a result, there is $e_2(q,\l,\theta_0,\alpha)$ such that for $\eps\leq e_2(q,\l,\theta_0,\alpha)$,
$$
\sup_{t\in T_{\theta_0}(\mathcal{K})}|d(q,\l,\theta,t)-d(q,\l,\theta_0,t)|\leq \inf_{t\in T_{\theta_0}(\mathcal{K})} d(q,\l,\theta_0,t)/2,
$$
that is, $\forall t\in T_{\theta_0}(\mathcal{K}),\:   |d(q,\l,\theta,t)-d(q,\l,\theta_0,t)|\leq d(q,\l,\theta_0,t)/2$. Finally, consider 
$(q,\l)\in \mathcal{C}$. In that case, $a(q,\l,\theta_0)=0$ and thus $b(q,\l,\theta_0)\neq 0$. Since   $\theta\in \Theta\mapsto b(q,\l,\theta)$ is continuous, for $\eps$ smaller than some positive number $e_3(q,\l,\theta_0,\alpha)$, we have $|b(q,\l,\theta)-b(q,\l,\theta_0)|\leq |b(q,\l,\theta_0)|/2$.
Summing up, we obtain \eqref{trucBtheta0alpha} for $\eps\leq e(\theta_0,\alpha,Q)=\min_{(q,\l)\notin \mathcal{C}}\{e_1(q,\l,\theta_0,\alpha)\wedge e_2(q,\l,\theta_0,\alpha)\}\: \wedge \:\min_{(q,\l)\in \mathcal{C}}e_3(q,\l,\theta_0,\alpha)$.

We have for all $\delta\in\{0,1\}$,  $\theta,\theta'\in\Theta$ with $\|\theta-\theta_0\|_\infty\leq \eps$ and $\|\theta'-\theta_0\|_\infty\leq \eps$, for all $q,\l\in\{1,\dots,Q\}$, for all $t\in T_{\theta_0}(\mathcal{K})$, when $\eps\leq e(\theta_0,\alpha,Q)$,
\begin{align*}
\frac{\left| \q_\delta(t,q,\l; \theta',\theta) -\q_\delta(t,q,\l; \theta_0,\theta_0)   \right|}{\|\theta-\theta_0\|_\infty\vee \|\theta'-\theta_0\|_\infty}\leq Q^2 \:\max_{q,\l} \sup_{t\in T_{\theta_0}(\mathcal{K})}  \sup_{(\theta',\theta) \in \mathcal{B}(\theta_0,\alpha)^2}\|\nabla_{(\theta',\theta)} \q_\delta(t,q,\l; \cdot,\cdot)\|_\infty,
\end{align*}
where $\nabla_{(\theta',\theta)} \q_\delta(t,q,\l; \cdot,\cdot)$ denotes the gradient of the function $(\theta',\theta)\in \mathcal{B}(\theta_0,\alpha)^2\mapsto \q_\delta(t,q,\l; \theta',\theta)$.
Now, thanks to the definition of $\mathcal{B}(\theta_0,\alpha)$, the formulas given in (i)-(ii)-(iii) of Section~\ref{sec:computlqvalgaussian} are active, without indicator, which means that $(t,\theta',\theta)\in T_{\theta_0}(\mathcal{K})\times\mathcal{B}(\theta_0,\alpha)^2\mapsto \nabla_{(\theta',\theta)} \q_\delta(t,q,\l; \cdot,\cdot)$ is continuous and thus for all $\delta\in\{0,1\}$ and $\eps\leq e(\theta_0,\alpha,Q)$, the quantities
$$
\sup_{t\in T_{\theta_0}(\mathcal{K})}  \sup_{(\theta',\theta) \in \mathcal{B}(\theta_0,\alpha)^2}\|\nabla_{(\theta',\theta)} \q_\delta(t,q,\l; \cdot,\cdot)\|_\infty
$$
for $(q,\l)\in\{1,\dots,Q\}^2$, 
are below some constant that depends only $\theta_0$, $\alpha$ and $Q$.

\subsubsection{A useful lemma}
 
  \begin{lemma}\label{lem:proba}
Let $\mathcal{D}
=\{(a,b,c,\mu,\sigma^2)\in\R^2\times \R\cup\{+\infty\}\times \R\times (0,\infty) \::\: a^2+b^2\neq 0\}$. 
Then the function
\begin{align}
(a,b,c,\mu,\sigma^2)\in \mathcal{D} \mapsto \P_{U\sim\mathcal{N}(\mu,\sigma^2)}(aU^2+bU+c<0).\label{Psimusigma}
\end{align}
 is continuous on $\mathcal{D}$.
 \end{lemma}

\begin{proof}
Consider $(a_0,b_0)\in\R^2\backslash\{0\}$, $c_0\in\R$, $\mu_0\in\R$, $\sigma_0>0$ and some $(a_n,b_n,c_n,\mu_n,\sigma_n)$ with $(a_n,b_n,c_n,\mu_n,\sigma_n)\to (a_0,b_0,c_0,\mu_0,\sigma_0)$ as $n$ tends to infinity. Consider $V\sim \mathcal{N}(0,1)$, $U_n=\mu_n+\sigma_n V$ and $U=\mu_0+\sigma_0 V$ so that $U_n\sim \mathcal{N}(\mu_n,\sigma_n^2)$ and $U\sim \mathcal{N}(\mu_0,\sigma_0^2)$. Then 
$a_n U_n^2+b_n U_n + c_n$ converges  to $a_0 U^2+b_0 U+c_0$ almost surely and thus also in distribution.
 Since the distribution of $a_0 U^2+b_0 U+c_0$ is continuous (because $a_0$ and $b_0$ are not both zero), we have by the Portmanteau Lemma that $\P(a_n U_n^2+b_n U_n+c_n<0)$ converges to $\P(a_0 U^2+b_0 U+c_0<0)$. Finally, a similar reasoning can be applied when $c_0=+\infty$, with a limit equal to $0$. 
The continuity follows.
\end{proof}

\end{document}